\documentclass[smallextended,envcountsect]{svjour3}
\usepackage{pifont}
\usepackage{epsfig,color,graphicx,subfigure,multirow,lscape}
\usepackage[colorlinks,linkcolor=blue,citecolor=blue]{hyperref}
\usepackage{tikz,cite}
\usepackage{graphicx,dsfont}
\usepackage{amsmath,amssymb}
\usepackage{appendix}
\usetikzlibrary{arrows, automata}
\smartqed

\journalname{}
\newtheorem{thm}{Theorem}[section]
\newtheorem{rem}{Remark}[section]
\newtheorem{lem}{Lemma}[section]
\newtheorem{fact}{Fact}[section]
\newtheorem{algo}{Algorithm}[section]

\newtheorem{asmp}{Assumption}[section]

\def\be{\begin{eqnarray}}
\def\ee{\end{eqnarray}}
\def\ben{\begin{eqnarray*}}
\def\een{\end{eqnarray*}}
\def\ba{\begin{array}}
\def\ea{\end{array}}
\def\bi{\begin{itemize}}
\def\ei{\end{itemize}}

\def\cX{{\mathcal X}}
\def\cY{{\mathcal Y}}
\def\cO{{\mathcal O}}
\def\bR{{\mathbb R}}

\def\cL{{\mathcal L}}

\def\cG{{\mathcal G}}
\def\prox{{\rm Prox}}

\def\R{{\mathbb R}}
\def\dom{{\mathbf{dom}}}
\def\interior{{\mathbf{int}}}

\def\dist{{\rm dist}}

\newcommand{\la}{\lambda}

\begin{document}

\title{A convex combination based primal-dual algorithm with linesearch for general convex-concave saddle point problems}
\titlerunning{PDA with convex combination for nonlinear saddle point problems}
\author{Xiaokai Chang$^{1}$
\and  Junfeng Yang$^2$
\and Hongchao Zhang$^3$
}

\institute{1 \ \ School of Science, Lanzhou University of Technology, Lanzhou, Gansu, P. R. China. Research supported by the National Natural Science Foundation of China (NSFC-12161053) and the Natural Science Foundation for Distinguished Young Scholars of Gansu Province (22JR5RA223). Email: xkchang@lut.edu.cn. \\
2 \ \ Department of Mathematics, Nanjing University, \#22 Hankou Road, Nanjing, P. R. China.  %Zipcode 210093.
Research supported by the National Natural Science Foundation of China (NSFC-12371301).
Corresponding author (jfyang@nju.edu.cn,  \url{http://maths.nju.edu.cn/~jfyang/}) \\
3 \ \ Department of Mathematics, Louisiana State University, Baton Rouge, LA 70803-4918. Phone (225) 578-1982. Fax (225) 578-4276.
Research supported by the National Natural Science Foundation of U.S.A (DMS-2110722
and DMS-2309549). Email: hozhang@math.lsu.edu. \url{http://www.math.lsu.edu/~hozhang/}
}

\date{Received: date / Accepted: date}

\maketitle

\begin{abstract}
Using convex combination and linesearch techniques, we introduce a novel primal-dual algorithm for solving structured convex-concave saddle point problems with a generic smooth non-bilinear coupling term. Our adaptive linesearch strategy works under specific local smoothness conditions, allowing for potentially larger stepsizes. For an important class of structured convex optimization problems, the proposed algorithm reduces to a fully adaptive proximal gradient algorithm without linesearch, thereby representing an advancement over the golden ratio algorithm delineated in [Y. Malitsky, Math. Program. 2020]. We establish global pointwise and ergodic sublinear convergence rate of the algorithm measured by the primal-dual gap function in the general case. When the coupling term is linear in the dual variable, we measure the convergence rate by function value residual and constraint violation of an equivalent constrained optimization problem. Furthermore, an accelerated algorithm achieving the faster ${\cal O}(1/N^2)$ ergodic convergence rate is presented for the strongly convex case, where $N$ denotes the iteration number.  Our numerical experiments on quadratically constrained quadratic programming and sparse logistic regression problems indicate the new algorithm is significantly faster than the comparison algorithms.
\end{abstract}

\keywords{Convex-concave saddle point problems \and non-bilinear coupling term \and convex combination \and primal-dual algorithm  \and linesearch}

\subclass{49M29 \and  65K10 \and 65Y20 \and 90C25 }

\section{Introduction}
\label{sec:intro}
Let $\R^p$ and $\R^q$ be finite-dimensional Euclidean spaces, each endowed with an inner product and the induced norm denoted by $\langle\cdot, \cdot\rangle$ and $\|\cdot\| =\sqrt{\langle\cdot,\cdot\rangle}$, respectively. Let $f:\R^p \rightarrow(-\infty, +\infty]$ and $g:\R^q \rightarrow(-\infty, +\infty]$ be extended real-valued proper closed and convex functions,  $\Phi: \dom(g) \times \dom(f^*)\rightarrow  \R$ be a continuous function with certain differentiability properties, convex in $x$ and concave in $y$, where $\dom(g)$ denotes the effective domain of $g$ defined by $\dom(g):= \{ x\in\R^m: g(x) < +\infty\}$. Denote the Legendre-Fenchel conjugate of $f$ by $f^*$, i.e., $f^*(y) = \sup_{u\in\R^p}\{ \langle y, u\rangle - f(u)\}$, $y\in\R^p$.
In this paper, we focus on the following structured convex-concave saddle point problem with a generic non-bilinear coupling term
\be\label{saddle_point}
\min\limits_{x\in \R^q}\max\limits_{y\in \R^p}  \cL(x, y):= g(x)+\Phi(x,y)-f^*(y).
\ee
This generic model encompasses several important special cases studied in the literature.  For instance, if $\Phi$ is bilinear, i.e., $\Phi(x,y)=\langle Kx,y\rangle$ for some matrix  $K\in\bR^{p\times q}$, then \eqref{saddle_point} reduces to the structured bilinear saddle point problem, which arises naturally from abundant interesting applications, including signal and image processing, machine learning, statistics, mechanics and economics, and so on, see, e.g., \cite{Chambolle2011A,Bouwmans2016Handbook,Yang2011Alternating,Hayden2013A,Bertsekas1982Projection}
and the references therein.  If, in addition, $g$ and $f^*$ are indicator functions of the unit simplex in $\R^q$ and $\R^p$, respectively, then
\eqref{saddle_point} reduces to the zero-sum matrix game problem.
If both $g$ and $f^*$ vanish, then \eqref{saddle_point} reduces to the unconstrained saddle point problem \cite{App2018GANs,MOS2020}, which has recently found abundant new applications in machine learning \cite{Goodfellow2014nets,Arjovsky2017nets}.
In general,  \eqref{saddle_point} covers a broad class of optimization problems, e.g., convex optimization with nonlinear conic constraints, which itself includes linear programming, quadratic programming, quadratically constrained quadratic programming, second-order cone programming and semidefinite programming as its subclasses, see \cite{EYNS2021} for details.
%
%Indeed, consider
%\be\label{cone-pro}
%\min_{x\in\R^q} \left\{g(x)+h(x),~~\mbox{s.t.}~~H(x)\in-\cK\right\},
%\ee
%where $\cK\in\bR^{p}$ is a closed convex cone, $g$ is convex (possibly
%nonsmooth),  $h$ is convex with a Lipschitz continuous gradient, and $H: \bR^q\rightarrow \bR^p$ is a smooth convex, Lipschitz function having a Lipschitz continuous Jacobian. Various optimization problems that frequently arise in many important applications are special
%cases of the conic problem in \eqref{cone-pro}, e.g., primal or dual formulations of $\ell_1$- or $\ell_2$-norm soft margin support vector machine (SVM), ellipsoidal kernel machines, kernel
%matrix learning, etc. Using Lagrangian duality, one can equivalently rewrite \eqref{cone-pro} to be
%\be\label{cone-pro1}
%\min_{x\in \R^q}\max_{y\in \R^p} ~~\{g(x) + h(x)+ \langle H(x),y\rangle-\iota_{\cK^*}(y)\},
%\ee
%which is a special case of \eqref{saddle_point}, i.e., $\Phi (x,y) = h(x) + \langle H(x),y\rangle$, and $f^*(y) =\iota_{\cK^*}(y)$ is
%the indicator function of $\cK^*$, where $\cK^*$ denotes the dual cone of $\cK$.

To solve \eqref{saddle_point}, popular choices include  various  first-order methods. By reformulating
\eqref{saddle_point} as a mixed variational inequality (MVI) problem and/or monotone operator inclusion problem, various methods have been developed in the literature, e.g., the extra-gradient method \cite{Korpelevich1976,Tseng2000A} and the optimistic gradient descent ascent method \cite{App2018GANs,MOS2020}.
In this paper, motivated from the convex combination technique introduced in \cite{Malitsky2019Golden}, we propose a new primal-dual algorithm (PDA) with adaptive linesearch. Primal-dual type algorithms were first introduced by Arrow and Hurwicz \cite{Uzawa58} and were recently further developed extensively, see \cite{Chambolle2011A,He2014On,EYNS2021,ChY2020Golden,Unified2023PDA} and the references therein.

%which is a natural generalization of the proximal gradient method  for minimax optimization problems and only requires one gradient computation per iteration.

\subsection{Notation}
Let $h$ be an extended real-valued proper closed and convex function defined on a finite dimensional Euclidean space $\R^m$.
The gradient operator and the subdifferential of $h$ at $x\in \R^m$, if exist, are, respectively, denoted by $\nabla h$ and $\partial h(x) := \{\xi\in\R^m: \, h(y) \geq h(x) + \langle \xi, y-x\rangle, \; \forall\, y\in\dom(h)\}$.
For $\lambda >0$, the proximal operator  of $\lambda h$ is given by
$  \prox_{\lambda h}(x) := \arg\min_{y} \{h(y) + {1\over 2\lambda }\|y-x\|^2\}$,
%
%\begin{eqnarray*}\label{def:prox}
%  \prox_{\lambda h}(x) := \arg\min_{y\in \R^m } \Big\{h(y) + {1\over 2\lambda }\|y-x\|^2\Big\}, \quad x\in \R^m,
%\end{eqnarray*}
%
which is uniquely well defined for any  $x\in \R^m$.
If $h$ is differentiable on  $\interior(\dom(h))$, the interior of $\dom(h)$, and $\nabla h$ is Lipschitz continuous with a Lipschitz constant $L_h\geq0$, i.e., $\|\nabla h(x)-\nabla h(y)\|\leq L_h\|x-y\|$ for all $x, y \in \interior(\dom(h))$, then we say that $h$ is $L_h$-smooth.
$\nabla h$ is said to be locally Lipschitz continuous if it is Lipschitz continuous over any bounded subset of $\interior(\dom(h))$.
Given an index set $\mathcal{I}$, $|\mathcal{I}|$ denotes its cardinality.
By convention, we define $1/0=\infty$, in case it happens.

%\comm{I changed $\dom$ to $\interior(\dom())$. Differentiability is only defined for interior points?}

\subsection{Related Work}
Next, we review some closely related work.
The Arrow-Hurwicz method \cite{Uzawa58} was initially proposed to treat \eqref{saddle_point} with bilinear coupling term, i.e., $\Phi(x,y)=\langle Kx,y\rangle$.
A natural extension to the generic saddle point problem \eqref{saddle_point} iterates for $n\geq 1$ as
\ben%\label{pda_AH}
%\mbox{Arrow-Hurwicz}:~~
\left\{
\ba{l}
x_{n}=\prox_{\tau_n g}\big(x_{n-1}-\tau_n  \nabla_x\Phi(x_{n-1},y_{n-1})\big), \smallskip\\
y_{n}=\prox_{\sigma_n f^*}\big(y_{n-1}+\sigma_n  \nabla_y\Phi(x_{n},y_{n-1})\big),
\ea\right.
\een
where $x_0\in\dom(g)$ and $y_0\in\dom(f^*)$ are initial points, and $\tau_n, \sigma_n > 0$ are stepsizes.
The heuristics of the Arrow-Hurwicz method
is to solve the minimax problem \eqref{saddle_point} by alternatingly minimizing with $x$, maximizing with $y$ and meanwhile incorporating the proximity technique by taking into account the latest information.
Convergence of the Arrow-Hurwicz method for the bilinear case was studied with small stepsizes in \cite{Esser2010General}, and sublinear convergence rate results were obtained in
\cite{Chambolle2011A,Nedic2009Subgradient} when $\dom(f^*)$ is bounded.
However, the Arrow-Hurwicz method does not converge in general, see \cite{He2014On,He22On} for counterexamples.  To remedy this issue, Korplevich \cite{Korpelevich1976} and Popov \cite{Popov1980} proposed two different modifications by introducing extrapolation and optimism into the Arrow-Hurwicz method, respectively. For smooth and strongly-convex-strongly-concave objective functions, it was shown to converge linearly \cite{Tseng1995,MOS2020}. The mirror-prox method \cite{Nemirovski04siam} generalizes the extragradient method and works with a general Bregman distance.

Based on the PDA popularized by Chambolle and Pock \cite{Chambolle2011A,Chambolle2016ergodic}, which is merely applicable to the bilinear case too, Zhu et.al \cite{Zhu23On} considered a special case of \eqref{saddle_point}, i.e., $\Phi(x,y) = \langle H(x),y\rangle$, where $H: \R^q\rightarrow \R^p$ is nonlinear and smooth such that $f\circ H$ is convex, and studied a PDA with convergence rate analysis.
Recently, Hamedani and Aybat \cite{EYNS2021} focused on the generic problem \eqref{saddle_point} and proposed the following scheme
\be\label{pda_CP}
%\mbox{CP-PDA}:~~
\left\{
\ba{l}
x_{n}=\prox_{\tau_n g}\big(x_{n-1}-\tau_n \nabla_x\Phi(x_{n-1},y_{n-1})\big), \smallskip\\
z_{n}=(1+\delta)\nabla_y\Phi(x_{n},y_{n-1}) -\delta \nabla_y\Phi(x_{n-1},y_{n-2}), \smallskip\\
y_{n}=\prox_{\sigma_n f^*}\big(y_{n-1}+\sigma_n z_n\big),
\ea\right.
\ee
where $\delta \in (0,1]$.
Apparently, \eqref{pda_CP} can be viewed as an extension of Chambolle and Pock's PDA \cite{Chambolle2011A,Chambolle2016ergodic} from the bilinear case to the generic nonlinear case. A linesearch strategy was also considered in \cite{EYNS2021} to choose the primal and dual stepsizes.
%
%
% a PDA with linesearch to solve \eqref{saddle_point}, which can be viewed as
%
%Specifically, $\nabla_y\Phi(x_{n},y_{n-1})$ in $y$-subproblem of \eqref{pda_AH} is replaced by the extrapolated point  $z_n=(1+\delta)\nabla_y\Phi(x_{n},y_{n-1}) -\delta \nabla_y\Phi(x_{n-1},y_{n-2})$, , resulting the following iterative scheme

Recently, Malitsky \cite{Malitsky2019Golden} proposed  a golden ratio algorithm (GRA) with fully adaptive stepsize for solving MVI problem.
Since the optimality condition of \eqref{saddle_point} can be represented by MVI problem, GRA can thereby be applied.
The iterate scheme of GRA applied to \eqref{saddle_point} takes a Jacobian  form
\ben%\label{GRA}
%\mbox{GRA}:~~
\left\{\ba{ll}
 (z_{n}^x, z_{n}^y)
 =\frac{\psi-1}{\psi} (x_{n-1}, y_{n-1})
                       + \frac{1}{\psi}(z_{n-1}^x, z_{n-1}^y),\smallskip \\
 x_{n}=\prox_{\tau_n g}(z_{n}^x-\tau_n \nabla_x\Phi(x_{n-1},y_{n-1})),  \smallskip \\
 y_{n}=\prox_{\tau_n f^*}(z_{n}^y+\tau_n \nabla_y\Phi(x_{n-1},y_{n-1})),
 \ea\right.
\een
where $\psi\in(1,(1+\sqrt{5})/2]$ is a parameter to determine the convex combination $z_n$, and the stepsize $\tau_n>0$ can be estimated adaptively.
Unfortunately, numerical experiments show that straightforward application to the MVI representation of  \eqref{saddle_point} is much less efficient than primal-dual type methods, e.g., the PDA scheme \eqref{pda_CP}, which are able to take advantage of problem structures thoroughly and can take different stepsizes in the primal and dual subproblems.
Motivated by \cite{Malitsky2019Golden}, a golden ratio PDA (GRPDA) was presented in \cite{ChY2020Golden,ChYZ2022GRPDAL}, with constant or adaptive stepsize determined by linesearch for solving the bilinear cases of  \eqref{saddle_point}.
Based on experimental evidence in \cite{ChYZ2022GRPDAL}, it has been observed that the %number of
additional linesearch trial steps taken by GRPDAs are %tends  to be  approximately one-third of
much less than
those required by Chambolle-Pock's PDA with linesearch \cite{Malitsky2018A}.
In \cite{ChY2022relaxed}, by carrying out a refined analysis, the region of convex combination parameter was expanded further  from $(1, (1+\sqrt{5})/2]$ to $(1, 1+\sqrt{3})$, which increased the weight of $x_{n-1}$ in the convex combination, leading to improved numerical performance.

Considering the aforementioned advantageous properties, our study in this paper is
an extension of GRPDA for convex-concave saddle point problem  \eqref{saddle_point} with coupling term from
 the bilinear case to the generic non-bilinear case.
Roughly speaking, given $z_0=x_0\in\dom(g)$ and $y_0\in\dom(f^*)$, for $n\geq 1$, our new algorithm
takes the following iterations
\be\label{GRPDA}
\mbox{PDAc:}~~\left\{\ba{rcl}
 z_{n}&=&\frac{\psi-1}{\psi} x_{n-1} + \frac{1}{\psi}z_{n-1}, \smallskip \\
 x_{n}&=&\prox_{\tau_n g} \big(z_{n}-\tau_n \nabla_x\Phi(x_{n-1},y_{n-1})\big), \smallskip \\
 y_{n}&=&\prox_{\sigma_n f^*}\big(y_{n-1}+\sigma_n \nabla_y\Phi(x_{n},y_{n-1})\big),
 \ea\right.\ee
where $\psi$ lies in the broader region $(1, 1+\sqrt{3})$ as obtained in \cite{ChY2022relaxed} for
 the bilinear case, and $\tau_n, \sigma_n > 0$ are stepsizes to be determined by
newly designed adaptive linesearch.

\subsection{Contributions}

 %, e.g., those in \cite{TranZhu2020, Radu2021minimax},

To solve \eqref{saddle_point}, existing algorithms require knowledge of the global Lipschitz constants of $\nabla_x\Phi$ and $\nabla_y\Phi$ with respect to $x$, $y$, and $(x,y)$,
% \comm{better to provide at least two references here},
which can be challenging to obtain in practice and  poor estimates of these constants can
significantly deteriorate the practical performance.
Moreover, even with known Lipschitz constants, the
stepsizes derived from these global Lipschitz constants are usually much overconservative
since they fail to utilize local geometry, resulting slow practical convergence.
  Although the backtracking linesearch scheme proposed in \cite{EYNS2021} estimates stepsizes adaptively, it requires to update both variables $x$ and $y$ in every trial linesearch step, leading to expensive computational cost and potentially lower efficiency.

Our first contribution is to propose an algorithm with adaptive linesearch
that addresses all the above mentioned limitations.
In particular, the new algorithm is in the convex combination based PDA framework (\ref{GRPDA}),
and applies a novel linesearch to estimate stepsizes adaptively. Unlike previous approaches, our algorithm only updates the dual variable $y$ for each linesearch trial step, rather than updating both variables. More importantly, our algorithm does not require any prior knowledge of the global Lipschitz constants and instead utilizes local geometry in the linesearch to improve the overall performance,
making the algorithm much more efficient across various applications.

Another contribution of this paper is a fully adaptive algorithm for solving the structural convex optimization problem
\begin{equation}
   \label{composite_opt}
\min_{x\in\R^q} g(x)+ h(x) \text{~~with~~} h(x) := \frac{1}{p}\sum\nolimits_{i=1}^p h_i(x),
\end{equation}
where, for each $i=1,\ldots, p$, $h_i: \R^q\rightarrow\R$ is convex and differentiable with locally Lipschitz continuous gradient $\nabla h_i$.
Let $H(x) := (h_1(x), \ldots, h_p(x))^\top$, $\mathds{1} := (1,\ldots, 1)^\top \in\R^p$,
and $\iota_{\mathds{1}/p }$ be the indicator function of the singleton $\{\mathds{1}/p\}$.
Then, \eqref{composite_opt} can be reformulated as
\be\label{pd_gf}
\min_{x\in\R^q}\max_{y\in\R^p} g(x)+\langle y, H(x)\rangle-\iota_{\mathds{1}/p }(y),
\ee
which is apparently a special case of \eqref{saddle_point}.
By applying our proposed algorithm in the paper (Algorithm \ref{algo1}) to this special case, we
obtain a fully adaptive proximal gradient method with convex combination,
which gives explicit stepsizes based on the available information computed through the iterations,
without the need of a linesearch procedure and any prior knowledge of Lipschitz constants,
see Section \ref{sec:adaptive} for details.

Moreover,  we establish the iterative global convergence and sublinear ergodic convergence rate
of the proposed algorithm under local Lipschitz continuity assumptions of  $\nabla_x\Phi$ and $\nabla_y\Phi$,
while many algorithms (e.g. \cite{EYNS2021,Nemirovski04siam}) require global
Lipschitz continuity properties for ensuring convergence.
Specifically, for the general nonbiliear case \eqref{saddle_point}, the primal-dual gap $\cL(x_{n},y^\star)- \cL(x^\star, y_n)$ is adopted to quantify the convergence rate,  where $(x^\star,y^\star)$ is any saddle point of $\cL(\cdot)$.
When the coupling term $\Phi(x,y)$ is linear in one of the variables, we first reformulate \eqref{saddle_point} as a constrained optimization problem and then establish convergence rate results using function value residual and constraint violation as in
\cite{SabT22SIOPT,Teboulle2014Rate,ChYZ2022GRPDAL}.
Furthermore, we propose an accelerated algorithm achieving faster convergence rate
for the strongly convex case. Our numerical experiments
on quadratically constrained quadratic programming (QCQP) and  sparse logistic regression (SLR) problems
show that the proposed algorithms are significantly faster (or even on the order of faster for solving
QCQP problems) than the comparison algorithms.

\subsection{Organization}
The organization of the remaining paper is outlined as follows. Section \ref{sec:asmp} provides basic assumptions, necessary facts, and notation. The main algorithm, a variant of PDAc with linesearch to determine stepsizes, is introduced in Section \ref{sec-GRPDA-L}. Convergence results and sublinear convergence rate results are also established in this section.
In Section \ref{sec:special problem}, we focus on the nonlinear compositional convex optimization problem, provide a different analysis based on function value residual and constraint violation and introduce an accelerated algorithm for the strongly convex case.
In Section \ref{sec:adaptive}, we demonstrate that the proposed approach reduces to a fully adaptive proximal gradient method  when applied to \eqref{composite_opt}.
Section \ref{sec-experiments} presents numerical results on QCQP and SRL problems. Comparisons with state-of-the-art algorithms are included as well.
Finally, Section \ref{sec-conclusion} provides some concluding remarks.

\section{Assumptions and Preliminaries}\label{sec:asmp}
The assumption of global Lipschitz continuity of $\nabla_x\Phi$ and $\nabla_y\Phi$, which is commonly used in many existing results \cite{EYNS2021,Nemirovski04siam}, may not hold in practice for many functions.
In this paper, we show that local Lipschitz continuity of $\nabla_x\Phi$ and $\nabla_y\Phi$ is sufficient for establishing convergence of the adaptive PDA algorithm with linesearch proposed in this paper.
A pair $(x^\star,y^\star) \in \dom(g)\times \dom(f^*)$ is said to be a saddle point of $\cal L(\cdot)$ or \eqref{saddle_point} if it satisfies
%the  saddle point condition
\ben%\label{saddle-L}
\cL({x^{\star}},y) \leq \cL({x^{\star}},{y^{\star}}) \leq\cL(x,{y^{\star}}) \text{~~for all~~} (x, y) \in \dom(g)\times \dom(f^*).
\een
We denote the set of  all  saddle points of $\cL(\cdot)$ by $\Omega$, i.e.,
\begin{equation}
\label{def:Omega}
\Omega = \{({x^{\star}}, {y^{\star}}) \in \dom(g)\times \dom(f^*)  \mid  -\nabla_x\Phi({x^{\star}}, {y^{\star}})  \in \partial g({x^{\star}}), \,\, \nabla_y\Phi({x^{\star}}, {y^{\star}})  \in \partial f^*({y^{\star}})\}.
\end{equation}
Throughout the paper, we make the following blanket assumptions.
\begin{asmp}\label{asmp-1}
Assume that problem \eqref{saddle_point} has at least one saddle point, i.e., $\Omega\neq \emptyset$. Moreover, $\dom(g)\times \dom(f^*)\subseteq \dom(\Phi )$ and
$\cL({x^{\star}},{y^{\star}})$ is finite.
\end{asmp}

\begin{asmp}\label{asmp-0}
Assume that    $\Phi: \dom(g) \times \dom(f^*)\rightarrow  \R$ is continuous such that
\begin{enumerate}
  \item[(i)] (\textbf{convexity and concavity})
for any $y\in \dom(f^*)$, $\Phi (\cdot, y)$ is  convex  and differentiable w.r.t. the first component, and for any $x \in \dom(g)$, $\Phi (x,\cdot )$ is concave  and differentiable w.r.t. the second component;

\item[(ii)] (\textbf{local Lipschitz continuity})
for any bounded subsets $\cX \subset \bR^q$ and $\cY \subset \R^p$, there exist  $L_{yy}\geq 0, L_{xx}\geq 0$, $L_{xy}> 0$ such that
for any $x, \tilde{x} \in \cX \cap \dom(g)$ and $y, \tilde{y} \in \cY \cap \dom(f^*)$ there hold
\ben
\| \nabla_y \Phi (x,y) - \nabla_y \Phi (x,\tilde{y})\| &\leq& L_{yy} \|y-\tilde{y}\|,\label{lip_y}\\
\| \nabla_x \Phi (x,y) - \nabla_x \Phi (\tilde{x}, \tilde{y})\| &\leq& L_{xx} \| x-\tilde{x}\| + L_{xy} \|y-\tilde{y}\|.
\label{lip_x}
\een
\end{enumerate}
\end{asmp}

\begin{rem}
We emphasize that the local Lipschitz continuity in Assumption \ref{asmp-0} (ii) is only for
theoretical analysis purpose. Our main algorithm,  Algorithm \ref{algo1}, is parameter-free
in the sense that it
does not depend on the local Lipschitz constants $L_{yy}, L_{xx}$, or $L_{xy}$ in Assumption \ref{asmp-0} (ii), nor on any other parameters associated with the saddle point problem.
While the local Lipschitz constants $L_{yy}, L_{xx}$, and $L_{xy}$ should
rely on the bounded subsets $\cX$ and $\cY$, for simplifying the analysis and notation,
we omit its explicit dependence on particular bounded subsets $\cX$ and $\cY$.
Under the assumption of local Lipschitz continuity, we will prove that the sequence generated by our algorithm is bounded (see Lemma \ref{lem_bound} (ii)) and, in fact, converges (see Theorem \ref{thm12}).
Therefore, in the analysis, it is sufficient to choose $L_{yy}, L_{xx}$, and $L_{xy}$ to be sufficiently large, ensuring that Assumption \ref{asmp-0} (ii) holds for bounded sets $\cX$ and $\cY$, which guarantees that the entire sequences generated by the algorithm are located within $\cX \times \cY$.
\end{rem}

In addition, we make the following assumptions on $f$ and $g$, which are widely satisfied in
many practical applications, see, e.g., \cite[Chapter 6]{Beck2017book}.

\begin{asmp}\label{asmp-2}
  Assume that the proximal operators of the component functions $f$ and $g$ either have closed form formulas or can be evaluated efficiently.
\end{asmp}

The following simple facts and identities are useful in our analysis.
\begin{fact}\label{fact_proj}
Let $h: \R^m\rightarrow (-\infty, +\infty]$ be an extended real-valued closed proper and strongly convex function with modulus $\gamma \geq 0$. Then for any $\tau>0$ and $x\in \R^m$, it holds that $z = \prox_{\tau h}(x)$ if and only if
$h(y) \geq h(z)+ {1\over \tau}\langle x-z, y-z\rangle + {\gamma \over 2}\|y-z\|^2$ for all $y\in \R^m$.
\end{fact}

\begin{fact}\label{fact_ab}
Let $\{u_n\}$ and $\{v_n\}$ be two real and nonnegative  sequences.
If, for some $\varepsilon \in (0,1)$,  $u_{n+1}\leq \varepsilon u_{n} +v_n$ for all $n\geq1$ and $\sum\nolimits_{n=1}^{\infty}v_n < \infty$, then $\sum\nolimits_{n=1}^{\infty} u_{n}  <\infty$.
\end{fact}
\begin{fact}\label{fact_uv}
For any $u, v, a, b\in \bR$ such that $u + v > 0$, there holds $\frac{uv}{u+v}(a+b)^2\leq ua^2+v b^2$.
\end{fact}
%
%\revise{Let $ f,g : \bR^q \rightarrow \bR$ are two proper l.s.c. convex functions with $\dom g \subseteq \dom f$. We denote the forward-backward operator $\cP: \dom g \times \bR_+ \rightarrow \dom g$ by
%\ben
%\cP(x,\alpha) := \prox_{\alpha g} (x-\alpha \nabla f(x)),~~\forall x\in \dom g \subseteq \dom f, \alpha > 0.
%\een
%\begin{fact}[{\cite[Lemma 1]{Huang14New}}]\label{fact_fb}
%For any $x\in\dom g$ and $\alpha_2 \geq \alpha_1> 0$, we have
%\ben
%\frac{\alpha_2}{\alpha_1}\|x - \cP(x,\alpha_1)\| \geq \|x - \cP(x,\alpha_2)\| \geq \|x - \cP(x,\alpha_1)\|.
%\een
%\end{fact}}
%
%
For any $x, y, z\in \bR^m$ and $\alpha\in\bR$, there hold
\be
2\langle x-y, x-z\rangle&=&  \|x-y\|^2 +  \|x-z\|^2  -\|y-z\|^2,\label{id}\\
\|\alpha x+(1-\alpha)y\|^2&=& \alpha \|x\|^2+(1-\alpha)\|y\|^2-\alpha(1-\alpha)\|x-y\|^2.\label{id2}
\ee

\section{PDAc with Linesearch}
\label{sec-GRPDA-L}
In this section, we introduce our PDA with convex combination for solving \eqref{saddle_point}, where adaptive stepsizes is adopted by linesearch. Define
 \be
\theta_n:&=& \nabla_x\Phi(x_{n},y_{n})-\nabla_x\Phi(x_{n-1},y_{n-1}),\label{def:q-La}\\
\Phi_n^y:&=& \Phi(x_n,y_{n-1})+\langle \nabla_y \Phi (x_n, y_{n-1}),y_n-y_{n-1} \rangle-\Phi (x_n,y_{n}),\label{def:phi-y} \\
\omega &:=& \omega(\xi,\varphi) := 2\psi-\xi-\frac{\psi^3\varphi}{1+\psi}, \quad \forall \psi\in(1,1+\sqrt{3}), \label{def:zeta}\\
\Theta_\psi &:=& \{(\xi,\varphi)~|~
\xi>0,~~\varphi>1~~\mbox{and}~~\omega(\xi,\varphi)>0\}.\label{def:Theta}
\ee
It can be verified that with $\psi\in(1,1+\sqrt{3})$, the set $\Theta_\psi$ is nonempty.
Using the notation defined in \eqref{def:q-La}-\eqref{def:Theta}, the basic scheme is summarized in Algorithm \ref{algo1}.
\vskip5mm
\hrule\vskip2mm
\begin{algo}
[PDAc with Linesearch (PDAc-L)]\label{algo1}
{~}\vskip 1pt {\rm
\begin{description}
\item[{\em Step 0.}] Choose $\psi\in(1,1+\sqrt{3})$,  $(\xi,\varphi)\in \Theta_\psi$, $\tau_{\max}>0$,
$\nu\in(0, 1)$,  $\mu\in(0,1)$, $\eta\in [0,1)$ and integer $M \ge 1$.
Choose $x_0 \in \dom(g),$ $y_0\in \dom(f^*)$, $\beta>0$ and $\tau_0 \in (0, \tau_{\max}]$.
    Set $z_{0}=x_0$, $\omega:= \omega(\xi,\varphi)$, $\delta_0=1$ and $n=1$.
\item[{\em Step 1.}] Compute
\be\label{x_updating}
  z_{n}=\frac{\psi-1}{\psi} x_{n-1} + \frac{1}{\psi}z_{n-1},~~
 x_{n}=\prox_{\tau_{n-1} g}(z_{n}-\tau_{n-1} \nabla_x\Phi(x_{n-1},y_{n-1})).
 \ee
\item[{\em Step 2.}] Set $\tau = \min\{\varphi\tau_{n-1}, \tau_{\max}\}$ and compute
\be\label{y_updating}
y_{n}&=&\prox_{\beta\tau_{n} f^*}(y_{n-1}+\beta\tau_{n} \nabla_y\Phi(x_{n},y_{n-1})),
\ee
where $\tau_n = \tau \mu^i$ and $i$ is the smallest nonnegative integer such that
\be\label{linesearch-cond1}
\frac{\tau_n\tau_{n-1}}{\xi}\|\theta_n\|^2 +2\tau_n \Phi_n^y &\leq& \nu r_n + (1-\nu) c_n
\ee
with $r_n = \omega\delta_{n-1} \|x_{n}-x_{n-1}\|^2
+\frac{1}{\beta}\|y_n-y_{n-1}\|^2$, $c_n = (\eta/|\mathcal{I}_n|)
\sum_{i \in \mathcal{I}_n} r_i$ and
$$\mathcal{I}_n = \left\{n-1, n-2, \ldots, \max\{n-M, 1\} \right\}.$$

%Otherwise, set $\tau_n =\mu\tau_n$ and go to (a).
\item[{\em Step 3.}] Set $\delta_n= \tau_n/\tau_{n-1}$,  $n \leftarrow n+1$ and go to Step 1.
  \end{description}
}
\end{algo}
\vskip1mm\hrule\vskip5mm

\begin{rem}\label{rem_tau_max}
We give the following remarks on Algorithm \ref{algo1}.
\begin{enumerate}
  \item  The constant $\tau_{\max}$ in Algorithm \ref{algo1} is to ensure that $\{\tau_n: n\geq1\}$ is bounded, and thus it can be chosen to be a very large value in practice, see \cite[Algorithm 1]{Malitsky2019Golden} for similar remarks.

    \item It's worth emphasizing that in Algorithm \ref{algo1}, only the variable $y$ needs to be updated within each linesearch step. In contrast, both  $x$ and $y$ need to be updated within each linesearch step in \cite[Algorithm 2.3]{EYNS2021}. Note that in many applications the objective function can be written as $\Phi(x,y)=\langle M(x),y\rangle+h(x)$, where $M:\R^q\rightarrow \R^p$. In this case,  $\nabla_x\Phi(x,y)=M'(x)^\top y+\nabla h(x)$, and $\Phi_n^y \equiv 0$ for our algorithm. Therefore, our algorithm  has a very cheap cost in checking the condition \eqref{linesearch-cond1}.  In contrast, validating the linesearch condition in \cite[Algorithm 2.3]{EYNS2021} is much more expensive since it requires computing $\nabla_x\Phi$ at each intermediate trial point.

    \item For any $\psi\in(1,1+\sqrt{3})$,  $\xi$ and $\varphi$ can be flexibly selected in the region $\Theta_\psi$ defined in \eqref{def:Theta}. For instance,   $\xi$ and $\varphi$ can be chosen in $\{(\xi,\varphi)~|~\xi>0, \, \varphi>1, \, 20\xi+27\varphi<60\}$ for $\psi=3/2$. % as in \cite{Malitsky2019Golden},
 % If one set $\psi=2$, parameters $\xi$ and $\varphi$ can be selected in $\{(\xi,\varphi)~|~0<\xi, 1<\varphi, ~~3\xi+8\varphi<12\}$.
\end{enumerate}
\end{rem}

%Note that in Algorithm \ref{algo1} parameter $\beta$ plays the role of the ratio $\frac{\sigma_n}{\tau_n}$, namely the stepsize $\sigma_n = \beta\tau_n$ in $y$-subproblem. To get convergence of Algorithm \ref{algo1} for each fixed $\beta$, it is of great important to select stepsizes $\tau_n$ carefully in Step 2.

Without mentioning repeatedly, below we let $({x^{\star}},{y^{\star}})$ be an arbitrarily point in $\Omega$ (see definition in \eqref{def:Omega})  and  $\{(z_n,x_n,y_n)\}$ be the sequence generated by Algorithm \ref{algo1}.
Recall that ${\cal L}(\cdot)$ is defined in \eqref{saddle_point}.
For any  $(x,y)\in\dom(g)\times\dom(f^*)$, we further define the primal-dual gap function by
\be\label{def:J}
J(x,y) := \cL(x,y^\star)- \cL(x^\star, y),
\ee
which was used in \cite{MOS2020,ChY2020Golden,EYNS2021} to measure convergence rate.
Apparently, $J(x,y)$ is jointly convex in $(x,y)$ and $J(x,y)\geq 0$ for any $(x,y)$.
In the following, we  present some basic properties of the sequence  $\{(z_n,x_n,y_n)\}$ and   explore the properties of $\tau_n$, based upon which convergence of the algorithm will be established. % in detail to establish convergence.

\subsection{Basic Properties of PDAc-L}
For any bounded subset $\cY\subseteq\R^p$, $x \in \dom(g)$ and $y, \tilde{y} \in \cY \cap \dom(f^*)$,
it follows from Assumption \ref{asmp-0} that
\be\label{concavity-L}
-\frac{L_{yy}}{2}\| y - \tilde{y}\|^2  \leq \Phi (x,y) - \Phi (x,\tilde{y}) - \langle \nabla_y \Phi (x,\tilde{y}), y -\tilde{y}\rangle \leq0.
\ee
Next, we present two lemmas, which play critical roles in the convergence analysis.

\begin{lem}\label{lem1}
%For any $(x^\star,y^\star)\in \Omega$, there holds
For  $\theta_n$, $\Phi_n^y$ and $J(\cdot)$  defined in  \eqref{def:q-La}, \eqref{def:phi-y} and \eqref{def:J}, respectively, there holds
\be
\tau_n J(x_n,y_n)
&\leq&\langle x_{n+1}-z_{n+1}, x^\star-x_{n+1}\rangle+ \frac{1}{\beta} \langle  y_{n}-y_{n-1}, y^\star-y_{n}\rangle \nonumber\\
&& + \psi\delta_{n} \langle x_{n}-z_{n+1}, x_{n+1}- x_n \rangle + \tau_n\langle \theta_n, x_{n}-x_{n+1}\rangle +\tau_n\Phi_n^y. \label{lem1:ineq}
\ee
\end{lem}
\begin{proof}
   It follows from  (\ref{x_updating}), (\ref{y_updating}) and Fact  \ref{fact_proj} that
\be
 \tau_{n}\big(g(x_{n+1})-g(x^\star)\big)
&\leq & \big\langle x_{n+1}-z_{n+1} + \tau_{n}\nabla_x\Phi(x_{n},y_{n}), ~ x^\star-x_{n+1}\big\rangle, \label{temp_01}\\
\tau_{n-1} \big(g(x_{n})-g(x_{n+1})\big) &\leq & \langle x_{n}-z_{n} + \tau_{n-1}\nabla_x\Phi(x_{n-1},y_{n-1}), ~ x_{n+1}-x_{n}\rangle, \label{tem1} \\
\tau_n \big(f^*(y_{n})-f^*(y^\star)\big)&\leq&\big\langle {1\over\beta} (y_{n}-y_{n-1}) - \tau_n \nabla_y\Phi(x_{n},y_{n-1}), ~y^\star-y_{n}\big\rangle. \label{temp_001}
\ee
Multiplying (\ref{tem1}) by $\delta_n = \tau_{n} /\tau_{n-1}$ and using
$x_{n}-z_{n} = \psi (x_{n}-z_{n+1})$, which follows from the definition of $z_n$ in \eqref{x_updating}, we obtain
\be\label{temp_02}
\tau_{n}\big(g(x_{n})-g(x_{n+1})\big)\leq\left\langle \psi\delta_n(x_{n}-z_{n+1}) + \tau_{n}\nabla_x\Phi(x_{n-1},y_{n-1}), ~x_{n+1}-x_{n}\right\rangle.
\ee
It follows from the right-hand-side of \eqref{concavity-L} that
%, the second term of inner product in \eqref{temp_001} can be \red{upper} bounded using concavity of $\Phi (x_n, y)$ in $y$ as follows:
\ben
-\langle \nabla_y \Phi (x_n, y_{n-1}), y^\star- y_n\rangle %& =&\langle \nabla_y \Phi (x_n, y_{n-1}), y_n -y^\star\rangle\\
&=&\langle \nabla_y \Phi (x_n, y_{n-1}), y_n- y_{n-1}\rangle + \langle \nabla_y \Phi (x_n, y_{n-1}), y_{n-1}- y^\star\rangle\\
&\leq& \langle \nabla_y \Phi (x_n, y_{n-1}), y_n- y_{n-1}\rangle + \Phi (x_n,y_{n-1}) - \Phi (x_n,y^\star).
\een
Using the definition of $J(\cdot)$ in \eqref{def:J}, % ${\cal L}(\cdot)$ and in \eqref{saddle_point} and , respectively,
taking the sum of (\ref{temp_01}), (\ref{temp_001})-(\ref{temp_02}), and using the above inequality, we obtain from elementary calculations that
\be\label{temp_03}
\tau_{n}J(x_n, y_n) %&=&\tau_{n}(\cL(x_{n},y^\star)- {\cL(x^\star, y_n)}) \nonumber\\
%&=& \tau_{n}\big(g(x_{n})+\Phi(x_n,y^\star)- f^*(y^\star)-g(x^\star)-\Phi(x^\star,y_n)+f^*(y_{n})\big)\nonumber\\
&\leq&\langle x_{n+1}-z_{n+1},~  x^\star-x_{n+1}\rangle +\frac{1}{\beta} \langle y_{n}-y_{n-1}, y^\star-y_{n}\rangle \nonumber \\
& & + \psi\delta_n\left\langle x_{n}-z_{n+1}, ~ x_{n+1}- x_{n}\right\rangle +\tau_{n}\cG_n,
\ee
where
%$\cG_n :=\langle \nabla_x\Phi(x_{n},y_{n}), ~ x^\star-x_{n+1}\rangle  + \big\langle \nabla_x\Phi(x_{n-1},y_{n-1}), ~x_{n+1}-x_{n}\big\rangle + \langle \nabla_y \Phi (x_n, y_{n-1}), y_n- y_{n-1}\rangle +\Phi (x_n,y_{n-1})-\Phi(x^\star,y_n)$.
%
\ben
\cG_n
&:=&\langle \nabla_x\Phi(x_{n},y_{n}), ~ x^\star-x_{n+1}\rangle  + \big\langle \nabla_x\Phi(x_{n-1},y_{n-1}), ~x_{n+1}-x_{n}\big\rangle\\
&& + \langle \nabla_y \Phi (x_n, y_{n-1}), y_n- y_{n-1}\rangle +\Phi (x_n,y_{n-1})-\Phi(x^\star,y_n).
\een
We can easily show, using % the definitions of $\theta_n$ and $\Phi_n^y$ in
\eqref{def:q-La}-\eqref{def:phi-y} and the convexity of $\Phi$ in $x$, that
\ben
%
%
%&=&\langle \theta_n, x_{n}-x_{n+1}\rangle+\langle\nabla_x\Phi(x_{n},y_{n}), ~ x^\star-x_{n}\rangle\\
%&&+ \langle \nabla_y \Phi (x_n, y_{n-1}), y_n- y_{n-1}\rangle +\Phi (x_n,y_{n-1})-\Phi(x^\star,y_n)\\
%
\cG_n &=&\langle \theta_n, x_{n}-x_{n+1}\rangle+\Phi_n^y  + \big(\Phi (x_n,y_{n})+\langle\nabla_x\Phi(x_{n},y_{n}), ~ x^\star-x_{n}\rangle-\Phi (x^\star,y_{n}) \big) \\
%&&+\Phi(x_n,y_{n-1})+\langle \nabla_y \Phi (x_n, y_{n-1}),y_n-y_{n-1} \rangle-\Phi (x_n,y_{n})\\
%
&\leq &\langle \theta_n, x_{n}-x_{n+1}\rangle +\Phi_n^y.
\een
This together with \eqref{temp_03} implies \eqref{lem1:ineq} immediately.
\end{proof}

%Then, by Lemma \ref{lem1} we have the following result.
For any $(x^\star, y^\star)\in \Omega$,  we define
\begin{equation} \label{ab_n}
\left\{\ba{rcl}
a_n &:=&\frac{\psi}{\psi-1}\|z_{n+1}-x^\star\|^2+\frac{1}{\beta}\|y_{n-1}-y^\star\|^2 +\omega\delta_{n-1}\|x_n-x_{n-1}\|^2, \smallskip \\
b_n &:=&-\frac{\tau_n\tau_{n-1}}{\xi}\|\theta_n\|^2 +\frac{1}{\beta}\|y_n-y_{n-1}\|^2 +\omega\delta_{n-1}\|x_{n}-x_{n-1}\|^2 -2\tau_n \Phi_n^y.
 \ea
 \right.
 \end{equation}

\begin{lem}\label{lem2}
%Let $J(\cdot)$ be defined \eqref{def:J}.
For all $n\geq 1$, there holds  $a_{n+1}+2 \tau_n J(x_n,y_n)\leq a_{n}-b_n$, where $a_n$ and $b_n$ are defined in \eqref{ab_n}.
\end{lem}
\begin{proof}
Fix $n\geq 1$ arbitrarily.
By applying identity (\ref{id}) to the first three inner products in \eqref{lem1:ineq} and reorganizing the terms, we obtain
\be\label{ineq11}
&&\|x_{n+1}-x^\star\|^2+\frac{1}{\beta}\|y_{n}-y^\star\|^2+ 2 \tau_n J(x_n,y_n)\nonumber\\
&\leq&\|z_{n+1}-x^\star\|^2+\frac{1}{\beta}\|y_{n-1}-y^\star\|^2 + 2\tau_{n}\langle \theta_n, x_{n}-x_{n+1}\rangle +2\tau_n \Phi_n^y \\
&&-\psi\delta_n\|z_{n+1}-x_n\|^2-(1-\psi\delta_n)\|x_{n+1}-z_{n+1}\|^2
-\psi\delta_n\|x_{n+1}-x_n\|^2-\frac{1}{\beta}\|y_{n}-y_{n-1}\|^2.  \nonumber
\ee
Since $x_{n+1}=\frac{\psi}{\psi-1}z_{n+2}-\frac{1}{\psi-1}z_{n+1}$, which follows from \eqref{x_updating}, we deduce from
(\ref{id2}) that
\be\label{x-to-z}
\|x_{n+1}-x^\star\|^2
&=& {\psi\over \psi-1} \|z_{n+2}-x^\star\|^2- {1 \over \psi-1} \|z_{n+1}-x^\star\|^2 + {\psi \over (\psi-1)^2} \|z_{n+2}-z_{n+1}\|^2\nonumber\\
&=& {\psi\over \psi-1} \|z_{n+2}-x^\star\|^2- {1 \over \psi-1} \|z_{n+1}-x^\star\|^2+\frac{1}{\psi}\|x_{n+1}-z_{n+1}\|^2,
\ee
where the second equality is due to $z_{n+2}-z_{n+1} = {\psi-1\over\psi} (x_{n+1} - z_{n+1})$.
By plugging \eqref{x-to-z} into (\ref{ineq11}), we obtain
\be\label{ineq_rate1}
&&\frac{\psi}{\psi-1}\|z_{n+2}-x^\star\|^2+\frac{1}{\beta}\|y_{n}-y^\star\|^2+ 2 \tau_n J(x_n,y_n)\nonumber\\
&\leq&\frac{\psi}{\psi-1}\|z_{n+1}-x^\star\|^2+\frac{1}{\beta}\|y_{n-1}-y^\star\|^2 + 2\tau_{n}\langle \theta_n, x_{n}-x_{n+1}\rangle+2\tau_n\Phi_n^y -\frac{1}{\beta}\|y_n-y_{n-1}\|^2  \nonumber \\
&&-\psi\delta_n\|z_{n+1}-x_n\|^2-\big(1+\frac{1}{\psi}-\psi\delta_n\big)\|x_{n+1}-z_{n+1}\|^2
-\psi\delta_n\|x_{n+1}-x_n\|^2.
\ee
By using Fact \ref{fact_uv} with $u=\psi\delta_n>0$ and $v=1+\frac{1}{\psi}-\psi\delta_n$ so that $u+v = 1+\frac{1}{\psi} >0$,
$a=\|z_{n+1}-x_n\|$, $b=\|x_{n+1}-z_{n+1}\|$ and noting   $\|x_{n+1}-x_{n}\|\leq  a+b$, we can easily obtain
\ben%\label{ineq-improve}
\psi\delta_n\|z_{n+1}-x_n\|^2+(1+\frac{1}{\psi}-\psi\delta_n) \|x_{n+1}-z_{n+1}\|^2
%&\geq&-\frac{\psi\delta_n(1+\psi-\psi^2\delta_n)}{\psi+1} \|x_{n+1}-x_n\|^2\nonumber\\
%&=&
\geq\psi\delta_n(1-\frac{\psi^2\delta_n}{1+\psi}) \|x_{n+1}-x_n\|^2.
\een
This together with \eqref{ineq_rate1} gives
\be\label{ineq_rate2}
&&\frac{\psi}{\psi-1}\|z_{n+2}-x^\star\|^2+\frac{1}{\beta}\|y_{n}-y^\star\|^2+ 2 \tau_n J(x_n,y_n)\nonumber\\
&\leq&\frac{\psi}{\psi-1}\|z_{n+1}-x^\star\|^2+\frac{1}{\beta}\|y_{n-1}-y^\star\|^2 + 2\tau_{n}\langle \theta_n, x_{n}-x_{n+1}\rangle+2\tau_n\Phi_n^y  \nonumber \\
&&-\frac{1}{\beta}\|y_n-y_{n-1}\|^2 -\psi\delta_n(2-\frac{\psi^2\delta_n}{1+\psi})\|x_{n+1}-x_n\|^2.
\ee
%Next, we bound the term $\langle \theta_n, x_n-x_{n+1}\rangle$. For any $y, y^\prime\in\R^q$  and $a>0$, we have $\langle y^\prime,y\rangle \leq
%\frac{a}{2}\|y\|^2+\frac{1}{2a}\|y^\prime\|^2$.
%Using this inequality with $a=\frac{\xi}{\tau_{n-1}}$ for constant $\xi>0$, we obtain that
Recall that $\delta_n = \tau_n/\tau_{n-1}$.
Plug the Cauchy-Schwartz inequality
$2\langle \theta_n, x_n-x_{n+1}\rangle \leq\frac{\xi}{\tau_{n-1}}\|x_n-x_{n+1}\|^2+ \frac{\tau_{n-1}}{\xi}\|\theta_n\|^2$, which holds for any $\xi>0$,
into \eqref{ineq_rate2} to  obtain
\be
&&\frac{\psi}{\psi-1}\|z_{n+2}-x^\star\|^2+\frac{1}{\beta}\|y_{n}-y^\star\|^2 +\delta_n(2\psi-\xi-\frac{\psi^3\delta_n}{1+\psi})\|x_{n+1}-x_n\|^2+ 2 \tau_n J(x_n,y_n)\nonumber\\
&\leq&\frac{\psi}{\psi-1}\|z_{n+1}-x^\star\|^2+\frac{1}{\beta}\|y_{n-1}-y^\star\|^2 + \frac{\tau_n\tau_{n-1}}{\xi}\|\theta_n\|^2+2\tau_n\Phi_n^y-\frac{1}{\beta}\|y_n-y_{n-1}\|^2. \label{yang-0821}
\ee
It follows from $\delta_n \leq \varphi$ (see Step 2 of Algorithm \ref{algo1}) and \eqref{def:zeta} that
$2\psi-\xi-\frac{\psi^3\delta_n}{1+\psi}\geq \omega$.
By reorganizing the terms in \eqref{yang-0821} and taking into account $a_n$ and $b_n$ defined in \eqref{ab_n}, we
obtain the desired result $a_{n+1}+2 \tau_n J(x_n,y_n)\leq a_{n}-b_n$.
\end{proof}

%The updating rule of $\tau_n$ in Algorithm \ref{algo1} is crucial to ensure that
%$b_n$ is also nonnegative, which further guarantees the convergence of the algorithm.
Based on Lemma~\ref{lem2}, we next establish some properties of the sequences generated by Algorithm \ref{algo1}.
Recall that  $\theta_n$ and $\Phi_n^y$  are defined in  \eqref{def:q-La} and \eqref{def:phi-y}, respectively.
Next, we show that the linesearch step in Step 2 of Algorithm \ref{algo1} is well-defined and establish some important properties
for the sequences $\{(x_n,y_n,z_n)\}$, $\{\tau_n\}$ and $\{\delta_n\}$, which are essential for proving the convergence results.
For convenience we define
\be\label{def:tau-underline}
\underline{\tau}:&=&\min\left\{\frac{\nu\xi}{\beta(2L_{xy}^2\tau_{\max}+ L_{yy}\xi)},~ \frac{\nu\omega\xi}{2L_{xx}^2\tau_{\max}}\right\}.
\ee

\begin{lem}\label{lem_bound}
The following claims hold. (i) The linesearch step of Algorithm \ref{algo1} always terminates,  i.e., $\{\tau_n\}$ is well defined;
(ii) The sequence $\{(x_n,y_n,z_n)\}$ generated by Algorithm \ref{algo1} is bounded; (iii) If $\tau_n \leq \underline{\tau}$, then the linesearch condition \eqref{linesearch-cond1} is satisfied;
and (iv) Assuming that $\tau_0\geq\mu \underline{\tau}$ and $\tau_{\max} \geq \mu\underline{\tau}$, then the sequences $\{\tau_n\}$ and $\{\delta_n\}$ generated by Algorithm \ref{algo1} are strictly separated from $0$. In fact, there hold $\tau_n\geq \mu \underline{\tau}$ and $\delta_n\geq \mu \underline{\tau}/\tau_{\max}$ for all $n\geq 1$.
%{(\color{red} this needs assumption  $\tau_0\geq\mu \underline{\tau}$ and $\tau_{\max} \geq \mu\underline{\tau}$)}
\end{lem}
\begin{proof}
(i) Fix $n$ arbitrarily  and recall that $\tau=\min\{\varphi\tau_{n-1}, \tau_{\max}\}$ at the $n$-th iteration.
For convenience, we define
\be\label{def:ytP}
\left.\ba{lll}y_{n}(\la) &:=& \prox_{\beta\la f^*}(y_{n-1}+\beta\la \nabla_y\Phi(x_{n},y_{n-1})) \text{~~for~~} \lambda > 0, \smallskip\\
\theta_n(\la):&=& \nabla_x\Phi(x_{n},y_{n}(\la))-\nabla_x\Phi(x_{n-1},y_{n-1}), \smallskip \\
\Phi_n^y(\la):&=& \Phi(x_n,y_{n-1})+\langle \nabla_y \Phi (x_n, y_{n-1}),y_n(\la)-y_{n-1} \rangle-\Phi (x_n,y_{n}(\la)).
\ea\right\}
\ee
It follows from \cite[Theorem 10.9]{Beck2017book}
that
%from Fact \ref{fact_fb}  that
$\|y_n(\la)-y_{n-1}\|\leq r:=\|y_n(\tau)-y_{n-1}\|<+\infty$
for any $\la\in(0,\tau]$. This implies that
the curve $\{ y_n(\la): \, \la\in(0,\tau] \}$ lies in the closed ball $B[y_{n-1}; r]$.
Assume, by contradiction, that the linesearch procedure defined in Algorithm~\ref{algo1} fails to terminate at the $n$-th iteration.
Then, for all  $i=0,1,2,\ldots$ and  $\la =\tau \mu^i$, we have
\be\label{linesearch-cond1-v}
\frac{\la \tau_{n-1}}{\xi}\|\theta_n(\la)\|^2 +2\la \Phi_n^y(\la) &>&
\nu r_n + (1-\nu) c_n \ge \nu r_n,
\ee
where $r_n$ and $c_n$ are given in  (\ref{linesearch-cond1}).
Since $y_n(\la)\in B[y_{n-1}; r]$ for all $\la=\tau \mu^i$ with $i=0,1,2,\ldots$, it follows from Assumption \ref{asmp-0} (ii) and the inequality on the left-hand-side of \eqref{concavity-L} that
\be\label{ineq_la}
\|\theta_{n}(\la)\|^2\leq2L_{xx}^2\|x_{n}-x_{n-1}\|^2+ 2L_{xy}^2\|y_{n}(\la)-y_{n-1}\|^2~~\mbox{and}~~\Phi_n^y(\la)\leq \frac{L_{yy}}{2}\|y_{n}(\la)-y_{n-1}\|^2.
\ee
Combining the above two inequalities with \eqref{linesearch-cond1-v} and  $\la=\tau \mu^i$, we obtain
\ben
&& {\tau\mu^i \over \xi} \Big( 2 \tau_{n-1}L_{xx}^2 \|x_{n}-x_{n-1}\|^2 +
(2 \tau_{n-1}L_{xy}^2+\xi L_{yy}) \|y_{n}(\la)-y_{n-1}\|^2\Big)\\
 &>& \nu r_n = \nu\Big(\omega\delta_{n-1} \|x_{n}-x_{n-1}\|^2
+\frac{1}{\beta}\|y_n(\la)-y_{n-1}\|^2\Big),
\een
which implies $2\mu^i \tau  \tau_{n-1}L_{xx}^2 / \xi >\nu \omega\delta_{n-1}$
or $\mu^i (2\tau  \tau_{n-1}L_{xy}^2+\xi \tau L_{yy})/\xi > \nu/\beta$.
This is impossible since $\mu^i\rightarrow 0$ as $i\rightarrow \infty$,  which indicates that the linesearch procedure must terminate.

(ii) By \eqref{linesearch-cond1}, which, from part (i) of this Lemma,
 is always satisfied, and the definition of $b_n$ in \eqref{ab_n}, for all $n \ge 1$ we have
\[
b_n \ge (1-\nu) (r_n - c_n) =
(1-\nu) \left(r_n - \frac{\eta}{|\mathcal{I}_n|} \sum_{i \in \mathcal{I}_n} r_i \right),
\]
where $\mathcal{I}_n = \left\{n-1, n-2, \ldots, \max\{n-M, 1\} \right\}$ and $M \ge 1$ is an integer.
Hence, for all $k \ge M+1$ we have
\begin{eqnarray} \label{sum_bn}
\sum_{n=M+1}^k b_n &\ge &  (1-\nu) \left( \sum_{n=M+1}^k r_n
- \frac{\eta}{M}\sum_{n=M+1}^k \sum_{i=n-M}^{n-1} r_i \right) \nonumber \\
&=& (1-\nu)(1-\eta) \sum_{n=M+1}^k r_n +(1-\nu) \eta
 \left( \sum_{n=M+1}^k r_n
- \frac{1}{M}\sum_{n=M+1}^k \sum_{i=n-M}^{n-1} r_i \right) \nonumber \\
&\ge& (1-\nu)(1-\eta) \sum_{n=M+1}^k r_n - (1-\nu) \eta  \sum_{n=1}^M r_n.
\end{eqnarray}
Since  $J(x_n,y_n)\geq 0$ (see the definition of $J(\cdot)$ in \eqref{def:J}),
it follows from Lemma \ref{lem2} that $a_{n+1}\leq a_n-b_n$ for all $n\geq 1$.
So, it follows from $r_n \ge 0$ for all $n$, $\nu \in (0,1), \eta \in [0,1)$ and \eqref{sum_bn}  that
\[
a_{k+1} \le a_{M+1} - \sum_{n=M+1}^k b_n \le a_{M+1} + (1-\nu) \eta  \sum_{n=1}^M r_n
\]
for all $k \ge M+1$. Hence, the sequence $\{a_n\}$ is bounded.
Using the definition of $a_n$ in \eqref{ab_n}, we have
\be
\frac{\psi}{\psi-1}\|z_{n+2}-x^\star\|^2+\frac{1}{\beta}\|y_{n}-y^\star\|^2 \leq a_{n+1}.
\ee
Hence, the boundedness of $\{a_n\}$ implies that both sequences $\{z_n\}$ and $\{y_n\}$ are bounded.
Since $x_n=\frac{\psi}{\psi-1}z_{n+1}-\frac{1}{\psi-1}z_{n}$, which follows from the first relation in \eqref{x_updating}, the sequence $\{(x_n,y_n,z_n)\}$ is bounded.

(iii) First, from part (ii) of this lemma, the sequence $\{(x_n,y_n)\}$ is bounded.
It then follows from Assumption \ref{asmp-0} (ii) and the inequality on the left-hand-side of
\eqref{concavity-L} that
\be\label{result1}
\|\theta_{n}\|^2\leq2L_{xx}^2\|x_{n}-x_{n-1}\|^2+ 2L_{xy}^2\|y_{n}-y_{n-1}\|^2~~\mbox{and}~~\Phi_n^y\leq \frac{L_{yy}}{2}\|y_{n}-y_{n-1}\|^2.
\ee
Using \eqref{result1}, we see that the linesearch condition \eqref{linesearch-cond1} is satisfied provided that
\be\label{condition-line}
\nu\omega\delta_{n-1}- 2\tau_{n}\tau_{n-1}L_{xx}^2/\xi\geq 0
~~\mbox{and}~~
\nu/\beta -  2\tau_n\tau_{n-1}L_{xy}^2 / \xi - \tau_n L_{yy}\geq 0.
\ee
Recall that $\delta_{n-1} = \tau_{n-1}/\tau_{n-2}$ and $\tau_{j} \leq \tau_{\max}$ for all $j\geq 0$.
It is easy to show from the definition of $\underline{\tau}$ in \eqref{def:tau-underline} that
the conditions in \eqref{condition-line} are indeed satisfied when $\tau_n \leq \underline{\tau}$.

(iv) Let $\underline{\tau} >0$ be defined in \eqref{def:tau-underline} and note
that $\tau_0\geq\mu \underline{\tau}$ and $\tau_{\max} \geq \mu\underline{\tau}$. % from Remark \ref{rem_tau_max}.
%Since $\mu \in (0,1)$, we have $\tau_{\max} \geq \mu \underline{\tau}$ as well.
Assume that $\tau_{n-1} \geq \mu \underline{\tau}$.
To show that the sequence $\{\tau_n\}$ is strictly separated from $0$, we
only need to show that $\tau_n \geq \mu \underline{\tau}$ as well.
Recall that $\tau =\min\{\varphi\tau_{n-1}, \tau_{\max}\}$.
Since $\varphi > 1$, we have $\tau \geq\min\{\tau_{n-1}, \tau_{\max}\} \geq \mu \underline{\tau}$.
Recall that $\tau_n=\tau \mu^i$ for some nonnegative integer $i$.
If $i=0$, then $\tau_n=\tau \geq \mu \underline{\tau}$.
If $i>0$, then $\hat{\tau}_n := \tau \mu^{i-1}$ must violate \eqref{linesearch-cond1}.
It then follows from part (iii) of this lemma that $\hat{\tau}_n>\underline{\tau}$ must hold. Hence, $\tau_n=\mu \hat{\tau}_n>\mu \underline{\tau}$.
As such, we have proved that $\tau_n \geq \mu \underline{\tau} > 0$ for all $n$.
Finally, it is obvious that $\delta_n= \tau_n/ \tau_{n-1} \geq \mu \underline{\tau} / \tau_{\max} > 0$ for all $n$.
This completes the proof of this lemma.
\end{proof}

\begin{rem}\label{rem_tau0}
(1.) By the definitions of $a_n$ and $b_n$ in \eqref{ab_n}, we have that $a_n$ is nonnegative.
However, the linesearch rule (\ref{linesearch-cond1}) in Algorithm \ref{algo1} not necessarilly ensures the nonnegativity of $b_n$
unless setting the parameter $\eta =0$ and therefore, it does not ensure the sequence $\{a_n\}$ is
the monotonely descreasing.  Hence, we may regard (\ref{linesearch-cond1}) as a ``nonmonotone" line search procedure,
which is often used in nonlinear optimization algorithms to improve both robustness and efficiency \cite{HZ,GLL}.
(2.) The conditions $\tau_0\geq \mu \underline{\tau}$ and $\tau_{\max}\geq \mu \underline{\tau}$ in Lemma \ref{lem_bound} (iv) can be easily ensured.
 First, choose $y_{-1}\in \dom(f^*)$ arbitrarily in a small neighborhood of the starting point $y_0$ such that $\nabla_x \Phi (x_0,y_{-1})\neq \nabla_x\Phi (x_0,y_0)$ and then compute
$\varpi = \|y_{-1}-y_0\|^2/\|\nabla_x \Phi (x_0,y_{-1})-\nabla_x\Phi (x_0,y_0)\|^2 \geq 1/L_{xy}^2$.
Set $\tau_0=  \mu \xi\varpi / (2\beta)$ and $\tau_{\max}=\max(\chi, \tau_0)$ with $\chi>1$. Then, we have $\tau_{\max}\geq \tau_0\geq \frac{\mu\nu\xi}{2\beta L_{xy}^2} \geq \frac{\mu\nu\xi}{2\beta L_{xy}^2 \tau_{\max}}\geq \mu \underline{\tau}$ due to the definition of $\underline{\tau}$ in \eqref{def:tau-underline}, $\nu\in(0,1)$ and $\tau_{\max}>1$.
Hence, without loss of generality, in the following we assume $\tau_0$ and $\tau_{\max}$ are chosen such that $\tau_0\geq \mu \underline{\tau}$ and $\tau_{\max}\geq \mu \underline{\tau}$.
\end{rem}

\subsection{Convergence}

Based on Lemma \ref{lem_bound}, we now establish global pointwise convergence and $\cO(1/N)$ ergodic sublinear convergence rate of Algorithm \ref{algo1}.

\begin{thm}[Global pointwise convergence]\label{thm12}
The sequence $\{(x_{n}, y_n)\}$ generated by Algorithm \ref{algo1} converges to a solution of the saddle point problem (\ref{saddle_point}).
\end{thm}
 \begin{proof}
Again, it follows from  $J(x_n,y_n)\geq 0$ and
Lemma \ref{lem2} that $a_{n+1}\leq a_n-b_n$ for all $n\geq 1$.
So, by (\ref{sum_bn}) and $a_n \ge 0$ for all $n$, we have
\be\label{sum_rn}
(1-\nu)(1-\eta) \sum_{n=M+1}^{\infty} r_n \le a_{M+1}  + (1-\nu) \eta  \sum_{n=1}^M r_n.
\ee
From the proof of Lemma \ref{lem_bound} (iv), we have $\delta_n\geq\underline{\delta} := \mu \underline{\tau} / \tau_{\max}  > 0$ for all $n\geq 1$.
Then, by the definition of $r_n$ in (\ref{linesearch-cond1}), we have
$r_n = \omega\delta_{n-1}\|x_{n}-x_{n-1}\|^2+\frac{1}{\beta} \|y_n-y_{n-1}\|^2
\ge \omega \underline{\delta} \|x_{n}-x_{n-1}\|^2+\frac{1}{\beta} \|y_n-y_{n-1}\|^2$.
Hence, we have from $\nu \in (0,1), \eta \in [0,1)$ and (\ref{sum_rn}) that
\[
\sum_{n=1}^\infty r_n < \infty, \quad
\sum_{n=1}^\infty \|x_{n}-x_{n-1}\|^2 < \infty \quad \mbox{and} \quad
 \sum_{n=1}^\infty \|y_{n}-y_{{n}-1}\|^2 < \infty,
\]
which also implies
$\lim\limits_{n\rightarrow\infty} r_n =\lim\limits_{n\rightarrow\infty}\|x_{n}-x_{n-1}\| = \lim\limits_{n\rightarrow\infty}\|y_{n}-y_{n-1}\| =0$.
%
%\be\label{three-lim=0}
%\lim\limits_{n\rightarrow\infty}\|x_{n+1}-x_n\|
%= \lim\limits_{n\rightarrow\infty}\|y_{n+1}-y_n\|
%=0.
%\ee
%
Let $u_n :=z_n-x_n$.
It is easy to verify from $z_{n+1}=(1-1/\psi)x_n+ z_n/\psi$ that $\psi u_{n+1}- u_{n}= \psi(x_n-x_{n+1})$.
Dividing both sides of this equality by $\psi-1$ and using   \eqref{id2} gives
\ben
\frac{\psi}{\psi-1}\|u_{n+1}\|^2-\frac{1}{\psi-1}\|u_{n}\|^2+ \frac{\psi}{(\psi-1)^2}\|u_{n+1}-u_n\|^2=\frac{\psi^2}{(\psi-1)^2}\|x_n-x_{n+1}\|^2.
\een
This implies that $\|u_{n+1}\|^2\leq \frac{1}{\psi}\|u_{n}\|^2+\frac{\psi}{\psi-1}\|x_n-x_{n+1}\|^2$. It then follows from
$\sum_{n=1}^{\infty}\|x_{n}-x_{n+1}\|^2<\infty$  and Fact \ref{fact_ab} (ii)  that $\sum_{n=1}^{\infty}\|u_{n}\|^2<\infty$, and thus $\lim_{n\rightarrow\infty} u_n = \lim_{n\rightarrow\infty} (z_n-x_n) = 0$.
Since the sequence $\{(x_n,y_n, z_{n})\}$ is bounded from Lemma \ref{lem_bound} (ii),
there exist $(x^*,y^*)$ and a subsequence of $\{n_k: k\geq 1\}\subseteq\{n: n\geq 1\}$   such that
$\lim\limits_{k\rightarrow\infty} x_{n_k} = x^*$ and $\lim\limits_{k\rightarrow\infty} y_{n_k} = y^*$,
which implies that $\lim_{k \to \infty} x_{n_k + 1} = \lim_{k \to \infty} x_{n_k} = \lim_{k \to \infty} z_{n_k} = x^*$
and $\lim_{k \to \infty} y_{n_k} = \lim_{k \to \infty} y_{n_k-1} = y^*$.
Since $z_{n+1}=(1-1/\psi)x_n+ z_n/\psi$ for all $n\geq 1$, we have $\lim_{k \to \infty} z_{n_k+1} = x^*$ as well.

Similar to \eqref{temp_01} and \eqref{temp_001}, for any $(x,y)\in \dom(g)\times \dom(f^*)$, there hold
\begin{equation} \label{sub-optmal}
\left\{\ba{l}
\tau_{n_k} \big(g(x_{n_k+1})-g(x)\big)
\leq
\langle x_{n_k+1}-z_{n_k+1} + \tau_{n_k} \nabla_x\Phi( x_{n_k},y_{{n_k}}), ~ x-x_{n_k+1}\rangle, \smallskip \\
\tau_{n_k} \big(f^*(y_{n_k})-f^*(y)\big)
\leq
\big\langle {1\over\beta} (y_{n_k}-y_{n_k-1}) - \tau_{n_k} \nabla_y\Phi( x_{n_k},y_{{n_k}-1}) , ~y-y_{n_k}\big\rangle.
\ea
\right.
\end{equation}
Then, dividing $\tau_{n_k}$ from both sides of \eqref{sub-optmal},
taking into account  that both $g$ and $f$ are closed (and thus lower semicontinuous) and
letting $k\rightarrow\infty$, we obtain
\begin{equation} \label{thm12-1}
g(x^*)-g(x) \leq
\langle  \nabla_x\Phi( x^*,y^*), ~ x-x^*\rangle
\text{~~and~~}
f^*(y^*)-f^*(y) \leq
- \langle  \nabla_y\Phi( x^*,y^*), ~y-y^*\rangle.
\end{equation}
Since \eqref{thm12-1} holds for any $(x,y)\in \dom(g)\times \dom(f^*)$, we have $-\nabla_x\Phi( x^*,y^*)\in \partial g(x^*)$ and $\nabla_y\Phi( x^*,y^*) \in \partial f^*(y^*)$, which implies that $(x^*,y^*)$ is a solution of the saddle point problem \eqref{saddle_point}.

Recall that $J(\cdot)$ and $a_n$ are defined in \eqref{def:J} and \eqref{ab_n}, respectively, which depend on an arbitrarily fixed solution pair $({x^{\star}},{y^{\star}})$.
Since $(x^*,y^*)$ is also solution of the saddle point problem \eqref{saddle_point}, we can replace $({x^{\star}},{y^{\star}})$ by $(x^*,y^*)$ in the first place. As such, there holds
$\lim_{k\rightarrow\infty} a_{n_k} = 0$ since
$\lim_{k\rightarrow\infty} z_{n_k+1} = x^*$ and
$\lim_{k\rightarrow\infty} y_{n_k-1} = \lim_{k\rightarrow\infty} y_{n_k} = y^*$.
Exactly same as (\ref{sum_bn}), for all $\ell \ge n_k \ge M+1$ we can obtain
\[
\sum_{i=n_k}^\ell b_i \ge (1-\nu)(1-\eta) \sum_{i=n_k}^k r_i - (1-\nu) \eta \sum_{i=n_k-M}^{n_k-1} r_i.
\]
Hence, we can derive from $0 \le a_{n+1} \le a_n - b_n$ and $r_n \ge 0$ for all $n \ge 1$ that
\[
0 \le a_{\ell+1} \le a_{n_k} + (1-\nu) \eta \sum_{i=n_k-M}^{n_k-1} r_i, \quad \forall \ell \ge n_k \ge M+1.
\]
Then, we have from $\lim_{k\rightarrow\infty} a_{n_k} = 0$
and $\lim_{n\rightarrow\infty} r_n = 0$ that
$\lim_{n\rightarrow\infty} a_n = 0$ as well.
As a result, due to the definition of $a_n$ in \eqref{ab_n}, there holds $\lim_{n\rightarrow\infty}(z_{n},y_{n}) = (x^*,y^*)$.
Again, by the first relation in \eqref{x_updating}, we have $\lim_{n\rightarrow\infty}x_n = x^*$.
This completes the proof.
\end{proof}

We next establish ergodic sublinear convergence rate of Algorithm \ref{algo1} using the primal-dual gap function $J(\cdot)$ defined in \eqref{def:J}.

\begin{thm}[Sublinear convergence rate]\label{thm22}
There exists a constant $C_1> 0$ such that for any $N\geq 1$ there holds
$J({\hat x}_N,{\hat y}_N) \leq C_1/N$,
%
%\ben
%\cL({\hat x}_N,y^\star)- \cL(x^\star, {\hat y}_N)\leq   {C_1 \over N}.
%\een
where ${\hat x}_N$ and ${\hat y}_N$ are defined as
\ben%\label{definition_XY0}
{\hat x}_N =\frac{1}{s_N}\sum_{n=1}^N\tau_n x_n
~~\mbox{and}~~
{\hat y}_N=\frac{1}{s_N} \sum_{n=1}^N\tau_n y_n
~~\mbox{with}~~
s_N=\sum_{n=1}^N\tau_n.
\een
\end{thm}
\begin{proof}
Recall from Lemma \ref{lem2} that
$2 \tau_n J(x_{n},y_n) \leq a_{n} - a_{n+1} - b_n$ for all $n\geq 1$,
a sum of which over $n = 1, \ldots, N$ yields
$ 2\sum\nolimits_{n=1}^N  \tau_n   J(x_{n},y_n)  \leq a_1 - a_{N+1}
- \sum_{n=1}^N b_n$. Then, it follows from (\ref{sum_bn}) that
$ 2\sum\nolimits_{n=1}^N  \tau_n   J(x_{n},y_n)  \leq a_1 + \tilde{C} $,
where $ \tilde{C} = \sum_{n=1}^M |b_n|
+ (1-\nu) \eta \sum_{n=1}^M r_n$.
Since $J(x, y)$ is jointly convex in $(x,y)$, it follows from the definitions of ${\hat x}_N$, ${\hat y}_N$ and Jensen's inequality that
$J({\hat x}_N, {\hat y}_N) \leq {1 \over s_N} \sum_{n=1}^N \tau_n  J(x_{n}, y_n))$.
%
%\be
%J({\hat x}_N, {\hat y}_N) \leq {1 \over s_N} \sum_{n=1}^N \tau_n  J(x_{n}, y_n)).
%\label{thm2-20}
%\ee
Combining %\eqref{thm2-10} and \eqref{thm2-20},
the two inequalities just derived, we obtain
$J({\hat x}_N, {\hat y}_N)\leq (a_1+\tilde{C}) / (2s_N)$.
%
%\ben%\label{ineq_them3.2}
%J({\hat x}_N, {\hat y}_N)\leq a_1 / (2s_N).
%\een
%
By  Lemma~\ref{lem_bound} (iv), it holds $\tau_n\geq\mu\underline{\tau}>0$ and then $s_N = \sum_{n=1}^N\tau_n \geq \mu\underline{\tau} N$. By defining $C_1= (a_1+\tilde{C})/(2\mu\underline{\tau})> 0$, the proof is completed.
\end{proof}

%For any fixed $(x^\star,y^\star)\in \Omega$, the primal-dual gap function $J(x,y):={\cal L}(x,y^\star) - {\cal L}(x^\star,y)$ was frequently used in the literature \cite{Chambolle2011A,Chambolle2016ergodic,EYNS2021} to measure convergence. It is easy to show that $J(x,y)\geq 0$ for any $(x,y)$ and $J(x,y)=0$ if  $(x,y)$ is any saddle point of  ${\cal L}(\cdot)$. However, $J(x,y)=0$ does not imply that $(x,y)$ is a saddle point, see \cite{Chambolle2011A}, and this is not a good measure of optimality.

Even though the primal-dual gap function $J(\cdot)$ is frequently adopted in the literature to measure sublinear convergence rate, see, e.g., \cite{Chambolle2011A,Chambolle2016ergodic,EYNS2021}, it has a flaw that it could vanish at nonstationary points.
Next, we consider a special yet important case, where the coupling term is linear in one of the variables.
We reformulate the problem as a constrained optimization problem and establish convergence rate results for Algorithm \ref{algo1} in terms of the conventional measure of function value residual and constraint violation of the reformulated problem. Moreover, when $g$ is strongly convex, we propose an accelerated algorithm and establish the faster $\cO (1/N^2)$ ergodic convergence rate.

\section{Nonlinear Compositional Convex Optimization Problem}\label{sec:special problem}
Let $g$ and $f$ be the same functions as in \eqref{saddle_point},
$H: \dom(g) \rightarrow \dom(f)$ be nonlinear and continuously differentiable,  and $h: \R^q\rightarrow\R$ be convex and $L_h$-smooth with some constant $L_h>0$.
%(again, $L_h$ is not needed in the algorithm).

In this section, we consider the following nonlinear compositional convex optimization problem
\be\label{primal}
\min_{x\in \R^q}\{P(x):=g(x)+h(x)+f(H(x))\}.
\ee
Using the fact that $f^{**}=f$,  \eqref{primal} can be represented as
\be\label{saddle-point-special}
\min_{x\in \R^q}\max_{y\in \R^p} \big\{\cL(x, y):= g(x)+\Phi(x,y)-f^*(y)\big\} \text{~~with~~} \Phi(x,y) := h(x)+\langle H(x),y\rangle,
\ee
which is obviously a special case of \eqref{saddle_point}. Assume that Assumptions \ref{asmp-1}-\ref{asmp-2} hold for \eqref{saddle-point-special}. Moreover, we assume that, for any $y\in \dom(f^*)$, $\langle H(x),y\rangle$ is convex in $x$. Under this assumption, it follows that
$f(H(x))=\max_{y\in \R^p}\{ \langle H(x),y\rangle - f^* (y)\} $ is convex  in $x$ as well.

Since $\Phi(x,y)$ is linear in $y$, the Lipschitz constant $L_{yy}$ defined in Assumption \ref{asmp-0} (ii)  can set to be $0$.
Note that problem \eqref{saddle-point-special} is an extension of the case studied in \cite{Zhu23On}, where $h\equiv 0$ was considered.
Define $l(y) := \max_{x\in \R^q}\{\langle H(x),y\rangle -g(x)-h(x)\}$ for $y\in\R^p$. Then, the
dual problem of \eqref{primal} is given by
\be\label{dual}
\max_{y\in \R^p}\left\{D(y) := -l(-y)-f^*(y)\right\}.
\ee
On the other hand, by introducing an auxiliary variable $w\in\R^p$, the primal problem \eqref{primal} can be equivalently represented as the following nonlinear equality constrained problem
\be\label{two-block}
\min\nolimits_{x\in\R^q, \, w\in\R^p}  \{F(x,w) := g(x)+h(x)+f(w)  \text{~~s.t.~~}   H(x)-w=0\}.
\ee
Let $y\in \R^p$ be the Lagrange multiplier
and denote the Lagrange function of \eqref{two-block} by
\begin{align}\label{def:PhiL}
      \widetilde{\cL}(x,w,y)      :=   F(x,w) + \langle y, H(x) - w\rangle, \;\; (x,w,y) \in\dom(g)\times\dom(f)\times\R^p.
\end{align}
Similar to Assumption \ref{asmp-1} for \eqref{saddle_point}, we make the following assumption.
\begin{asmp}\label{asmp-linear-y}
Assume that problem \eqref{saddle-point-special} has at least one saddle point.
%, i.e., $\widetilde{\Omega}$ defined in \eqref{def:Omegas}
\end{asmp}
Under Assumption \ref{asmp-linear-y}, strong duality holds between \eqref{primal} and \eqref{dual}, and there exists $(x^\star, y^\star)\in\dom(g)\times\dom(f^*)$ such that $P(x^\star)=D(y^\star) =\cL(x^\star, y^\star)$. %, see, e.g.,\cite{Rockafellar2004Variational}.
As such, $x^\star$ and $y^\star$ are optimal for the primal and dual problems, respectively, and
$({x^{\star}},{y^{\star}})$ is a solution of the primal-dual problem \eqref{saddle-point-special}.
Let $w^\star = H(x^\star)$. Then, $(x^\star, w^\star)$ is a solution of \eqref{two-block} and
%
%
%It follows from \cite[Corollaries 28.2.2 and 28.3.1]{Rockafellar1970Convex} that $({x^{\star}}, {w^{\star}}) \in \R^q\times \R^p$ is a solution of \eqref{two-block} if and only if there exists an optimal solution ${y^{\star}}\in \R^p$  to the dual problem \eqref{dual} such that
%
$({x^{\star}}, {w^{\star}}, {y^{\star}})$ is a saddle point of $\widetilde{\cL}(\cdot)$, i.e.,
\ben%\label{saddle-L}
\widetilde{\cL}({x^{\star}},{w^{\star}},y) \leq \widetilde{\cL}({x^{\star}},{w^{\star}},{y^{\star}}) \leq\widetilde{\cL}(x,w,{y^{\star}}) \text{~~for all~~} (x, w, y) \in \dom(g)\times \dom(f) \times \R^p.
\een
Denote the set of saddle points of $\widetilde{\cL}(\cdot)$ by $\widetilde{\Omega}$, which is nonempty under Assumption \ref{asmp-linear-y}, i.e.,
%which is nonempty under our assumption and is given by
\be\label{def:Omegas}
\widetilde{\Omega} = \left\{\ba{r}({x^{\star}}, {w^{\star}}, {y^{\star}}) \in \dom(g)\times \\
\dom(f)\times \dom(f^*)\ea ~\Big|~
\ba{l}
-H'({x^{\star}})^\top y^\star-\nabla h(x^\star) \in \partial g({x^{\star}}) \\
{y^{\star}} \in \partial f({w^{\star}}), \,\, H({x^{\star}}) = {w^{\star}}
\ea\right\}\neq \emptyset.
\ee
By using the Moreau's decomposition $y = \prox_{f/\sigma}(y) +  {1\over\sigma}  \prox_{\sigma f^*}( \sigma y)$, which holds for any $\sigma > 0$ and $y\in \R^p$, $y_{n}$ defined in \eqref{y_updating} can be split as
\be\label{y1_updating}
y_{n} = y_{n-1}+ \beta\tau_{n}(H(x_{n})-w_{n})
 \text{~~with~~} w_{n} := \prox_{ f/(\beta\tau_{n}) }\big( y_{n-1}/(\beta\tau_{n}) + H(x_{n})\big).
\ee
Moreover, recall that in this case $L_{yy} =0$ and $\Phi_n^y =0$.
Hence, the linesearch procedure using condition (\ref{linesearch-cond1}) can be much simplified.
In the rest of this section, without mentioning repeatedly, we always fix arbitrary a primal-dual solution triplet $({x^{\star}}, {w^{\star}},{y^{\star}})\in \widetilde{\Omega}$, let $\{(z_n,x_n,y_n)\}$ be the sequence generated by Algorithm \ref{algo1} when applied to the special case \eqref{saddle-point-special}, and $\{w_n\}$ be given by \eqref{y1_updating}.
Furthermore, we define
\be\label{def:Js}
  \widetilde{J}(x,w,y)
:= \widetilde{\cL}(x,w,y)  - \widetilde{\cL}({x^{\star}},{w^{\star}}, y)
= F(x,w)+\langle y, H(x)-w\rangle  -  F({x^{\star}},{w^{\star}}),
\ee
where $F(x,w)$ and $\widetilde{\cL}(x,w,y)$ are given in \eqref{two-block} and \eqref{def:PhiL}, respectively.

\subsection{Convergence Rate Analysis of PDAc-L for \eqref{saddle-point-special}} %rate measured by constrained optimization problem \eqref{two-block}}
In this section, we establish convergence rate results of PDAc-L (i.e., Algorithm \ref{algo1}) when applied to the particular case \eqref{saddle-point-special}. These convergence rate results are measured in terms of function value residual and constraint violation of the equivalent problem \eqref{two-block}.

\begin{lem}\label{lem1-s}
Let $\theta_n$ be defined in \eqref{def:q-La}.
Then, for any $y\in \bR^p$, there holds
\be
\tau_n \widetilde{J}(x_n,w_n,y)
&\leq&\langle x_{n+1}-z_{n+1}, x^\star-x_{n+1}\rangle+ \frac{1}{\beta} \langle  y_{n}-y_{n-1}, y-y_{n}\rangle \nonumber\\
&& + \psi\delta_{n} \langle x_{n}-z_{n+1}, x_{n+1}- x_n \rangle + \tau_n\langle \theta_n, x_{n}-x_{n+1}\rangle. \label{lem1s:ineq}
\ee
\end{lem}
\begin{proof}
It follows from \eqref{temp_01}, \eqref{temp_02}, (\ref{y1_updating}) and Fact  \ref{fact_proj} that
\be\label{temps_001}
\left\{\ba{rcl}\tau_{n}\big(g(x_{n+1})-g(x^{\star})\big)
&\leq & \big\langle x_{n+1}-z_{n+1} + \tau_{n}\nabla_x \Phi(x_{n},y_{n}), ~ x^{\star}-x_{n+1}\big\rangle, \smallskip \\
\tau_{n}\big(g(x_{n})-g(x_{n+1})\big)&\leq&\left\langle \psi\delta_n(x_{n}-z_{n+1}) + \tau_{n}\nabla_x \Phi(x_{n-1},y_{n-1}), ~x_{n+1}-x_{n}\right\rangle, \smallskip \\
\tau_{n}\big(f(w_{n})-f({w^{\star}})\big) &\leq & - \tau_{n}\big\langle y_{n-1}+ \beta\tau_{n}(H(x_{n})-w_{n}), ~{w^{\star}}-w_{n} \big\rangle.
\ea\right.
\ee
By taking a sum to (\ref{temps_001}),
followed by adding $\tau_n \big(h(x_n) + \langle y, H(x_n)-w_n\rangle  - h(x^\star) \big)$ to both sides, using the notation $\Phi(\cdot)$, $F(\cdot)$ and $\widetilde{J}(\cdot)$ defined in \eqref{saddle-point-special}, \eqref{two-block} and \eqref{def:Js}, respectively, and taking into account the relations $y_{n}=y_{n-1} + \beta\tau_{n}(H(x_{n})-w_{n})$ and $H({x^{\star}})={w^{\star}}$, we obtain after elementary calculations that
\begin{align}\label{temps_03}
& \tau_{n}  \widetilde{J}(x_n,w_n,y)
= \tau_{n}\big(F(x_n,w_n)   +\langle y, H(x_n)-w_n\rangle - F(x^\star,w^\star) \big) \\
\leq \, & \langle x_{n+1}-z_{n+1},~  x^\star-x_{n+1}\rangle+\frac{1}{\beta} \langle y_{n}-y_{n-1}, y-y_{n}\rangle
+ \psi\delta_n\left\langle x_{n}-z_{n+1}, ~ x_{n+1}- x_{n}\right\rangle +\tau_{n}\widetilde{\cG}_n, \nonumber
\end{align}
where $\widetilde{\cG}_n$ is defined by
\ben
\widetilde{\cG}_n
%&:=&\langle \nabla_x \Phi(x_{n},y_{n}), ~ x^\star-x_{n+1}\rangle  + \big\langle \nabla_x \Phi(x_{n-1},y_{n-1}), ~x_{n+1}-x_{n}\big\rangle\\
%&& + \langle y_n, H(x_n)-H(x^\star)\rangle+ h(x_n)-h(x^\star) \\
&:=&\langle \nabla_x \Phi(x_{n},y_{n}), \,x^\star-x_{n+1}\rangle  + \big\langle \nabla_x \Phi(x_{n-1},y_{n-1}),\, x_{n+1}-x_{n}\big\rangle
 +\Phi (x_n,y_n)- \Phi(x^\star,y_n) \\
&=&\langle \theta_n, x_{n}-x_{n+1}\rangle + \langle\nabla_x \Phi(x_{n},y_{n}), ~ x^\star-x_{n}\rangle+ \Phi (x_n,y_n)- \Phi(x^\star,y_n)\\
&\leq &\langle \theta_n, x_{n}-x_{n+1}\rangle,
\een
where the second equality follows from the definition of $\theta_n$ in \eqref{def:q-La} and inequality is due to the convexity of $\Phi(\cdot)$ in $x$.
This together with \eqref{temps_03} implies \eqref{lem1s:ineq} immediately.
\end{proof}

Recall that $b_{n}$  and $\omega := \omega(\xi,\varphi)$ are defined in \eqref{ab_n} and \eqref{def:zeta}, respectively.
Similar to Lemma \ref{lem2}, we have the following result.
\begin{lem}\label{lem2s}
%Let $\widetilde{a}_n (y)$ be defined in \eqref{def:an-tilde}
For any $y\in \bR^p$ and $n\geq 1$,  there holds
\ben%\label{results-abs}
\widetilde{a}_{n+1}(y)+2 \tau_n \widetilde{J}(x_n,w_n,y)\leq \widetilde{a}_{n}(y)-b_n, \quad \forall y\in\bR^p,
\een
where
$\widetilde{a}_n (y) := \frac{\psi}{\psi-1}\|z_{n+1}-x^\star\|^2+\frac{1}{\beta}\|y_{n-1}-y \|^2 +\omega\delta_{n-1}\|x_n-x_{n-1}\|^2$.
\end{lem}
\begin{proof}
  The proof is completely analogous to  that of Lemma \ref{lem2} and is therefore omitted.
\end{proof}

\begin{thm}[Sublinear convergence rate]\label{thm:scr}
There exists a constant $C_2> 0$ such that for any $N\geq 1$ there hold
$|F({\hat x}_N,{\hat w}_N)-F({x^{\star}},{w^{\star}})| \leq  C_2 / N$
and $\|H({\hat x}_N)-{\hat w}_N\| \leq   (2C_2/c)/N$,
where $c>0$ is a constant satisfying  $c \geq 2\|{y^{\star}}\|$, and ${\hat x}_N$ and ${\hat w}_N$ are defined as
\ben%\label{definition_XY}
{\hat x}_N =\frac{1}{s_N}\sum_{n=1}^N\tau_n x_n
~~\mbox{and}~~
{\hat w}_N=\frac{1}{s_N} \sum_{n=1}^N\tau_n w_n
~~\mbox{with}~~
s_N=\sum_{n=1}^N\tau_n.
\een
\end{thm}
\begin{proof}
Let $y\in\bR^p$ be arbitrarily fixed.
Recall that $b_n$ and $\widetilde{a}_n(y)$ are defined in \eqref{ab_n}  and Lemma \ref{lem2s}, respectively.
 By definition, the sequence  $\{\widetilde{a}_n(y)\}$ is nonnegative.
Lemma \ref{lem2s} implies that
$2 \tau_n \widetilde{J}(x_n,w_n,y)\leq \widetilde{a}_{n}(y) - \widetilde{a}_{n+1}(y) - b_n$, a sum of which over
$n = 1, \ldots, N$ yields
\be
 2\sum_{n=1}^N  \tau_n  \widetilde{J}(x_n,w_n,y) \leq \widetilde{a}_1(y) - \widetilde{a}_{N+1}(y)
 - \sum_{n=1}^N b_n \leq \widetilde{a}_1(y) - \sum_{n=1}^N b_n
 \le \widetilde{a}_1(y) + \tilde{C},
\label{thm2-1}
\ee
where the last inequality follows from (\ref{sum_bn}) and
 $ \tilde{C} = \sum_{n=1}^M |b_n| + (1-\nu) \eta \sum_{n=1}^M r_n$.
Since $\widetilde{J}(x,w,y)$ is convex in $(x,w)$, it follows from the definitions of ${\hat x}_N$ and ${\hat w}_N$ and Jensen's inequality that
\be
\widetilde{J}\big({\hat x}_N,{\hat w}_N,y\big) \leq {1 \over s_N} \sum_{n=1}^N \tau_n  \widetilde{J}(x_n,w_n,y).
\label{thm2-2}
\ee
Combining \eqref{thm2-1}, \eqref{thm2-2} and the definition  of $\widetilde{J}(\cdot)$ in \eqref{def:Js}, we obtain
\ben%\label{ineq_them3.2}
F({\hat x}_N,{\hat w}_N)+\langle y, H({\hat x}_N)-{\hat w}_N\rangle-F({x^{\star}},{w^{\star}})\leq
(\widetilde{a}_1(y) + \tilde{C}) / (2s_N).
\een
By Lemma~\ref{lem_bound} (iv), we have $s_N = \sum_{n=1}^N\tau_n \geq \mu\underline{\tau} N$. By taking the maximum on both sides of
the above inequality %\eqref{ineq_them3.2}
over $\|y\|\leq c$ and defining $\bar{C}= \sup_{y} \{\widetilde{a}_1(y): \,  \|y\|\leq c\} / (2\mu\underline{\tau}) > 0$ and $C_2 = \bar{C}+ \tilde{C}$, we obtain
\be\label{key-result}
F({\hat x}_N,{\hat w}_N)+c\|H( {\hat x}_N)-{\hat w}_N\|-F({x^{\star}},{w^{\star}})\leq C_2/N,
\ee
which implies $F({\hat x}_N,{\hat w}_N)-F({x^{\star}},{w^{\star}})\leq C_2/N$.
Recall that $\widetilde{{\cal L}}(\cdot)$ is defined in \eqref{def:PhiL}.
It follows from $\widetilde{{\cal L}}({x^{\star}},{w^{\star}},{y^{\star}}) \leq \widetilde{{\cal L}}({\hat x}_N,{\hat w}_N,{y^{\star}})$,
$H({x^{\star}}) = {w^{\star}}$ and $\|{y^{\star}}\| \leq c/2$ that
\begin{eqnarray}\label{jy-5}
F({x^{\star}},{w^{\star}})-F({\hat x}_N,{\hat w}_N)&\leq&\langle {y^{\star}}, H({\hat x}_N)-{\hat w}_N\rangle
 \leq (c/2)\|H({\hat x}_N)-{\hat w}_N\|,
\end{eqnarray}
which together with \eqref{key-result} implies
\begin{eqnarray*}
c\|H({\hat x}_N)-{\hat w}_N\|&\leq& F({x^{\star}},{w^{\star}}) - F({\hat x}_N,{\hat w}_N)+  C_2 / N
\leq  (c/2)\|H({\hat x}_N)-{\hat w}_N\|+  C_2 / N.
\end{eqnarray*}
As a result, we derive $\|H({\hat x}_N)-{\hat w}_N\|\leq (2C_2/c)/N$. It then follows from \eqref{jy-5} that
$F({x^{\star}},{w^{\star}})-F({\hat x}_N,{\hat w}_N)\leq C_2/N$, and thus
$|F({\hat x}_N,{\hat w}_N)-F({x^{\star}},{w^{\star}})| \leq C_2/N$.
The proof is completed.
\end{proof}

\subsection{Acceleration When $g$ is Strongly Convex}\label{sec-acc-L}
In this section, we assume that $g$ is strongly convex with modulus $\gamma > 0$, i.e.,
\ben%\label{convex-ineq}
g(x) \geq g(\tilde{x}) + \langle \zeta, x - \tilde{x}\rangle +
\frac{\gamma}{2} \| x -\tilde{x}\|^2,~~ \forall x, \tilde{x} \in \dom (g), ~ \zeta \in \partial g(\tilde{x}).
\een
Under this assumption, we are able to present an accelerated algorithm by adaptively tuning the parameter $\beta$ in Algorithm
\ref{algo1}.
Recall that $\psi\in(1,1+\sqrt{3})$ and $\omega := \omega(\xi, \varphi)$ is defined in \eqref{def:zeta}.
For convenience, we define
%Firstly, for any $\psi\in(1,1+\sqrt{3})$, notice \eqref{def:zeta} and define
\be
\Theta_{sc} &:=& \{(\psi, \xi,\varphi)~|~
\xi>0,~~\psi>\varphi>1~~\mbox{and}~~\omega(\xi,\varphi)>0\}.\label{def:Theta-sc}
\ee
Recall that for problem \eqref{saddle-point-special}  we have $\Phi(x,y) = h(x)+\langle H(x),y\rangle$ and
$\theta_n$ is defined in  \eqref{def:q-La}.
The proposed accelerated algorithm is summarized below.

\vskip5mm
\hrule\vskip2mm
\begin{algo}
[Accelerated PDAc-L (aPDAc-L) when $g$ is $\gamma$-strongly convex]\label{algo1-sc}
{~}\vskip 1pt {\rm
\begin{description}
\item[{\em Step 0.}] Choose $\psi\in(1,1+\sqrt{3})$ and $(\psi, \xi,\varphi)\in\Theta_{sc}$, $\tau_{\max}>0$ and $\mu\in(0,1)$. Choose $x_0 \in \dom(g),$ $y_0\in \dom(f^*)$, $\beta_0>0$ and $\tau_0 \in (0, \tau_{\max}]$.
    Set $z_{0}=x_0$, $\omega:= \omega(\xi,\varphi)$, $\delta_0=1$ and $n=1$.

\item[{\em Step 1.}] Compute $z_n$ and $x_n$ according to \eqref{x_updating} and
define
\be
 \rho_n := \frac{\psi-\varphi}{\psi+\varphi\gamma\tau_{n-1}},  ~~
 \beta_{n} := (1+ \gamma\rho_n \tau_{n-1})\beta_{n-1}   \text{~~and~~}
 \kappa_{n-1} := \omega \delta_{n-1}+ \gamma\tau_{n-1}.
 \label{beta-gstrong}
\ee
\item[{\em Step 2.}] Set $\tau = \min\{\varphi\tau_{n-1}, \tau_{\max}\}$ and compute
\ben
w_{n} = \prox_{ f/(\beta_n\tau_{n}) }\big( y_{n-1}/(\beta_n\tau_{n}) + H(x_{n})\big)
\text{~~and~~}
y_{n} = y_{n-1}+ \beta_n\tau_{n}(H(x_{n})-w_{n}),
\een
where $\tau_n = \tau \mu^i$ and $i$ is the smallest nonnegative integer such that
\be\label{linesearch-cond-sc}
\frac{\tau_n\tau_{n-1}}{\xi}\|\theta_n\|^2 &\leq& \frac{\kappa_{n-1}\beta_{n-1}}{\beta_n} \|x_{n}-x_{n-1}\|^2
+\frac{1}{\beta_n}\|y_n-y_{n-1}\|^2.
\ee
%Otherwise, set $\tau_n\leftarrow \mu\tau_n$ and go to (a).
\item[{\em Step 3.}] Set $\delta_n= \tau_n/\tau_{n-1}$, $n \leftarrow n+1$ and go to Step 1.
  \end{description}
}
\end{algo}
\vskip1mm\hrule\vskip5mm

Recall that $\widetilde{J}(\cdot)$ is defined in   \eqref{def:Js}.
Similarly to Lemma \ref{lem1-s}, we can derive the following stronger result due to the strong convexity of $g$. %obtain following lemma together with these stronger inequalities.

\begin{lem}\label{lem1-sc}
For any $y\in\R^p$, there holds
\be
\tau_n \widetilde{J}(x_n,w_n,y)
&\leq&\langle x_{n+1}-z_{n+1}, x^\star-x_{n+1}\rangle+ \frac{1}{\beta_n} \langle  y_{n}-y_{n-1}, y-y_{n}\rangle
+ \tau_n\langle \theta_n, x_{n}-x_{n+1}\rangle
\nonumber\\
&& + \psi\delta_{n} \langle x_{n}-z_{n+1}, x_{n+1}- x_n \rangle  - \frac{\gamma\tau_n}{2} \big( \|x_{n+1}-{x^{\star}}\|^2 +  \|x_{n+1}-x_n\|^2\big).
\label{lem1-sc:ineq}
\ee
\end{lem}
\begin{proof}
As $g$ is $\gamma$ strongly convex, we can strengthen the first two inequalities in \eqref{temps_001} by
\ben
 \left\{\ba{rcl}\tau_{n}\big(g(x_{n+1})-g(x^{\star})\big)
&\leq & \big\langle x_{n+1}-z_{n+1} + \tau_{n}\nabla_x \Phi(x_{n},y_{n}), ~ x^{\star}-x_{n+1}\big\rangle- \frac{\gamma\tau_n}{2} \|x_{n+1}-{x^{\star}}\|^2, \smallskip \\
\tau_{n}\big(g(x_{n})-g(x_{n+1})\big)&\leq&\left\langle \psi\delta_n(x_{n}-z_{n+1}) + \tau_{n}\nabla_x \Phi(x_{n-1},y_{n-1}), ~x_{n+1}-x_{n}\right\rangle- \frac{\gamma\tau_n}{2} \|x_{n+1}-x_n\|^2.
\ea\right.
\een
The remaining proof is entirely analogous to Lemma \ref{lem1-s} and is thus omitted for simplicity.
\end{proof}

Similar to Lemma \ref{lem2s}, we can establish the following result.
\begin{lem}
For any $y\in \bR^p$ and $n\geq 1$,  there holds
\be\label{key-ineq2}
\beta_{n+1}A_{n+1}(y)+\beta_{n}\tau_{n} \widetilde{J}(x_n,w_n,y)
\leq \beta_{n}A_{n}(y) - \beta_{n}B_n,
\ee
where
\be\label{def:AnyBn}
\left\{
\ba{rcl}
A_{n}(y) &:=&\frac{\psi+\gamma\tau_n}{2(\psi-1)}\|z_{n+1}-{x^{\star}}\|^2+ \frac{1}{2\beta_{n}}\|y_{n-1}-y\|^2+\frac{\kappa_{n-1}\beta_{n-1}}{\beta_n}\|x_{n}-x_{n-1}\|^2, \smallskip \\
B_n &:=&
\frac{1}{\beta_n}\|y_n-y_{n-1}\|^2 +\frac{\kappa_{n-1}\beta_{n-1}}{\beta_n}\|x_{n}-x_{n-1}\|^2 -\frac{\tau_n\tau_{n-1}}{\xi}\|\theta_n\|^2.
\ea
\right.
\ee
\end{lem}
\begin{proof}
By applying (\ref{id}) to the first, second, and fourth inner products in \eqref{lem1-sc:ineq}, and then reorganizing the terms, we obtain
\ben
(1+\gamma \tau_n)\|x_{n+1}&-&x^\star\|^2+\frac{1}{\beta_n}\|y_{n}-y\|^2+ 2 \widetilde{J}(x_n,w_n,y)\nonumber\\
&\leq&\|z_{n+1}-x^\star\|^2+\frac{1}{\beta_n}\|y_{n-1}-y\|^2 -\frac{1}{\beta_n}\|y_{n}-y_{n-1}\|^2 + 2\tau_{n}\langle \theta_n, x_{n}-x_{n+1}\rangle \nonumber\\
&&-\psi\delta_n\|z_{n+1}-x_n\|^2-(1-\psi\delta_n)\|x_{n+1}-z_{n+1}\|^2
-(\psi\delta_n+\gamma \tau_n)\|x_{n+1}-x_n\|^2.
\een
Then, by plugging \eqref{x-to-z} into the above inequality and combining the terms, we obtain
\be\label{temp-ineq-sc}
&&\frac{\psi(1+\gamma\tau_n)}{\psi-1}\|z_{n+2}-x^\star\|^2+\frac{1}{\beta_n}\|y_{n}-y\|^2+ 2 \widetilde{J}(x_n,w_n,y)\nonumber\\
&\leq&\frac{\psi+\gamma\tau_n}{\psi-1}\|z_{n+1}-x^\star\|^2+\frac{1}{\beta_n}\|y_{n-1}-y\|^2 -\frac{1}{\beta_n}\|y_{n}-y_{n-1}\|^2 + 2\tau_{n}\langle \theta_n, x_{n}-x_{n+1}\rangle  \\
&&-\psi\delta_n\|z_{n+1}-x_n\|^2-\Big(1+\frac{1+\gamma \tau_n}{\psi}-\psi\delta_n\Big)\|x_{n+1}-z_{n+1}\|^2
-(\psi\delta_n+\gamma \tau_n)\|x_{n+1}-x_n\|^2. \nonumber
\ee
On the other hand, Fact \ref{fact_uv} with $u=\psi\delta_n$, $v=1+\frac{1+\gamma \tau_n}{\psi}-\psi\delta_n$,
% so that $u+v = 1+\frac{1+\gamma \tau_n}{\psi} >0$,
$a=\|z_{n+1}-x_n\|$, $b=\|x_{n+1}-z_{n+1}\|$ and  $\|x_{n+1}-x_{n}\|\leq  a+b$ imply
\ben\label{ineq-improve-sc}
\psi\delta_n\Big(1-\frac{\psi^2\delta_n}{1+\psi+\gamma \tau_n}\Big) \|x_{n+1}-x_n\|^2 \leq
\psi\delta_n \|z_{n+1}-x_n\|^2+ \Big(1+\frac{1+\gamma \tau_n}{\psi}-\psi\delta_n\Big) \|x_{n+1}-z_{n+1}\|^2.
\een
Moreover, since we have $2\langle \theta_n, x_n-x_{n+1}\rangle \leq\frac{\xi}{\tau_{n-1}}\|x_n-x_{n+1}\|^2+ \frac{\tau_{n-1}}{\xi}\|\theta_n\|^2$, substituting this into \eqref{temp-ineq-sc} and utilizing the above inequality, we deduce
\be\label{ineq_rate2-sc}
 \frac{\psi(1+\gamma\tau_n)}{\psi-1}\|z_{n+2}&-&x^\star\|^2 +\frac{1}{\beta_n}\|y_{n}-y\|^2+ 2 \tau_n \widetilde{J}(x_n,w_n,y)\nonumber\\
&\leq&\frac{\psi+\gamma\tau_n}{\psi-1}\|z_{n+1}-x^\star\|^2+\frac{1}{\beta_n}\|y_{n-1}-y\|^2 + \frac{\tau_n\tau_{n-1}}{\xi}\|\theta_n\|^2  \nonumber \\
&&-\frac{1}{\beta_n}\|y_n-y_{n-1}\|^2 -\Big(\delta_n\big(2\psi-\xi-\frac{\psi^3\delta_n}{1+\psi+\gamma \tau_n}\big)+ \gamma\tau_n\Big)\|x_{n+1}-x_n\|^2\nonumber\\
&\leq&\frac{\psi+\gamma\tau_n}{\psi-1}\|z_{n+1}-x^\star\|^2+\frac{1}{\beta_n}\|y_{n-1}-y\|^2 + \frac{\tau_n\tau_{n-1}}{\xi}\|\theta_n\|^2  \nonumber \\
&&-\frac{1}{\beta_n}\|y_n-y_{n-1}\|^2 -\kappa_n\|x_{n+1}-x_n\|^2,
\ee
where the second inequality follows from $\kappa_n =\omega\delta_n+ \gamma\tau_n$ (see \eqref{beta-gstrong}), $(\psi, \xi,\varphi)\in\Theta_{sc}$ (see  \eqref{def:Theta-sc} and \eqref{def:zeta}) and $\delta_n\leq \varphi$.
It follows from $\tau_{n+1}\leq\varphi\tau_{n}$ and the definition of $\rho_n$ in \eqref{beta-gstrong} that
\be\label{jy-6}
  \frac{\psi(1+\gamma\tau_n)}{\psi+\gamma\tau_{n+1}}
&\geq& \frac{\psi(1+\gamma\tau_n)}{\psi+\gamma\varphi\tau_{n}}
= 1+\frac{(\psi-\varphi)\gamma\tau_n}{{\psi}+\gamma\varphi\tau_n}
=  1+\rho_{n+1}\gamma\tau_n.
\ee
Consequently, we have
\ben\label{key-ineq}
%(1+\gamma\tau_n)\frac{\psi}{\psi-1}
\frac{\psi(1+\gamma\tau_n)}{\psi-1} = \frac{\psi(1+\gamma\tau_n)}{\psi+\gamma\tau_{n+1}} \frac{\psi+\gamma\tau_{n+1}}{\psi-1}
\stackrel{\eqref{jy-6}}\geq (1+\rho_{n+1}\gamma\tau_n)\frac{\psi+\gamma\tau_{n+1}}{\psi-1}
\stackrel{\eqref{beta-gstrong}}=\frac{\beta_{n+1}}{\beta_n}\frac{\psi+\gamma\tau_{n+1}}{\psi-1}.
\een
By plugging the above inequality into (\ref{ineq_rate2-sc}) and using  the definitions of $A_n(y)$ and $B_n$ in \eqref{def:AnyBn},
we obtain the desired result \eqref{key-ineq2}.
\end{proof}

The subsequent lemma plays a crucial role in proving convergence rate results for Algorithm~\ref{algo1-sc}. Its proof shares similarity to that of Lemma~\ref{lem_bound} and is deferred to Appendix \ref{proof:lem-beta-cnn}.
\begin{lem}\label{lem-beta-sc}
(i) The linesearch step of Algorithm~\ref{algo1-sc}, i.e., Step 2, always terminates.
(ii) The sequence $\{(x_n,y_n,z_n): n\geq 1\}$ generated by Algorithm \ref{algo1-sc} is bounded;
(iii) There exist  constants $c_1, c_2>0$ such that $\sqrt{\beta_n} \tau_n\geq c_1$ and $\beta_n \geq c_2 n^2$ for all $n\geq 1$.
\end{lem}

Now, we are ready to present the pointwise convergence and ergodic $\cO(1/N^2)$ convergence rate results of Algorithm \ref{algo1-sc}.
Note that, since $g$ is strongly convex,  ${x^{\star}}$ is unique.
\begin{thm}[Convergence results]
  \label{thm:gstrong}
 Let $\{(z_n,x_n,w_n,y_n, \beta_n, \tau_n): n\geq 1\}$ be the sequence generated by Algorithm \ref{algo1-sc}.
Then, we have
(i)   there exist constants $C_1>0$ and $C_2 > 0$ such that
$\|z_{n+1}-{x^{\star}}\| \leq C_1/n$ and  $\|x_{n+1}-{x^{\star}}\| \leq C_2/n$ for any $n\geq 1$;
%(b)  there exists a subsequence of $\{y_n: n \ge 1\}$ converging to $y^*$ such that $({x^{\star}}, y^*)$ is a solution of (\ref{saddle-point-special}); \\
and (ii) there exists a constant $C_3 >0$ such that
$
|F({\hat x}_N,{\hat w}_N)-F({x^{\star}},{w^{\star}})| \leq {C_3 / N^2}$ and
$
\|H({\hat x}_N)-{\hat w}_N\| \leq   {(2C_3/c) / N^2}$ for any $N\geq 1$,
where $c > 0$ is a constant satisfying $c \geq 2\|{y^{\star}}\|$, and ${\hat x}_N$ and ${\hat w}_N$ are defined as
\be\label{x_n}
{\hat x}_N=\frac{1}{s_N}\sum_{n=1}^N \beta_{n}\tau_{n}x_n
\text{~~and~~}
{\hat w}_N=\frac{1}{s_N}\sum_{n=1}^N \beta_{n}\tau_{n}w_n
\text{~~with~~}
s_N=\sum_{n=1}^N \beta_{n}\tau_{n}.
\ee
\end{thm}
\begin{proof}
It follows from \eqref{linesearch-cond-sc} and \eqref{def:AnyBn} that $B_n\geq0$ for any $n\geq 1$.
Let $y\in\R^p$ be arbitrarily fixed.
By dropping $B_n$ on the right-hand-side of \eqref{key-ineq2} and taking a sum over $n = 1, \ldots, N$, we obtain
\be\label{beta-A}
\beta_{N+1}A_{N+1}(y)+\sum_{n=1}^N\beta_{n}\tau_{n} \widetilde{J}(x_n,w_n,y) \leq \beta_{1}A_{1}(y).
\ee
Recall that $\widetilde{J}(x,w,y)$ is defined in \eqref{def:Js} and $\widetilde{J}(x_n,w_n,{y^{\star}})$ is always nonnegative.
Then,
  we can infer from \eqref{beta-A} and the definition of $A_n(y)$ in \eqref{def:AnyBn} that
\be
\|z_{N+2}-{x^{\star}}\|^2 \leq \frac{2(\psi-1)}{\psi +  \gamma \tau_{N+1}}
{\beta_{1}A_{1}({y^{\star}}) \over \beta_{N+1}} \leq {2\beta_{1}A_{1}({y^{\star}}) \over \beta_{N+1}}.
\label{z-N}
\ee
Consequently, it follows from (\ref{z-N})  and item (iii)  of Lemma \ref{lem-beta-sc} that $\|z_{N+2}-{x^{\star}}\| \leq C_1/(N+1)$ with $C_1 := \sqrt{2\beta_1A_1({y^{\star}})/c_2} > 0$. Since $z_{N+2}=\frac{\psi-1}{\psi} x_{N+1} + \frac{1}{\psi}z_{N+1}$,
we derive $\|x_{N+1}-{x^{\star}}\| \leq C_2/N$ for some $C_2 >0$.
Hence, property (i) holds.

On the other hand, $\widetilde{J}(x,w,y)$  is convex with respect to $(x,w)$ for any $y$.
After discarding the nonnegative term $\beta_{N+1}A_{N+1}(y)$ from (\ref{beta-A}), we can utilize (\ref{x_n}) and Jessen's inequality to obtain
\be
\widetilde{J}({\hat x}_N,{\hat w}_N,y) \leq   {1\over  s_N}\sum_{n=1}^N\beta_{n}\tau_{n} \widetilde{J}(x_n,w_n,y) \leq { \beta_{1}A_{1}(y) \over s_N}.\label{G-XY}
\ee
Furthermore, it follows from % the definitions of $\beta_n$ and $\rho_n$ in
\eqref{beta-gstrong} that
$\beta_n \tau_n = (\beta_{n+1} - \beta_n) / (\rho_{n+1} \gamma) \ge (\beta_{n+1} - \beta_n) / \gamma$.
Hence, we have
$s_N =\sum_{n=1}^N \beta_{n}\tau_{n} \ge (\beta_{N+1} - \beta_1) / \gamma \ge c_3 N^2$ for some $c_3 > 0$,
since  $\beta_{N+1} \ge c_2 (N+1)^2$ by item (iii)  of Lemma \ref{lem-beta-sc}.
Consequently, it follows from (\ref{G-XY}) that
$$
\widetilde{J}({\hat x}_N,{\hat w}_N,y) =F({\hat x}_N,{\hat w}_N)+\langle y, H( {\hat x}_N)-{\hat w}_N\rangle-F({x^{\star}},{w^{\star}})\le (\beta_1/c_3) A_1(y)/ N^{2}.
$$
By taking supremum on both sides of the above inequality over $\|y\|\leq c$, we obtain
\ben
F({\hat x}_N,{\hat w}_N)+c\|H( {\hat x}_N)-{\hat w}_N\|-F({x^{\star}},{w^{\star}})\leq C_3/N^2,
\een
where $C_3= \sup_y\{(\beta_1/c_3)A_1(y): \|y\| \leq c\} > 0$.
Since we have $c \geq 2\|{y^{\star}}\|$, we can use similar reasoning as in Theorem \ref{thm:scr} to establish the validity of property (ii).
\end{proof}

\section{A Fully Adaptive Proximal Gradient Method}
\label{sec:adaptive}
This section is devoted to the special case \eqref{pd_gf}, which fits in \eqref{saddle_point} with $\Phi(x,y) = \langle H(x),y\rangle$ and $f^*(y) = \iota_{\mathds{1}/p }(y)$.
In this case, we have $\nabla_x \Phi(x,y) = H'(x)^\top y$ with $H'(x) = [\nabla h_1(x), \ldots, \nabla h_p(x)]^\top$, $\nabla_y \Phi(x,y) = H(x)$ and $L_{yy}=0$.
Recall that $h := {1\over p}\sum_{i=1}^{p}h_i$ as defined in  \eqref{pd_gf}.
Applying Algorithm \ref{algo1} to problem \eqref{pd_gf}, we have $y_n \equiv \mathds{1}/p\in \bR^p$ for all $n\geq 1$.
As a result, we have $\nabla_x \Phi(x_{n},y_{n}) =   \nabla h(x_{n})$ for any $n\geq 1$, $\theta_n =  \nabla h(x_n)-\nabla h(x_{n-1})$ and
$\Phi_n^y\equiv 0$ for all $n$ (see definitions of $\theta_n$ and $\Phi_n^y$ in \eqref{def:q-La} and \eqref{def:phi-y}).
Recall that $\delta_{n-1} = \tau_{n-1}/\tau_{n-2}$.
Therefore, the linesearch condition \eqref{linesearch-cond1}
can be equivalently rewritten as
\[
\tau_n \le \frac{\xi (\nu r_n + (1-\nu)c_n)}{\tau_{n-1}},
\]
where $r_n$ and $c_n$ are defined in  \eqref{linesearch-cond1}.
Setting $\eta =0$, we have $c_n=0$ for all $n$ and the above condition will reduce to
\ben%\label{linesearch-cond2}
%\frac{\tau_n\tau_{n-1}}{\xi}\|\nabla h(x_n)-\nabla h(x_{n-1})\|^2 &\leq& \nu\omega\delta_{n-1} \|x_{n}-x_{n-1}\|^2,
\tau_n\leq\frac{\nu \xi\omega }{\tau_{n-2}} \cdot \frac{\|x_{n}-x_{n-1}\|^2}{\|\nabla h(x_n)-\nabla h(x_{n-1})\|^2}.
\een
%where $\omega=\omega(\xi,\varphi)$ is defined in \eqref{def:zeta}.
As a result,  by setting parameter $\eta=0$, Algorithm \ref{algo1} reduces to the following
fully adaptive proximal gradient method with convex combination (aPGMc) given below.

\vskip5mm
\hrule\vskip2mm
\begin{algo}
[A fully adaptive PGM with convex combination (aPGMc) for \eqref{pd_gf}]\label{algo2}
{~}\vskip 1pt {\rm
\begin{description}
\item[{\em Step 0.}] Choose $\psi\in(1,1+\sqrt{3})$, $\nu\in(0, 1)$,  $(\xi,\varphi)\in \Theta_\psi$, $\tau_{\max}>0$ and $\mu\in(0,1)$. Choose $x_0 \in \dom(g)$ and $\tau_0 \in (0, \tau_{\max}]$.
    Set $z_{0}=x_0$, $\omega:= \omega(\xi,\varphi)$, $\tau_{-1}=\tau_0$   and $n=1$.

\item[{\em Step 1.}] Compute
\ben
  z_{n}=\frac{\psi-1}{\psi} x_{n-1} + \frac{1}{\psi}z_{n-1},~~
 x_{n}=\prox_{\tau_{n-1} g}(z_{n}-\tau_{n-1} \nabla h(x_{n-1})).
\een

\item[{\em Step 2.}] Set
$\tau_n=\min\Big\{ \varphi\tau_{n-1},~ \frac{\nu  \xi  \omega}{\tau_{n-2}}\cdot \frac{\|x_{n}-x_{n-1}\|^2}{\|\nabla h(x_n)-\nabla h(x_{n-1})\|^2},~ \tau_{\max}\Big\}$,   $n \leftarrow n+1$
and go to Step 1.
\end{description}
}
\end{algo}
\vskip1mm\hrule\vskip5mm

%From the second condition in \eqref{condition-tau_n}, a larger stepsize may be adopted when a larger value of ${\bar \omega}(\xi,\varphi):=\xi\omega(\xi,\varphi)=\xi \big(2\psi-\xi -\frac{\psi^3\varphi }{1+\psi}\big)$ can be used. However for a given $\psi\in (1,1+\sqrt{3})$, from $\nabla_\varphi{\bar \omega}(\xi,\varphi)=-\frac{\psi^3\xi}{1+\psi}<0$ and $\nabla_\xi{\bar \omega}(\xi,\varphi)=2\psi-2\xi-\frac{\psi^3\varphi}{1+\psi}=0$, larger value of ${\bar \omega}(\xi,\varphi)$ appears when $\varphi$ closes to 1 and $\xi=\psi-\frac{\psi^3\varphi}{2(1+\psi)}$. However, small $\varphi$ hinders the increase of stepsize and is conducive to adaptive adjustment \comm{??}. Hence, for given $\psi\in(1,1+\sqrt{3})$, we firstly select $\varphi$ by considering set $\Theta_\psi$ \comm{??}, and then set $\xi=\psi-\frac{\psi^3\varphi}{2(1+\psi)}$, see Fig.\ref{Fig psi}.

%From the results shown in Fig. \ref{Fig psi}, $(\psi, \varphi)$ in GRA \cite{Malitsky2019Golden} is selected in \blue{blue dotted line} with $\xi=\frac{\psi}{2}$, which contained in the expanded area with $(\xi,\varphi)\in\Theta_\psi$. Special selection $(\psi=3/2, \varphi=10/9)$ in GRA \cite{Malitsky2019Golden} corresponds to the case of $\xi=\psi-\frac{\psi^3\varphi}{2(1+\psi)}=3/4$ and ${\bar \omega}(\xi,\varphi)=\xi^2=9/16$.

%This setting increases the proportion of latest $x_{n-1}$ in convex combination, allows stepsize to increase at a faster rate $\varphi=6/5$, and then improves numerical performance though smaller ${\bar \omega}(\xi,\varphi)$ is used.

Define ${\bar \omega} := {\bar \omega}(\xi,\varphi) := \xi\omega = \xi \omega(\xi,\varphi)=\xi \big(2\psi-\xi -\frac{\psi^3\varphi }{1+\psi}\big)$  and recall that the interested parameters are restricted by $\psi\in (1,1+\sqrt{3})$, $\varphi>1$ and $\xi>0$.
From Step 2 of Algorithm \ref{algo2},  aside from the value of $\varphi$, larger ${\bar \omega}$ likely results to  larger stepsize $\tau_n$. Apparently,  ${\bar \omega}(\xi,\varphi)$ is monotonically decreasing with respect to $\varphi$.
On the other hand,  ${\bar \omega}(\xi,\varphi)$ attains its maximum value at $\xi=\psi-\frac{\psi^3\varphi}{2(1+\psi)}$ for any $\psi$ and $\varphi$.
Practically, it is not desirable to choose $\varphi$ very close to $1$ since that will hinder the increase of stepsize.
For any given $\psi\in(1,1+\sqrt{3})$, our strategy is to choose $\varphi$ first, followed by setting $\xi=\psi-\frac{\psi^3\varphi}{2(1+\psi)}$.
The relation  between $\psi$ and $\varphi$ determined by $\xi=\psi-\frac{\psi^3\varphi}{2(1+\psi)}$ with a sample of different choices of $\xi$ are drawn in Figure \ref{Fig psi}.

\begin{figure}[htp]
\center
\includegraphics[width=0.6\textwidth]{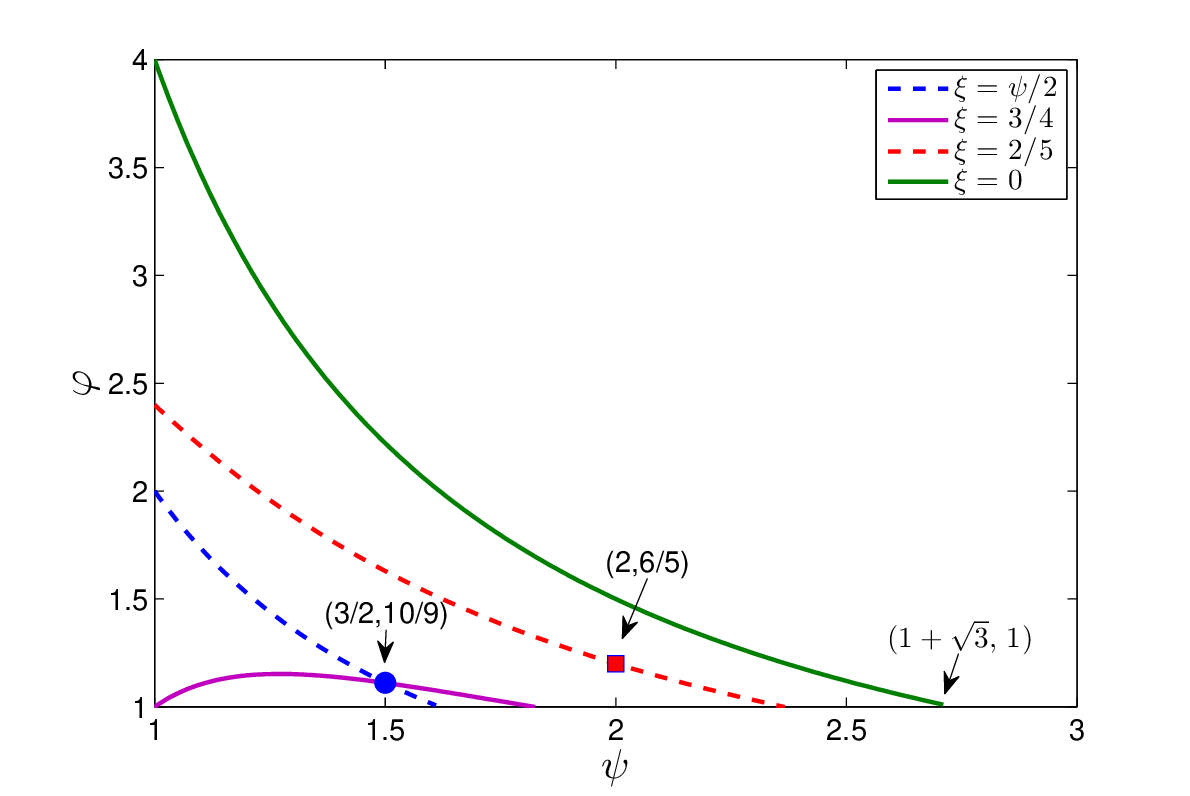}
\caption{The relation between $\psi$ and $\varphi$ determined by $\xi=\psi-\frac{\psi^3\varphi}{2(1+\psi)}$ with different values of $\xi$.} \label{Fig psi}
%\comm{$\xi=0$ and $\xi=2\psi-{\psi^3\varphi \over 1+\psi}$ correspond to the same curve?}\reply{No. For any $(\psi,\varphi)$, let $c = 2\psi-\frac{\psi^3\varphi }{1+\psi}$, then $\bar{\omega}=\xi(c-\xi)$. ${\bar \omega}$ attains its maximum value at $\xi=\frac{c}{2}$, and $\bar{\omega}=0$ when $\xi=0$ or $\xi=c=2\psi-{\psi^3\varphi \over 1+\psi}$. } {\color{blue} Yes, I know what you mean. What I was saying is that fixing $\xi=0$ or  $\xi=2\psi-{\psi^3\varphi \over 1+\psi}$ in $\xi=\psi-\frac{\psi^3\varphi}{2(1+\psi)}$ would give the same relation between $\psi$ and $\varphi$, which is $\psi^2\varphi = 2(1+\psi)$. Thus, $\xi=0$ and  $\xi=2\psi-{\psi^3\varphi \over 1+\psi}$ give the same cure, i.e., the green curve.}
\end{figure}

In fact, our adaptive rule of choosing algorithmic parameters, i.e., the formula for $\tau_n$ in Step 2 of Algorithm \ref{algo2}, generalizes
that in \cite{Malitsky2019Golden}.
In particular, aside from the parameter $\nu\in(0,1)$,  the adaptive rule adopted in \cite[Algorithm 1]{Malitsky2019Golden} essentially corresponds to setting $\xi = {\psi /2}$ in the relation $\xi=\psi-\frac{\psi^3\varphi}{2(1+\psi)}$, which gives $\bar \omega = {\psi^2/ 4}$, see \cite[Eq. (21)]{Malitsky2019Golden}.
Moreover, the algorithm presented in \cite[Page 3]{Malitsky2019Golden} corresponds to $(\psi,\varphi)=({3/ 2}, {10/ 9})$
and $\xi = {\psi/ 2} = {3/ 4}$, which gives  $\bar \omega = {9/ 16}$.

In our numerical experiments, we first set $(\psi,\varphi)=(2, {6/ 5})$ and then compute $\xi = \psi-\frac{\psi^3\varphi}{2(1+\psi)}=2/5$, which gives ${\bar \omega}(\xi,\varphi)= 4/25$. The resulting algorithm appears as
\ben
%\mbox{aPGMc:}~~
\left\{\ba{rcl}
\tau_n&=&\min\left\{\frac{6}{5}\tau_{n-1},~\frac{4\nu}{25\tau_{n-2}}\frac{ \|x_{n}-x_{n-1}\|^2}{\|\nabla h(x_{n})-\nabla h(x_{n-1})\|^2},~\tau_{\max}\right\}, \smallskip \\
 z_{n+1}&=& \big( x_{n} + z_{n}\big)/2,\smallskip  \\
 x_{n+1}&=&\prox_{\tau_{n} g}(z_{n+1}-\tau_{n} \nabla h(x_{n})),
\ea\right.
 \een
Compared with the algorithm in \cite[Page 3]{Malitsky2019Golden}, our choice of $\psi = 2$ results to larger weight  for the latest iterate $x_{n-1}$ in the convex combination. Furthermore, the choice of $\varphi=6/5$ allows the stepsize to increase at a faster rate, compared to  the algorithm in \cite[Page 3]{Malitsky2019Golden}, which sets the increasing factor to be $10/9$.
Experimental results show that our algorithm is more efficient, see Section \ref{sec-experiments}.

%where %$x_{0}, x_1 = z_{1} \in \bR^q,$ $\tau_{-1}=\tau_0>0$ and
%$\tau_{\max}>0$ is a parameter to bound the stepsize $\tau_n$.
%\comm{need to specify how to choose parameters in aPGMc to reduces to the above}
% \revise{Here, for convenience we adopt the convention $c/0=+\infty$ for any $c\geq 0$.}
%This scheme improves the GRA presented in \cite{Malitsky2019Golden} in the range of pivotal parameters,

\section{Numerical Results}
\label{sec-experiments}
In this section, we provide numerical results from two examples to demonstrate the performance of the proposed Algorithms \ref{algo1} and \ref{algo1-sc}, which we will refer to as PDAc-L and aPDAc-L. The first example is a quadratic constrained quadratic programming (QCQP) problem, and the second is a sparse logistic regression (SLR) problem. We implemented our experiments in Matlab R2013b on a 64-bit Windows PC with an Intel(R) Core(TM) i5-4590 CPU@3.30 GHz and 8GB of RAM. The codes used in the experiments are available at \url{https://github.com/xkchang-opt/PDAc-Linesearch}.

\subsection{Implementation Detials}
As noted in \cite{Nesterov2013gradient}, rounding errors can occur when the objective function's value is used in linesearch, which may cause the termination criterion to fail when generating iterates with very close function values.
Note that, in the linesearch step of Algorithm \ref{algo1}, computing the difference $\Phi(x_n,y_{n-1})-\Phi(x_n,y_{n})$ is necessary to evaluate $\Phi_n^y$.
To enhance stability and eliminate the use of function values, we have incorporated an alternative termination criterion into the linesearch procedure by replacing $\Phi_n^y$ with $\widetilde{\Phi}_n^y :=\langle \nabla_y \Phi (x_n, y_{n-1})-\nabla_y \Phi (x_n, y_{n}),y_n-y_{n-1} \rangle$, as suggested in \cite{EYNS2021}.
Due to the concavity property of $\Phi(x,y)$ in $y$, we can derive $\Phi_n^y \leq \widetilde{\Phi}_n^y$
and $\Phi_n^y = \widetilde{\Phi}_n^y$ when function $\Phi (x, \cdot)$ is quadratic for fixed $x$,
which results in a more robust and reliable termination condition. From our observations, this modification makes the algorithm numerically more stable.

The performance of primal-dual algorithms is widely recognized to depend not only on the primal and dual step-sizes but also on their ratio, as has been established in previous studies like \cite{ChY2022relaxed,Sun2014A}.
If we define the dual stepsize as $\sigma_n :=\beta\tau_n$, then the ratio $\beta= \sigma_n/\tau_n$ can greatly affect the performance of the algorithm. It is important to note that a fixed ratio may not always produce the best numerical results, even though the convergence is assured for any $\beta>0$.
Choosing the optimal ratio $\beta$ is crucial for achieving good numerical performance in practice, and an adaptive approach that balances primal and dual feasibility violations can help achieve this. Our adaptive approach involves evaluating inexpensive distances and can be easily applied in various scenarios.

We will now outline our strategy for dynamically adjusting the ratio $\beta$. While we do not have a complete theory for this strategy, we have observed its effectiveness in various applications, as shown in Figure \ref{Fig adaptive}.
%
%
%\be\label{def:pdinf}
%\verb"pinf"_n: = \dist(\nabla\Phi_y(x_n,y_n), \partial f^*(y_n)),~~~ %\verb"dinf"_n:=\dist(-\nabla\Phi_x(x_n,y_n), \partial g(x_n)),
%\ee
%where $\dist(v,S)$ represents the distance from the vector $v$ to the set $S$ measured by the $\ell_1$-norm.
%Note also from \cite{Sun2014A} that the adaptive selecting of ratio $\beta =\frac{\sigma_n}{\tau_n}$ can improve performance of PDAs than fixed constant, we thus
%
Many applications allow for the easy computation of the subdifferential of $g$ and $f^*$. Using this information, we can define the primal and dual infeasibilities (denoted respectively by $\verb"pinf"_n$ and $\verb"dinf"_n$) using $\partial f^*(y_n)$, $\partial g(x_n)$, $\nabla\Phi_y(x_n,y_n)$, and $-\nabla\Phi_x(x_n,y_n)$,  see \eqref{def:pdinf} for their definitions in the case of QCQP problems. With this approach, we can dynamically adjust the value of $\beta$ to ensure
that $\verb"pinf"_n$ and $\verb"dinf"_n$ are balanced.
%{\color{red} (the definition of $\verb"pinf"_n$ and $\verb"dinf"_n$ ?)}
%
%Let $r_n := \texttt{pinf}_n/\texttt{dinf}_n$. Inspired by the work of Sun et al. \cite{Sun2014A}, we have adopted an adaptive rule to determine the value of $\beta$, given by:
%
In particular, we have adopted the following adaptive rule to determine the value of $\beta$:
\be\label{adaptive-beta}
\beta=\left\{\ba{ll} \max\{0.8 \beta \underline{\beta}\},~~& ~~\mbox{if}~~ r_n \leq 0.8 , \smallskip\\
\beta~~&~~\mbox{if}~~  r_n \in (0.8, 1.25), \smallskip \\
\min\{1.25 \beta, \overline{\beta}\}~~& ~~\mbox{if}~~  r_n \geq 1.25,
\ea\right.
\text{~~with~~}r_n := \texttt{pinf}_n/\texttt{dinf}_n,
\ee
where $\underline{\beta}$ and $\overline{\beta}$ are lower and upper bounds on $\beta$.
In our experiments, we have initialized $\beta = 1$, and set the bounds as $\underline{\beta} = 0.01$ and $\overline{\beta}=100$.

The parameters of Algorithms \ref{algo1} and \ref{algo1-sc} are specified as: $\psi=2$, $\varphi=6/5$, $\nu = 0.9$, $\mu=0.7$, $\xi = 2/5$, $M=5$ and $\eta = 0.9$.
Parameters $\tau_0$ and $\tau_{\max}$ are generated as in Remark \ref{rem_tau0} with $\chi=10^6$ to meet the requirement of Lemma \ref{lem_bound} (iv).

\subsection{Experiments on QCQP Problems}
In this subsection, we provide a comparison of PDAc-L (Algorithm \ref{algo1}) against two other algorithms, namely PDA with backtracking (denoted by PDB) \cite{EYNS2021} and the fully adaptive golden ratio algorithm (denoted by aGRAAL) \cite{Malitsky2019Golden} on a set of convex QCQP problems, for which data is randomly generated, as tested in \cite{EYNS2021}.
In particular, the tested problem  is given by
\be\label{QCQP}
\left\{\ba{rcl}
h_{\text{opt}}&:=&\min\limits_{x\in X} h(x):=\frac{1}{2}x^\top A_0 x + b^\top_0 x \\
&& \mbox{s.t.~~} h_j (x):=\frac{1}{2}x^\top A_j x + b^\top_j x - c_j \leq 0,~~~
j \in \{ 1,\ldots,m\},
\ea\right.
\ee
where $X := [-10,10]^n $,   and the problem data is generated randomly as follows: $\{b_j\}^m_{j=0} \subseteq \R^n$ with elements drawn from the standard Gaussian distribution, $\{c_j\}^m_{j=1} \subseteq \R$ with elements drawn from a uniform distribution over $[0, 1]$, and $A_j = \Lambda^\top_j S_j \Lambda_j$ for $j\in \{0,1,\ldots,m\}$. Here, $\Lambda_j \in \R^{n \times n}$ is a random orthonormal matrix, and $S_j \in \R^{n\times n}_+$ is a diagonal matrix whose diagonal elements are generated uniformly and randomly from the interval $[0, 100]$, with $0$ allowed as the diagonal element of $S_j$.
Define $H(x) := (h_1(x),\ldots,h_m(x))^\top$ and  $\Phi(x,y) :=  h(x) + \langle y, H(x)\rangle$ for $(x,y)\in\R^n\times\R^m$.
Then, \eqref{QCQP} can be represented as
$\min\nolimits_{x}\max\nolimits_{y} g(x)+\Phi(x,y)-f^*(y)$,
where $g(x)=\iota_{X}(x)$, the indicator function of $X$, and $f^*(y)=\iota_+(y)$, the indicator function of the nonnegative orthant.
Note that $\Phi(x,y)$ is linear with respect to $y$.
Let $H'(x)$ be the Jacobian matrix of $H(x)$.
 It is easy to derive that
$\nabla_x \Phi(x,y) = H'(x)^\top y+A_0 x+b_0$ and $\nabla_y \Phi(x,y) = H(x)$.
As pointed out in the second point of Remark \ref{rem_tau_max}, for PDAc-L applied to this problem, we have $\Phi_n^y \equiv 0$ and therefore checking the linesearch condition \eqref{linesearch-cond1} is computationally inexpensive.

It is easy to derive from  \eqref{y1_updating} and {\bf Fact 2.1} that $y_n\in\partial f(w_n)$ for any $n\geq 1$, which
implies that $w_n\in\partial f^*(y_n)$.
Recall that  $\widetilde{\cL}(\cdot)$ and $\widetilde{\Omega}$ are defined in \eqref{def:PhiL} and \eqref{def:Omegas}, respectively, for
problem \eqref{saddle-point-special}. It is thus reasonable to
define the primal and dual feasibility violations at the $n$th iteration respectively by
\be\label{def:pdinf}
\verb"pinf"_n := \|H(x_n)-w_n\|_1 \text{~~and~~}\verb"dinf"_n:= {\dist(-H'(x_n)^\top y_n- \nabla h(x_n),\, \partial g(x_n)) \over 1+\|x_n\|_1},
\ee
where $\dist(v,S)$ represents the distance from the vector $v$ to the set $S$ measured by the $\ell_1$-norm.
Then, the ratio $\beta$ was tuned adaptively according to  \eqref{adaptive-beta}.
We terminated PDB and aGRAAL via $\max\{e_{\text{obj}}(x_n), e_{\text{con}}(x_n)\} \leq \epsilon$, where
\be\label{stop-1}
e_{\text{obj}}(x) := |h(x)-h_{\text{opt}}|/|h_{\text{opt}}|
\text{~~and~~}
e_{\text{con}}(x) := \frac{1}{m}\sum^m\nolimits_{j=1} \max\{h_j (x),0\},
\ee
with   $h_{\text{opt}}$ being computed using MOSEK via CVX\footnote{Downloaded from http://cvxr.com/cvx/}.
To terminate PDAc-L, in addition to the condition $\max\{e_{\text{obj}}(x_n), e_{\text{con}}(x_n)\} \leq \epsilon$, we also require $\max\{\verb"pinf"_n, \verb"dinf"_n\}<\epsilon_{pd}$.
Furthermore, all the tested algorithms were terminated as well if a maximum number of iterations, named $n_{\max}$, was reached.
In this set of experiments, we set $\epsilon=10^{-8}$,  $\epsilon_{pd}=10^{-6}$ and $n_{\max}=5\times 10^4$.

To examine the effectiveness of PDAc-L with $\beta$ being fixed or adaptively tuned via \eqref{adaptive-beta},
we first ran experiments with $n=100$ and $m=10$.
The performance comparison results are presented in Figure \ref{Fig adaptive}.
%
%As shown in Figure \ref{Fig adaptive}(a), for fixed $\beta$, the number of iterations of PDAc-L decreases firstly in the overall trend and then increases along with the increasing of the value of $\beta$, despite fluctuating when $\beta>5$.
%
%
%
As shown in Figure \ref{Fig adaptive}(a), the number of iterations used by PDAc-L initially decreases with increasing values of $\beta$. However, it then exhibits an overall increasing trend, with some fluctuations that may be attributed to the random nature of the test data.
In contrast, when $\beta$ was adaptively tuned using \eqref{adaptive-beta}, PDAc-L terminated after only 254 iterations, as shown in Figure \ref{Fig adaptive}(b). This is faster than PDAc-L using any fixed values of $\beta$.

\begin{figure}[htp]
\center
\subfigure[Number of iterations for fixed $\beta$.]
{\includegraphics[width=0.45\textwidth]{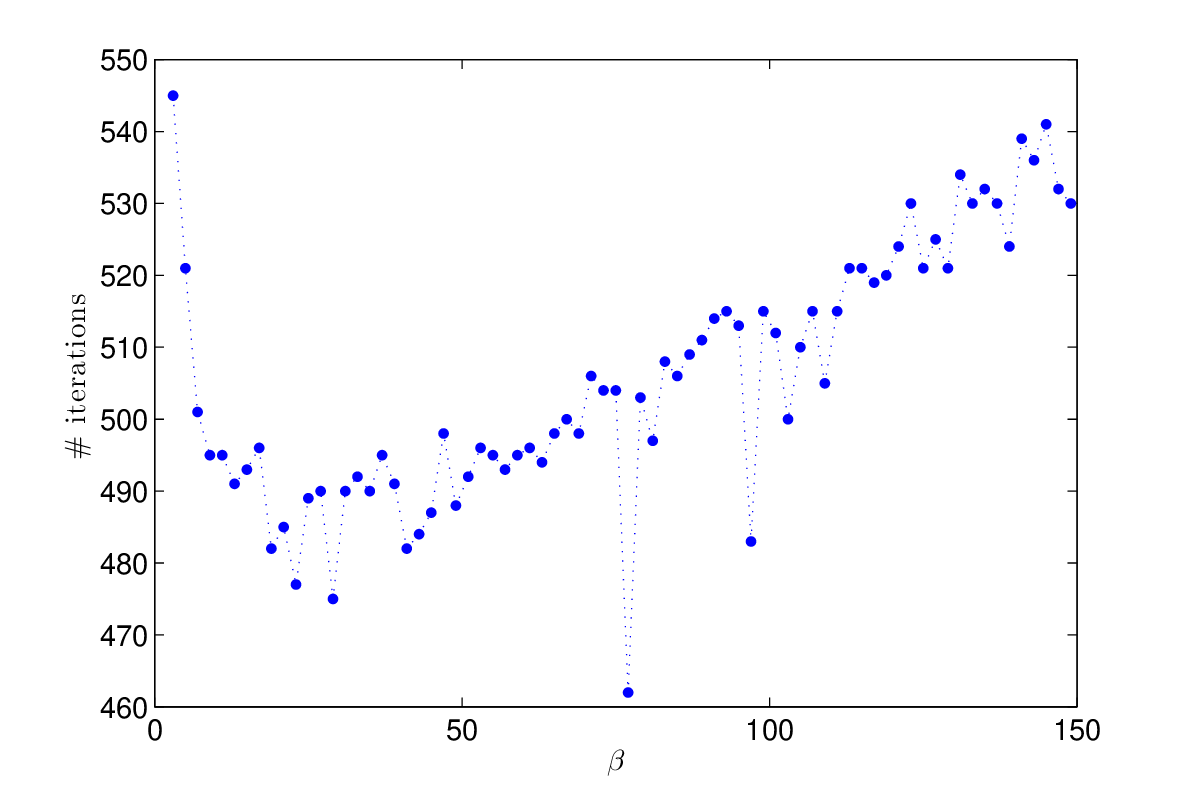}}
\subfigure[Adaptive rule v.s. $\beta=30$.]
{\includegraphics[width=0.45\textwidth]{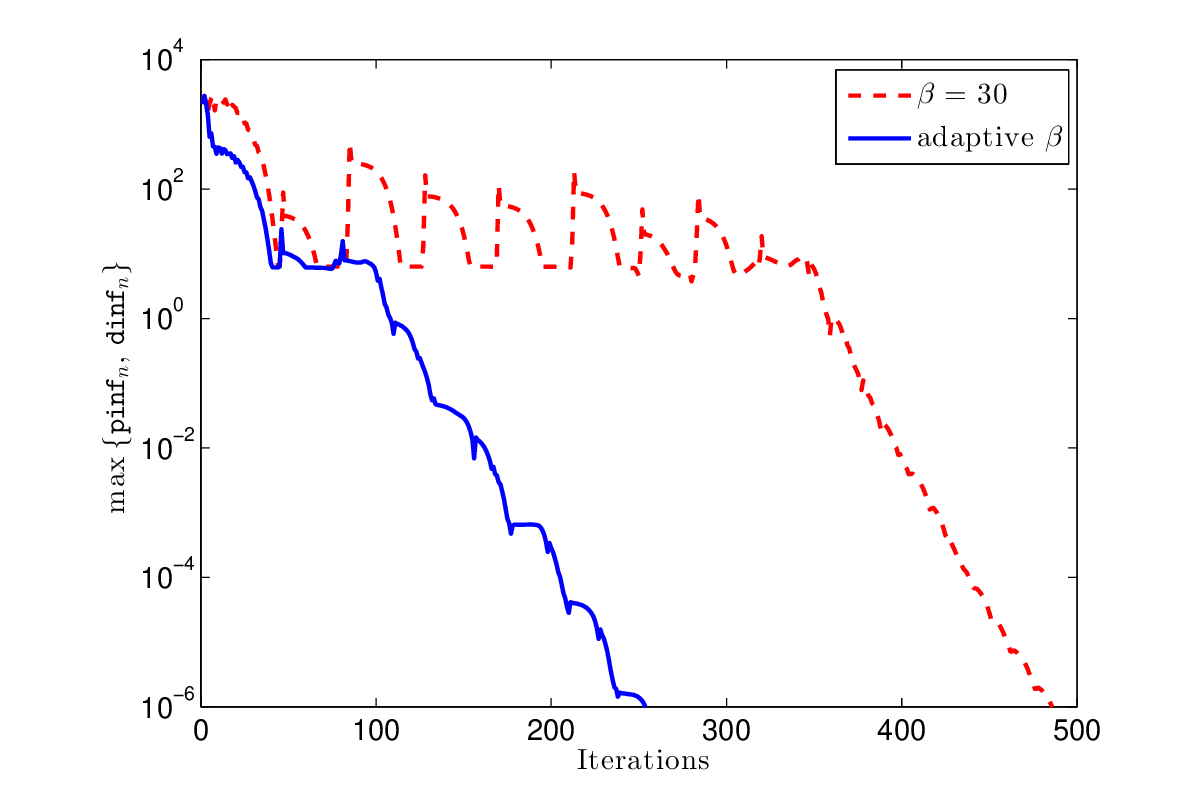}}
\caption{Performance comparison of PDAc-L with $\beta$ being constant or adaptively determined by \eqref{adaptive-beta}.
}\label{Fig adaptive}
\end{figure}

\begin{figure}[htp]
\center
\subfigure[$\beta=30$.]
{\includegraphics[width=0.45\textwidth]{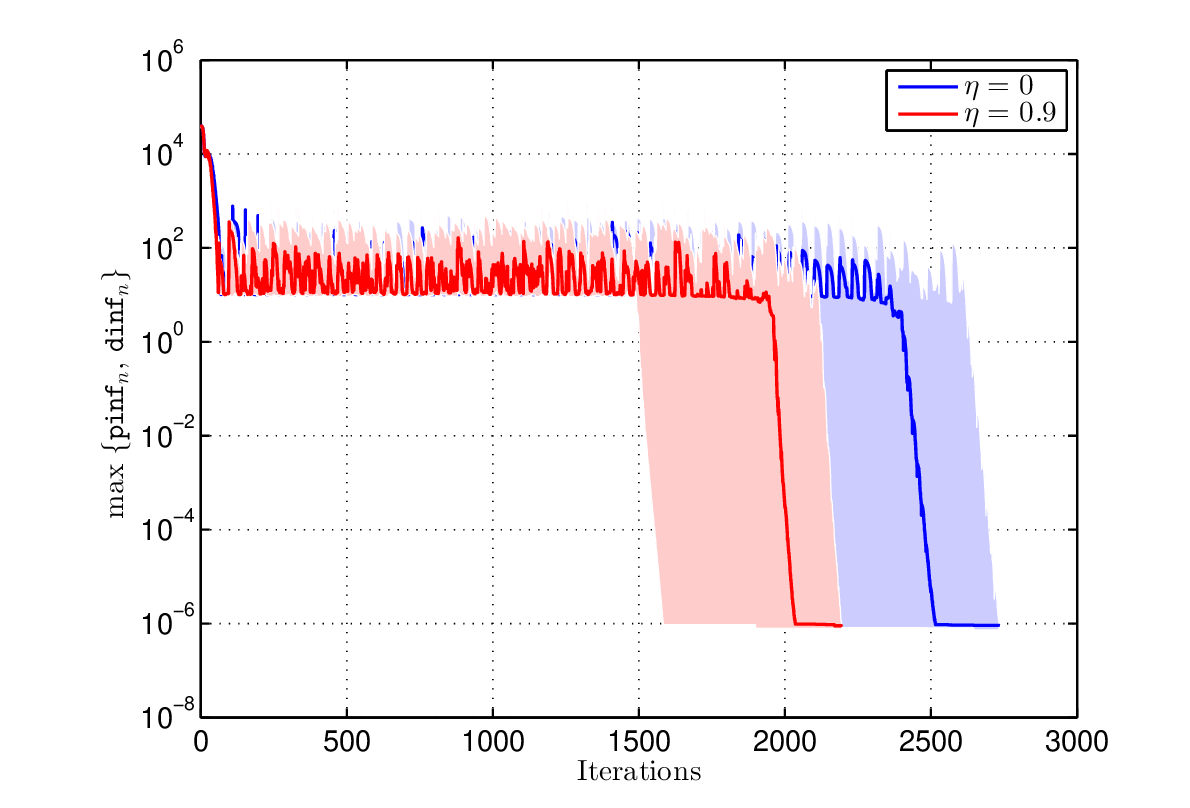}}
\subfigure[Adaptive $\beta$.]
{\includegraphics[width=0.45\textwidth]{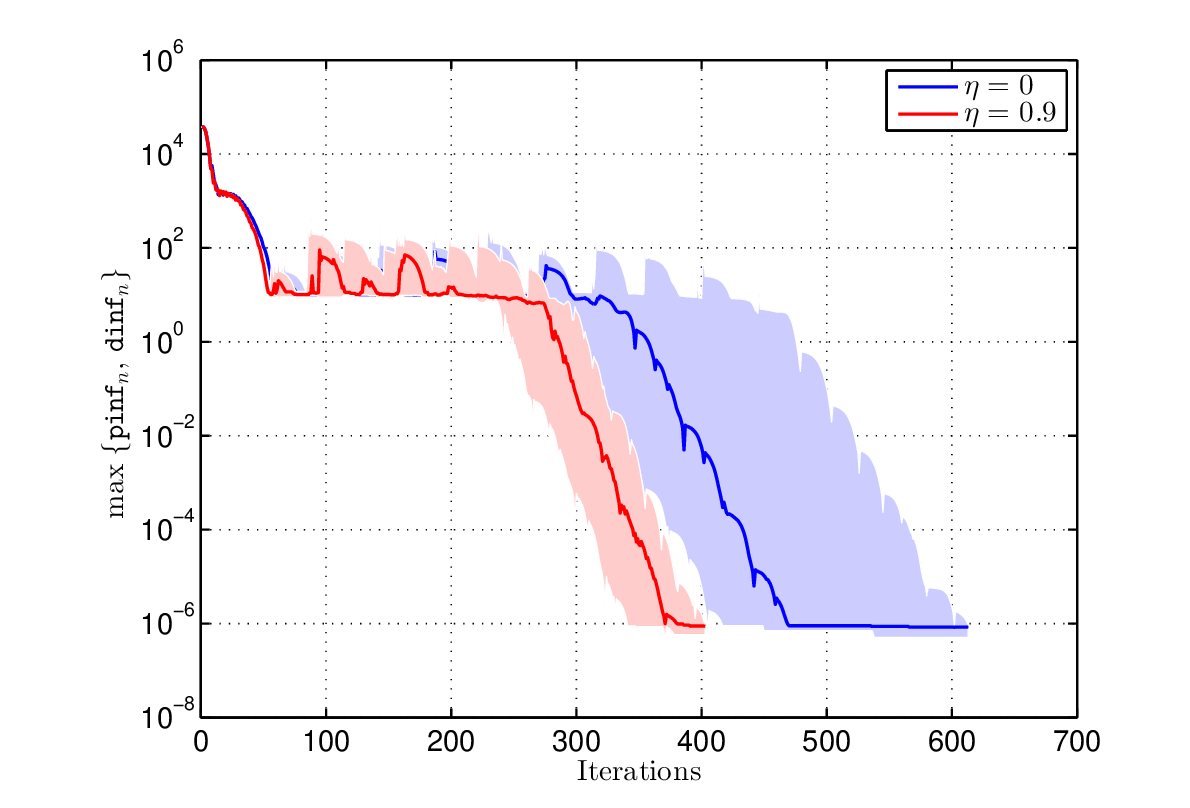}}
\caption{Performance comparison of PDAc-L with $\eta=0$ and $\eta=0.9$.
}\label{Fig eta}
\end{figure}

Additionally, we evaluated the impact of $\eta$ on the numerical performance of PDAc-L for 10 QCQP problems randomly generated with $n=500$ and $m=10$. Specifically, we considered fixed $\beta=30$ and adaptive $\beta$ tuning. The results are depicted in Figure \ref{Fig eta}, in which each solid line represents the median performance over the $10$ random instances, and the shaded area surrounding each line indicates the varying range of the corresponding values across the $10$ random instances. We observe that setting $\eta=0.9$ enhances the method's performance in comparison to the case where $\eta$ equals $0$. This is why we adopted $\eta=0.9$ for our experiments.

\begin{figure}[htp]
\center
%\hspace{-0.3cm}
\subfigure[\#LS consumed by PDB and PDAc-L.]{
\includegraphics[height=0.25\textwidth, width=0.6\textwidth]{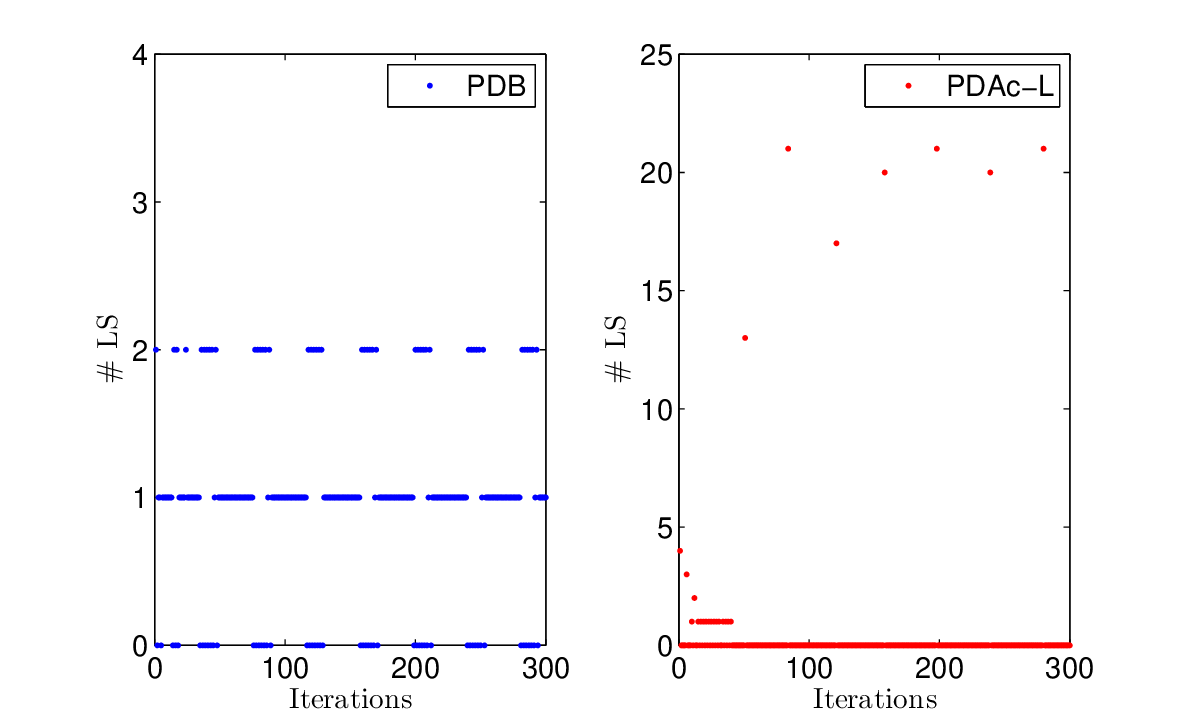}}
\hspace{-1cm}
\subfigure[Cumulative results of \#LS.]{
\includegraphics[height=0.25\textwidth, width=0.38\textwidth]{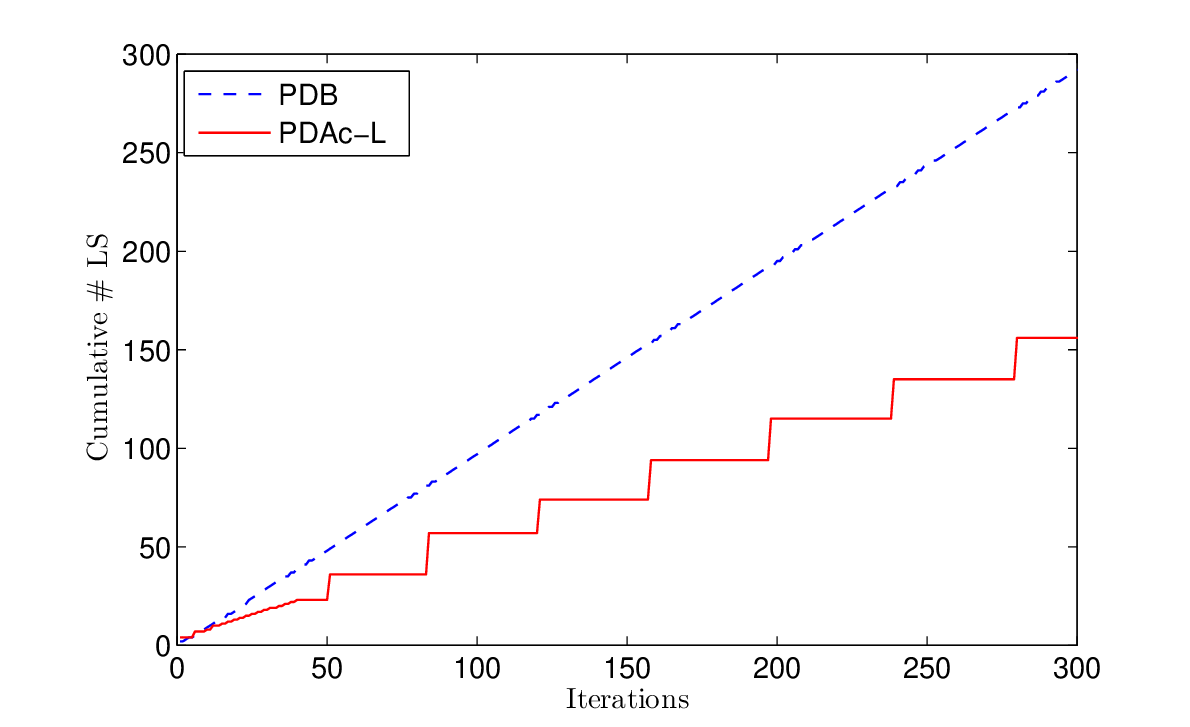}}
\caption{Details of linesearch for the problem \eqref{QCQP} with $n=500$ and $m=10$.
(a) Comparison results of extra linesearch trial steps taken by PDB and PDAc-L.
(b) Cumulative results for both algorithms.
}\label{Fig 50}
\end{figure}

To further investigate the linesearch step of PDB and PDAc-L, we solved \eqref{QCQP} with $n=500$ and $m=10$. The details of the linesearch are illustrated in Figure \ref{Fig 50}. As shown in Figure \ref{Fig 50}(a), PDB takes between $1$ and $2$ extra linesearch trials per iteration, while PDAc-L takes between $1$ and $23$ extra linesearch trials per iteration. However, as the majority of iterations in PDAc-L succeeded in the linesearch on the very first trial (i.e., no extra trial steps were needed), on average, PDAc-L requires much fewer extra linesearch steps than PDB, as shown by the cumulative results presented in Figure \ref{Fig 50}(b). Note that the results shown in Figure \ref{Fig 50} are only for the first $300$ iterations, but similar trends can be observed for more iterations.

Table \ref{table1} displays a comparative evaluation of aGRAAL, PDB, and PDAc-L, testing their performance on QCQP problems of varying dimensions. The table reports the number of iterations (\texttt{Iter}), total CPU time (\texttt{Time}, in seconds), and the number of extra linesearch trial steps (\#LS) required by PDAc-L and PDB algorithms. Note that aGRAAL does not require any linesearch.

\begin{table}[htpb]
\caption{Comparison results of aGRAAL, PDB and PDAc-L on the QCQP problems with different values of $(n,m)$.}
\label{table1}
\center
\small
\begin{tabular}{|c|c|rc|rcc|rcc|}
\hline
 &  & \multicolumn{2}{|c|}{aGRAAL }&\multicolumn{3}{|c|}{PDB}& \multicolumn{3}{|c|}{ PDAc-L}\\
$n$&$m$ &\texttt{Iter} &\texttt{Time} &\texttt{Iter}& \texttt{Time} &\#LS	&\texttt{Iter}& \texttt{Time} &\#LS\\
\hline
%100&1  & 6333 	&  1.2& 4137 	&  1.2 & 4026 &  684& 0.2 & 192  	\\
100&10  & 5092&       4.8&2777 &       4.2&2704&227&       0.1&105\\ 	
 100 & 30 & 9504&      26.4&6471 &      28.1&6312&1102&       1.0&552\\
 100 &  50 & 13760&      66.1&10646 &      81.7&10388&1958&       3.1&989\\
 \hline
 %500&  1 &   2753	&  1.5 &  903 	&  1.2   & 877  &   1494 &  0.9  &  431   	\\
 500&  10  &  6189&      36.5&2465 &      34.8&2402&391&       1.8&193  	\\
 500&  30  & 12834&     228.5&4543 &     190.5&4433&644&       8.6&318   	\\
 500&  50  &   19267	&  483.2  & 5210 &   317.9  & 5081  &  1315 &   23.6  &  657 	\\
\hline
%1000&  1  & 1945	&  5.3 &643 & 4.9& 623& 9581 &  24.1 & 2775 \\
1000&  10  & 8478	&  158.6 & 2488& 140.3 &2425 & 524 & 7.9  &  264 	\\
1000&  30  & 23750 & 1362.9  & 5877 &  976.3 & 5736 & 1141 &  51.3 &    580	\\
1000&  50  & 45604 & 4097.2 & 7909	& 2108.0  & 7643& 1824 &  141.9 &   914 	\\
 \hline
\end{tabular}
\end{table}

\begin{figure}[htbp]
\centering
\subfigure[]{
\includegraphics[width=0.46\textwidth]{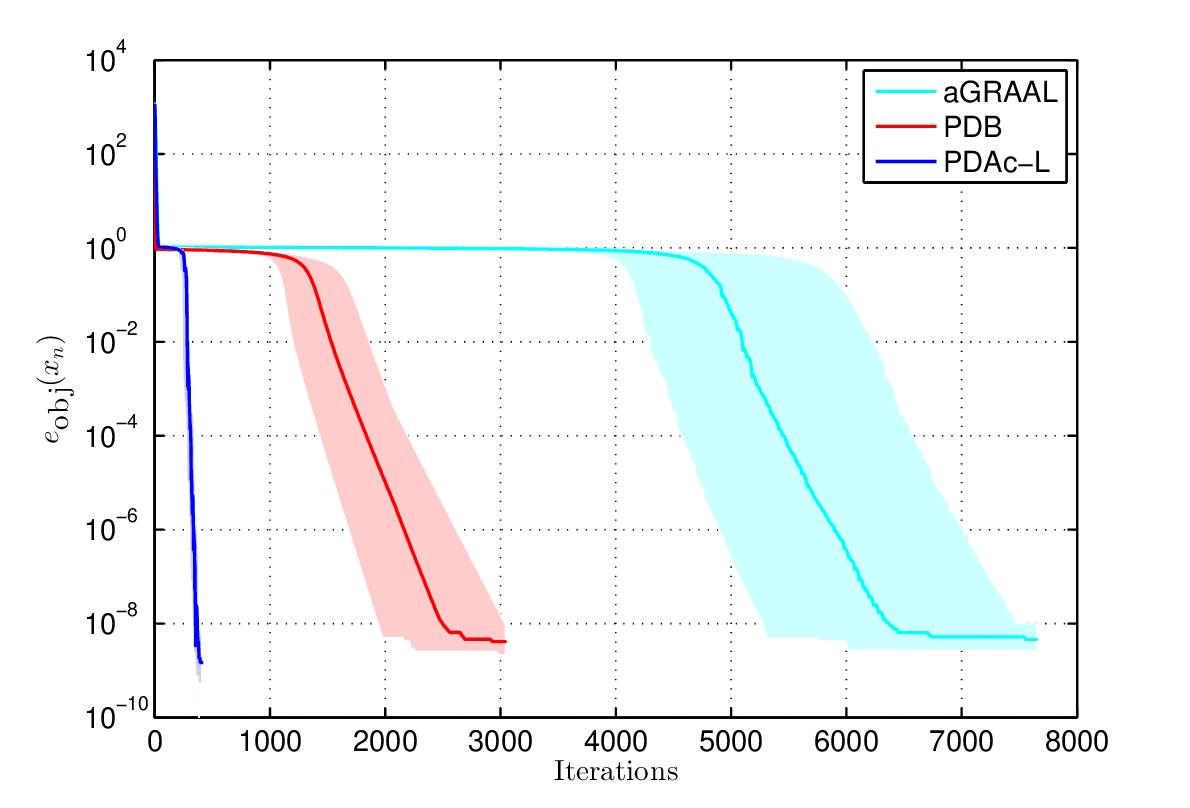}}
\subfigure[]{
\includegraphics[width=0.46\textwidth]{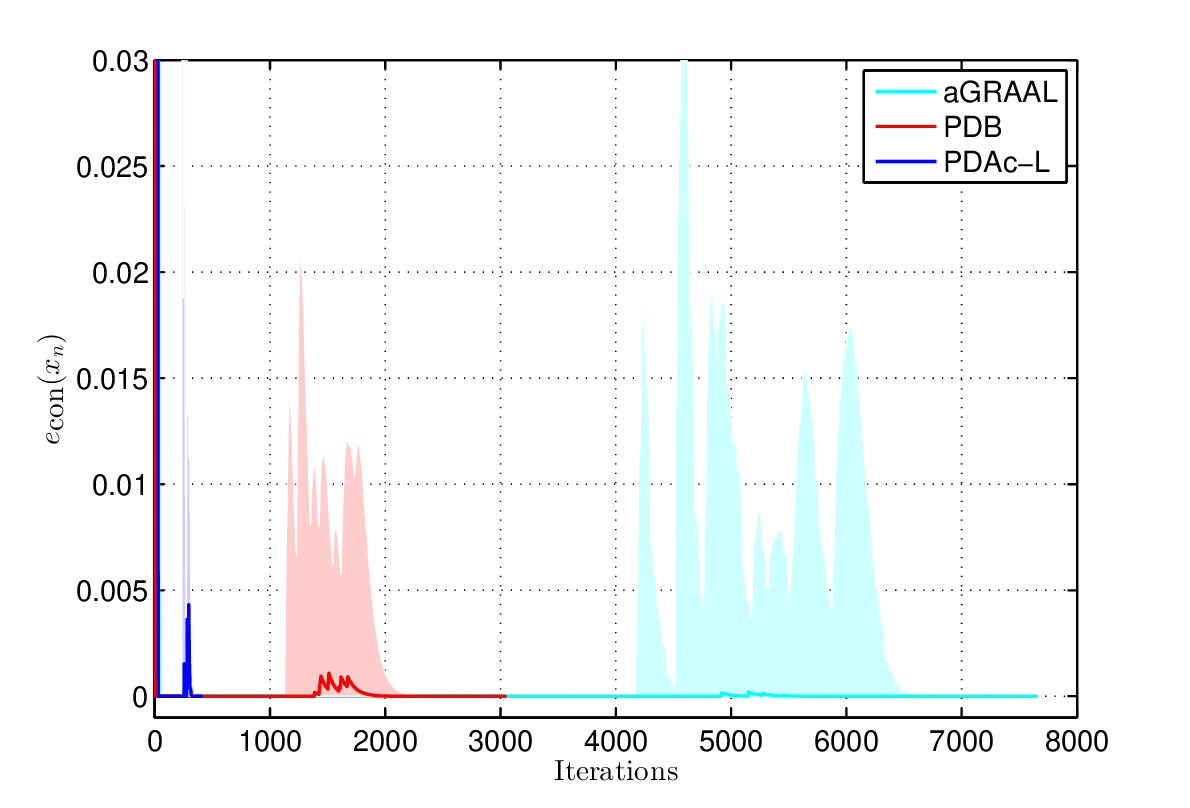}}
\subfigure[]{
\includegraphics[width=0.46\textwidth]{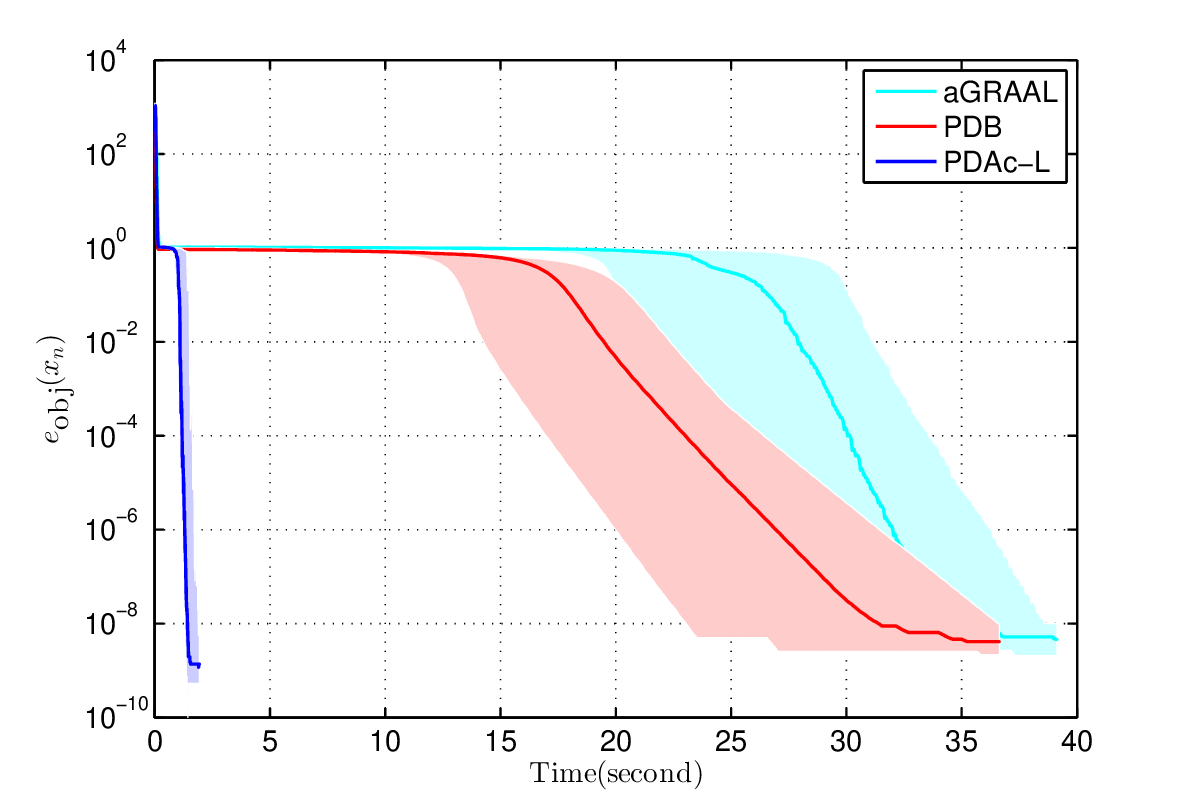}}
\subfigure[]{
\includegraphics[width=0.46\textwidth]{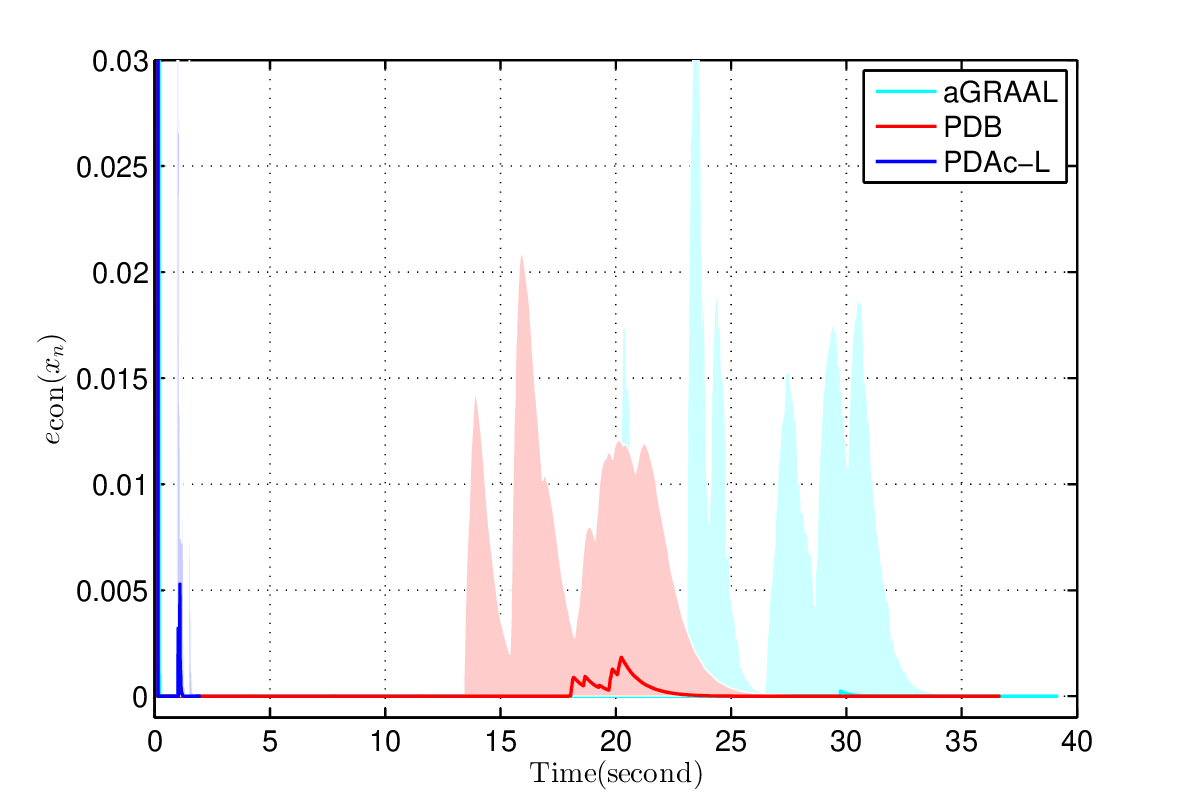}}
\caption{Comparison  results in terms of function value residual (left) and feasibility violations (right) on 10 random QCQPs with $n=500$ and $m=10$. First row:  convergence as the iterations proceeded. Second row: convergence  as CPU time proceeded.
}
\label{Fig:QCQP}
\end{figure}

The results in Table \ref{table1} indicate that aGRAAL runs the slowest among the three algorithms, whereas PDB performs better than aGRAAL, and PDAc-L records the fastest performance. Additionally, in an average sense, PDB requires around one additional linesearch trial step per iteration, while PDAc-L needs roughly one extra linesearch trial step per two iterations. Furthermore, PDAc-L requires significantly fewer iterations than PDB, resulting in much less CPU time consumption for PDAc-L with an adaptive parameter
 $\beta$ than for PDB.

We now proceed to implement aGRAAL, PDB, and PDAc-L on $10$ QCQP instances randomly generated with $n=500$ and $m=10$. The values of $e_{\text{obj}}(x_n)$ and $e_{\text{con}}(x_n)$ as defined in \eqref{stop-1} against the number of iterations and CPU time were plotted in Figure \ref{Fig:QCQP}. The same as in Figure \ref{Fig eta}, each solid line represents the median values, while the shaded area surrounding the lines denotes the variation range observed across the random runs.
The plots shown in Figure \ref{Fig:QCQP} confirmed our earlier conclusion derived from the results in Table \ref{table1}.
Furthermore, the results in Figure \ref{Fig:QCQP} also indicate that PDAc-L presents a smaller variance than the other two algorithms, which suggests that PDAc-L is more stable.

In order to assess the performance of the accelerated algorithm aPDAc-L (Algorithm \ref{algo1-sc}) in solving problem \eqref{QCQP} with a strongly convex $h$, we conducted a test using data generated in a manner similar to that used for the convex case, with the only difference being that, for $j = 0$, we set the diagonal elements of $S_0$ to be randomly generated from the range of $[1,101]$.
We can rewrite $S_0$ as the sum of a positive semidefinite matrix $\tilde{S}_0$ and the identity matrix $I$: $S_0=\tilde{S}_0+I$. This enables us to express $A_0$ as $\Lambda^\top_0 (\tilde{S}_0+I) \Lambda_0=\tilde{A}_0+I$. Moreover, we can represent $h(x)$ as $\tilde{h}(x)+\frac{1}{2}\|x\|^2$, where $\tilde{h}(x) :=\frac{1}{2}x^\top \tilde{A}_0 x + b^\top_0 x$. By doing so, we can apply aPDAc-L to solve $\min\nolimits_{x}\max\nolimits_{y} g(x) +\Phi(x,y)-\iota_+(y)$, where $g(x)$ is strongly convex and defined as $g(x)= \iota_{X}(x)+\frac{1}{2}\|x\|^2$, $\Phi(x,y)=\tilde{h}(x)+\langle H(x),y\rangle$, and $f^*(y) = \iota_+(y)$.

We tested three optimization algorithms: PDAc-L with $\beta=1$, accelerated PDB (aPDB) \cite[Algorithm 2.2]{EYNS2021} and aPDAc-L, both with $\beta_0=1$. We applied these algorithms to a set of 10 randomly generated strongly convex QCQP instances, with $n=500$ and $m=10$, as described earlier.

The results displayed in Figure \ref{Fig:QCQP-sc}, which show the evolution of the objective and constraint violation using $e_{\text{obj}}(x_n)$ and $e_{\text{con}}(x_n)$, demonstrate that aPDAc-L and PDAc-L outperform aPDB in CPU time (as shown in the bottom row of the figure). This is likely due to the fact that aPDAc-L and PDAc-L only require dual variable updates, while aPDB needs to update both primal and dual variables during each linesearch step.
In addition, the superior performance of aPDAc-L and PDAc-L in terms of CPU time can be partially attributed to the fact that these algorithms require fewer additional linesearch trials on average.
Furthermore, aPDAc-L outperforms PDAc-L due to its capacity to leverage the strong convexity of the problem.

\begin{figure}[htp]
\centering
\subfigure[]{
\includegraphics[width=0.46\textwidth]{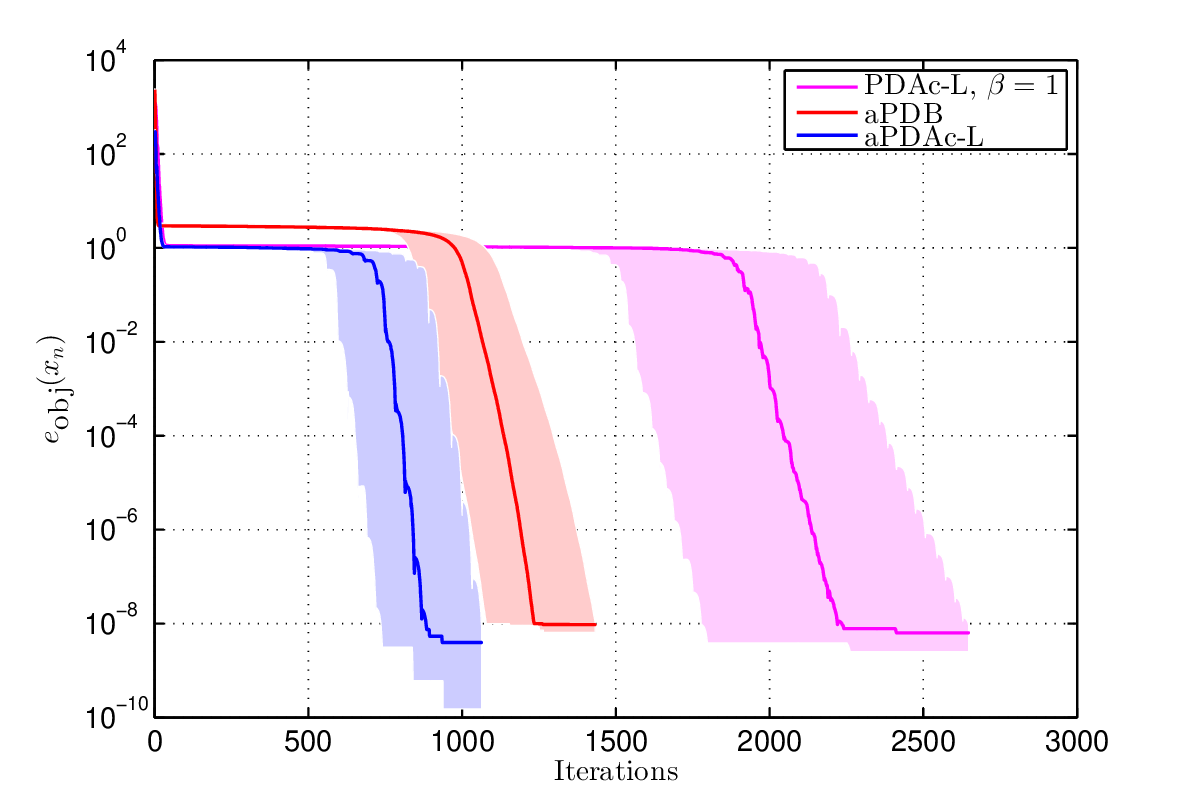}}
\subfigure[]{
\includegraphics[width=0.46\textwidth]{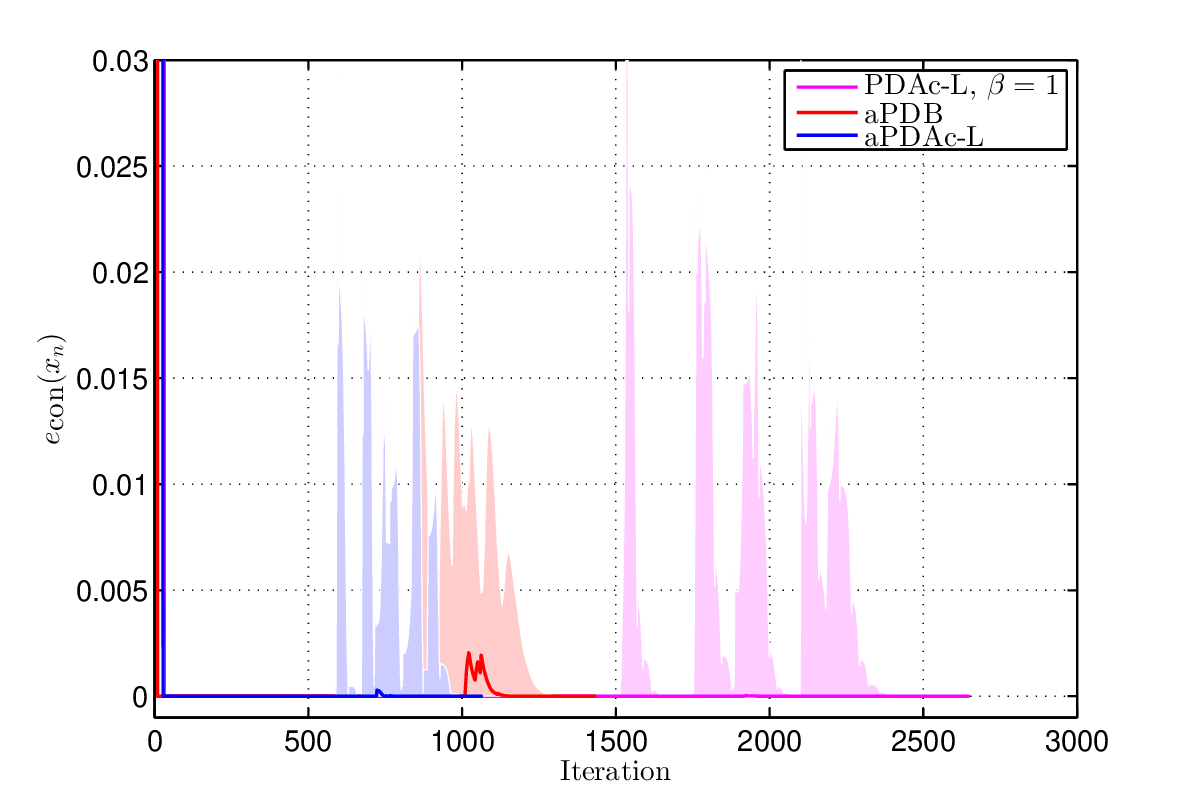}}
\subfigure[]{
\includegraphics[width=0.46\textwidth]{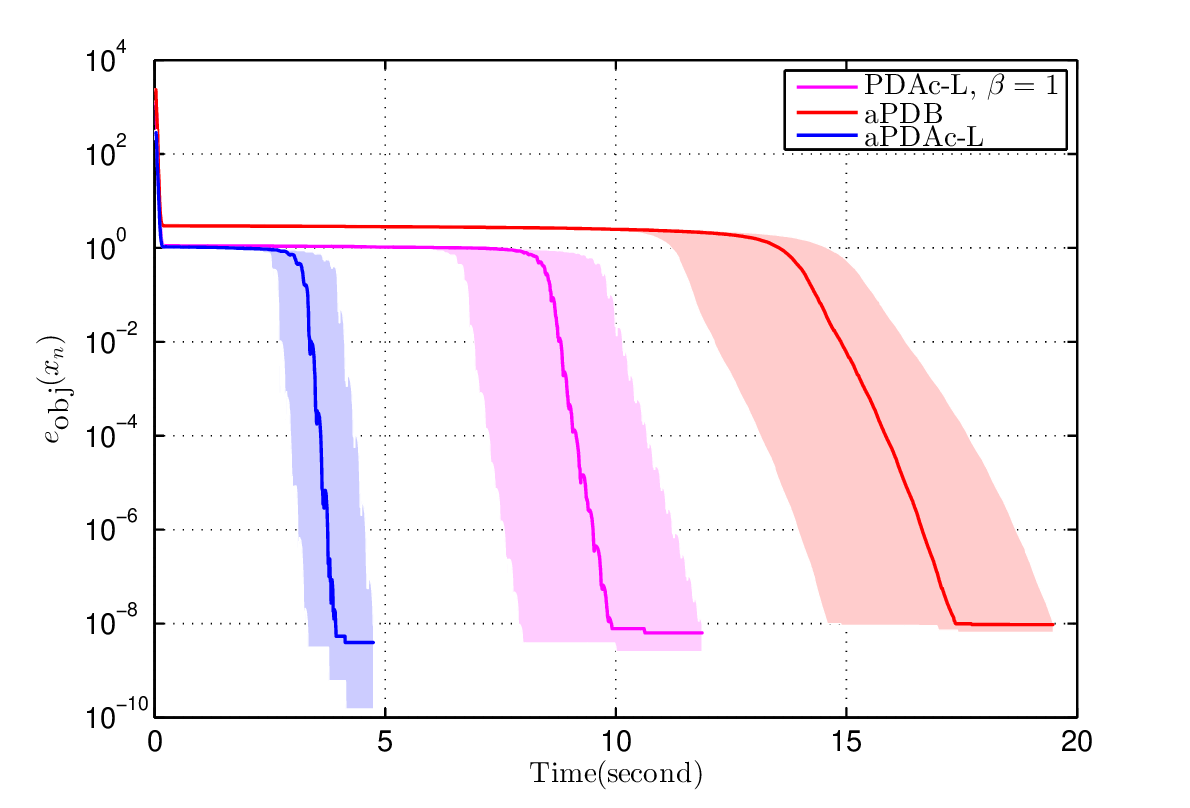}}
\subfigure[]{
\includegraphics[width=0.46\textwidth]{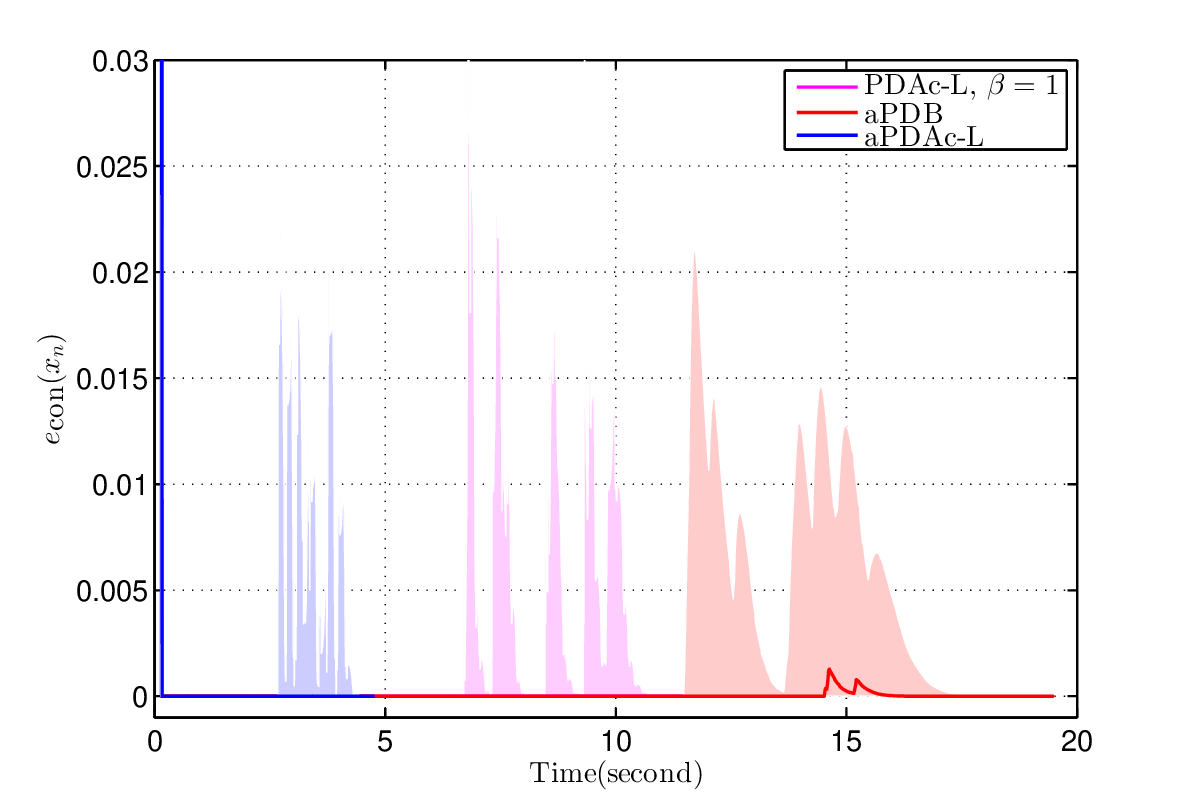}}
\caption{Comparison  results in terms of function value residual (left) and feasibility violations (right) on 10 random strongly convex QCQPs with $n=500$ and $m=10$. First row:  convergence as the iterations proceeded. Second row: convergence  as CPU time proceeded.
}
\label{Fig:QCQP-sc}
\end{figure}
\subsection{Experiments on SLR Problems}
The second problem under investigation is the SLR problem for binary classification. The training set is represented as $\{(a_i,b_i)\in \mathbb{R}^n \times \{\pm 1\}: i = 1,\ldots,m\}$, where $a_i$ denotes the feature vector for each sample and $b_i$ is the corresponding binary label. The SLR problem  is formulated as
\be\label{SLR}
F_{\text{opt}} := \min\limits_{x\in\bR^n}~\Big\{F(x):= t \|x\|_1 + \frac{1}{m}\sum\nolimits_{i=1}^m \log (1 + \exp(-b_ia^\top_ix)) \Big\},
\ee
where $t>0$ is a regularization parameter. By letting $h(x) := \frac{1}{m}\sum\nolimits_{i=1}^m \log (1 + \exp(-b_ia^\top_ix))$ and $g(x) := t\|x\|_1$, \eqref{SLR} reduces to \eqref{composite_opt} and thus can be solved by Algorithm \ref{algo2}.
In our experiments, we set $t=0.005\|A^\top b\|_{\infty}$ as in \cite{Malitsky2019Golden} with $A^\top =[a_1, a_2, \ldots  a_m]$ and $b=(b_1, b_2,\ldots,b_m)^\top$.
We next compare the proposed Algorithm \ref{algo2} (aPGMc) with aGRAAL proposed in \cite{Malitsky2019Golden}.
We used $\psi=3/2$ and $\varphi=10/9$ for aGRAAL, and $\psi=2$ (which is larger than that for aGRAAL) and either $\varphi=10/9$ or $\varphi=6/5$ for aPGMc.
In our experiments, we took two popular datasets from LIBSVM\footnote{Website: \url{https://www.csie.ntu.edu.tw/~cjlin/libsvmtools/datasets/}}: \emph{a9a} with $(m,n)=(32561,123)$ and \emph{rcv1} with $(m,n)=(20242,47236)$.

Following \cite{Malitsky2019Golden}, we executed aPGMc and aGRAAL for a sufficient number of iterations until the condition $\|x_n-\prox_{\tau_n g}(x_n - \tau_n\nabla f(x_n))\|\leq 10^{-6}$ was satisfied. Next, we defined $F_\text{opt}$ as the smallest function value $F(x_n)$ observed during the execution.
Figure~\ref{Fig:SLR} shows a comparison of the decreasing behavior of $F(x_n)-F_\text{opt}$ with respect to CPU time. The results indicate that aPGMc with either $\varphi=10/9$ or $\varphi=6/5$ converges faster than aGRAAL for both datasets. One plausible reason for this faster convergence is that aPGMc uses $\psi=2$, which is larger than the value of $\psi=3/2$ used in aGRAAL, and increasing $\psi$ brings $z_n$ closer to the current iterate $x_{n-1}$.
For aPGMC, on the other hand, the use of the larger value $\varphi=6/5$ results in slightly faster convergence compared to $\varphi=10/9$. This may be because increasing $\varphi$ leads to a possibly faster increase in the stepsize $\tau_{n-1}$, as defined in Algorithm \ref{algo2}.

\begin{figure}[htpb]
\centering
\subfigure[a9a]{
\includegraphics[width=0.46\textwidth]{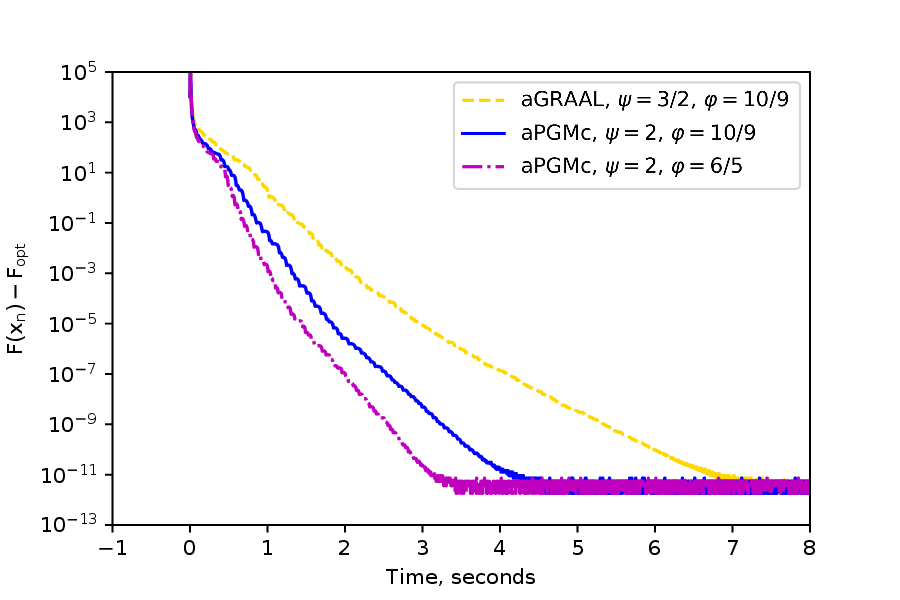}}
\subfigure[rcv1]{
\includegraphics[width=0.46\textwidth]{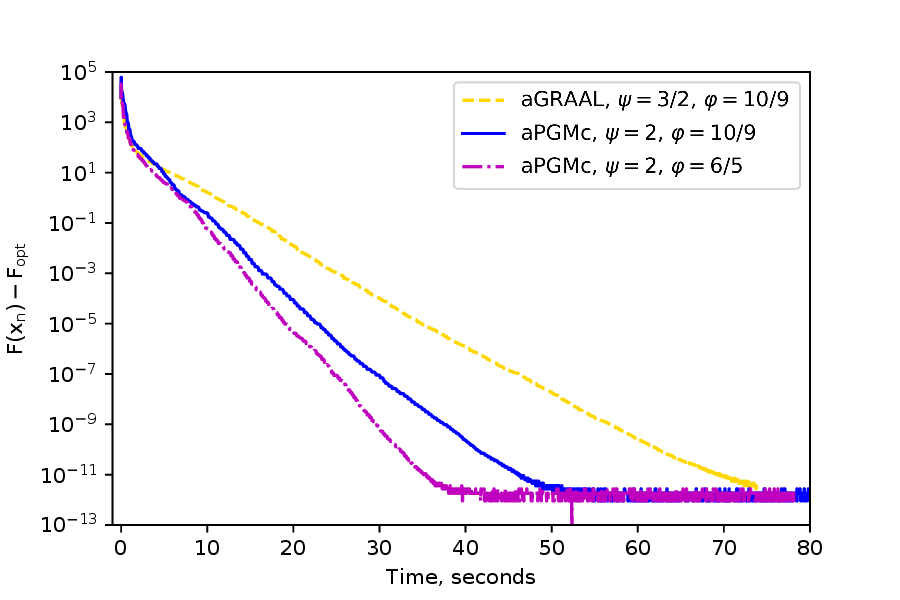}}
\caption{Comparisons results of aGRAAL and aPGMc on SLR problem.
The decreasing behavior of $F(x_n)-F_\text{opt}$ with respect to CPU time on problem \eqref{SLR} are presented.
Letf: dataset a9a. Right: dataset rcv1.}
\label{Fig:SLR}
\end{figure}

\section{Conclusion}
\label{sec-conclusion}
This paper proposes a  convex combination based primal-dual algorithm  with adaptive linesearch for solving structured convex-concave saddle point problems with generic nonlinear coupling term. The proposed linesearch strategy has the advantage of updating only one variable per linesearch iteration, which is computationally cheaper for many practical problems. Global pointwise convergence and $\cO(1/N)$ ergodic convergence rates are also established. For optimization problems with nonlinear compositional structure, ergodic convergence rates are measured by both the function value residual and constraint violations. Furthermore, when one of the component functions is strongly convex, faster $O(1/N^2)$ ergodic convergence rate results, quantified by the same measures, are established by adaptively choosing some algorithmic parameters.
For composite convex optimization problems with finite sum structure,  the proposed algorithm
will reduce to a fully adaptive proximal gradient algorithm  with convex combination, which gives
explicit stepsize rules without requiring linesearch.
 The resulting algorithm allows for a wider choice of key parameters compared to the adaptive algorithm proposed by Malitsky \cite{Malitsky2019Golden}. Numerical experiments on QCQP and SLR problems demonstrate the benefits of the proposed algoirthms.
%incorporating adaptive stepsizes and more relaxed parameters.

\section*{Acknowledgements}
The authors would like to express gratitude towards the authors of \cite{EYNS2021,Malitsky2019Golden} for sharing their codes, which were used for fair comparisons in this study.

\begin{appendix}
\section{Proof of Lemma \ref{lem-beta-sc}}\label{proof:lem-beta-cnn}
\begin{proof}
(i) Recall that for problem \eqref{saddle-point-special} we have $\Phi(x,y) = h(x)+\langle H(x),y\rangle$.
Let $\theta_n(\la)$ and $\Phi_n^y(\la)$ be defined in \eqref{def:ytP} and  define $y_{n}(\la)  :=  \prox_{\beta_n \la f^*}(y_{n-1}+\beta_n\la H(x_n))$ for $\lambda > 0$.
Assume, by contradiction, that the linesearch procedure defined in Algorithm~\ref{algo1-sc} fails to terminate at the $n$th iteration.
Then for all  $i=0,1,2,\ldots$ and  $\la =\tau \mu^i$, where
$\tau = \min\{\varphi\tau_{n-1}, \tau_{\max}\}$ as defined in Step 2 of the algorithm, we have
\be\label{linesearch-cond1-v-sc}
\frac{\la \tau_{n-1}}{\xi}\|\theta_n(\la)\|^2  >
\frac{\kappa_{n-1}\beta_{n-1}}{\beta_n} \|x_{n}-x_{n-1}\|^2
+\frac{1}{\beta_n}\|y_n(\la)-y_{n-1}\|^2.
\ee
Similar to the proof of Lemma \ref{lem_bound}, for all $\la=\tau \mu^i$ with $i=0,1,2,\ldots$, we have $y_n(\la)\in B[y_{n-1}; r]$.
Then, by combining \eqref{ineq_la}, \eqref{linesearch-cond1-v-sc} and  $\la=\tau \mu^i$, we obtain
\ben
 {2 \tau\mu^i \tau_{n-1}\over \xi} \Big( L_{xx}^2 \|x_{n}-x_{n-1}\|^2 +
 L_{xy}^2 \|y_{n}(\la)-y_{n-1}\|^2\Big) > \frac{\beta_{n-1}\kappa_{n-1}}{\beta_n} \|x_{n}-x_{n-1}\|^2
+\frac{1}{\beta_n}\|y_n(\la)-y_{n-1}\|^2,
\een
which implies $2 \tau \mu^i \tau_{n-1}L_{xx}^2 / \xi >  \beta_{n-1}\kappa_{n-1} / \beta_n$
or $ 2\tau \mu^i \tau_{n-1}L_{xy}^2/\xi > 1/\beta_n$ for all $i\geq 0$.
This is impossible since $\mu^i\rightarrow 0$ as $i\rightarrow \infty$,  which indicates that the linesearch procedure must terminate.

(ii)  We can observe that $\widetilde{J}(x_n,w_n,y^\star)$, $A_n(y^\star)$, and $B_n$ are all nonnegative and that $\beta_n\geq \beta_{n-1}$ for all $n\geq1$, as defined in \eqref{beta-gstrong}. Therefore, using \eqref{key-ineq2}, we can show that $\beta_1A_n(y^\star)\leq \beta_nA_n(y^\star) \leq \beta_{n-1}A_{n-1}(y^\star) \leq \cdots \leq \beta_1A_1(y^\star)$.
Hence, using the definition of $A_n(y^\star)$ in \eqref{def:AnyBn}, we can derive
\ben
\frac{\psi}{2(\psi-1)}\|z_{n+1}-{x^{\star}}\|^2\leq A_n(y^\star) \leq A_1(y^\star) \text{~~and~~}
\frac{1}{2}\|y_{n-1}-y^{\star}\|^2\leq \beta_nA_n(y^\star) \leq \beta_1A_1(y^\star),
\een
which implies that  $\{z_n: n\geq 1\}$ and $\{y_n: n\geq 1\}$ are bounded. Since $x_n=(\psi z_{n+1}-z_{n})/(\psi-1)$, which follows from the first relation in \eqref{x_updating}, the sequence $\{(x_n,y_n,z_n): n\geq 1\}$ is bounded.

(iii) Let $h(\tau) := 1+ \frac{(\psi-\varphi)\gamma\tau}{\psi+\varphi\gamma\tau}$, which is an increasing function of $\tau>0$. Since $\tau_{n}\leq\tau_{\max}$ for all $n\geq 0$, it follows that
%\ben
$\varsigma := h(\tau_{\max}) \geq h(\tau_n) =1+\gamma \rho_{n+1}\tau_{n} \geq 1$.  From part (ii) of this lemma, the sequence $\{(x_n,y_n): n\geq 1\}$ is bounded. Further considering the left-hand-side inequality in \eqref{result1}, we observe that the linesearch condition \eqref{linesearch-cond-sc} is satisfied provided that
\be\label{condition-line-sc}
\beta_{n-1}\kappa_{n-1}/\beta_n- 2\tau_{n}\tau_{n-1}L_{xx}^2/\xi\geq 0
~~\mbox{and}~~
1/\beta_n -  2\tau_n\tau_{n-1}L_{xy}^2 / \xi \geq 0.
\ee
Since $\kappa_{n-1}\geq \omega \delta_{n-1}  = \omega\tau_{n-1}/\tau_{n-2}$, $\beta_n = (1+ \gamma\rho_n \tau_{n-1})\beta_{n-1} \leq \varsigma\beta_{n-1}$ and $\tau_n \leq \tau_{\max}$ for all $n$,
we have
\be\label{jy-7}
\frac{\kappa_{n-1}\beta_{n-1}}{\tau_{n-1}\sqrt{\beta_n}}\geq \frac{\omega \beta_{n-1}}{\tau_{n-2}\sqrt{\beta_n}}\geq \frac{\omega \sqrt{\beta_{0}}}{\tau_{\max}\sqrt{\varsigma}}.
\ee
For convenience, we define
$\Upsilon_1 := \frac{\xi }{2L_{xx}^2} \frac{\omega \sqrt{\beta_{0}}}{\tau_{\max}\sqrt{\varsigma}}$ and $\Upsilon_2:={ \xi\over 2 L_{xy}^2}$.
It follows from \eqref{jy-7} that
the left-hand-side inequality in \eqref{condition-line-sc} holds when
$\sqrt{\beta_n}\tau_n\leq \Upsilon_1$, while the right-hand-side inequality is equivalent to
 $\beta_n\tau_n\tau_{n-1} \leq \Upsilon_2$.

Then, for any fixed $n\geq 1$, setting $\tau_n=\underline{\tau}:=\min\big( \Upsilon_1/\sqrt{\beta_n}, \Upsilon_2/(\beta_n\tau_{n-1})\big)$ will make the two inequalities in \eqref{condition-line-sc} being satisfied, and so is the linesearch condition \eqref{linesearch-cond-sc}.
Since $\tau_n = \tau\mu^i$, where $i$ is the smallest nonnegative integer such that \eqref{linesearch-cond-sc} is satisfied, there must hold $\tau_n>\mu \underline{\tau}$, which implies that either $\tau_n > {\mu \Upsilon_1 / \sqrt{\beta_n} }$ or $\tau_n > {\mu \Upsilon_2 / (\beta_n\tau_{n-1}) }$ for all $n\geq 1$.
In the former case, we have $\sqrt{\beta_n}\tau_n > \mu \Upsilon_1$.
In the later case, with $n$ replaced by $n+1$, we have $\beta_{n+1}\tau_{n}\tau_{n+1} > {\mu \Upsilon_2  }$. Then, by
$\tau_{n+1} \leq \varphi \tau_{n} $, we derive
%\ben
$\beta_{n+1}\tau_{n}^2\geq\beta_{n+1}\tau_{n}\tau_{n+1}/\varphi> \mu\Upsilon_2/\varphi$,
and thus
\ben
\beta_{n}\tau_{n}^2
\stackrel{(\ref{beta-gstrong})}= {\beta_{n+1} \tau_{n}^2 \over   1+\gamma \rho_{n+1}\tau_{n} }
>\frac{\mu \Upsilon_2/\varphi}{1+\gamma \rho_{n+1}\tau_{n}}\geq \frac{\mu\Upsilon_2}{\varphi\varsigma } > 0.
\een
Hence,  we always have $ \sqrt{\beta_n} \tau_n > c_1:=  \min(\mu \Upsilon_1, \sqrt{\mu\Upsilon_2/(\varphi\varsigma)})> 0$.
Since $\tau_n\leq \tau_{\max}$, we have
\be\label{relation-beta}
\beta_{n+1} \stackrel{(\ref{beta-gstrong})}
= \beta_{n}\Big(1+ \frac{\gamma(\psi-\varphi)\tau_n}{\psi+\varphi\gamma\tau_{n}}\Big)
 \geq \beta_{n}+ \frac{ \gamma (\psi-\varphi)\sqrt{\beta_n}\tau_n}{\psi+\varphi\gamma \tau_{\max}} \sqrt{\beta_n} \geq \beta_{n}+
  \varrho \sqrt{\beta_{n}},
\ee
%where the first inequality is due to , $\Xi := > 0$, and the second inequality follows from $\tau_n \sqrt{\beta_n}> c_1$.
where $\varrho:= \frac{\gamma (\psi-\varphi)c_1}{\psi+\varphi\gamma \tau_{\max}}>0$.
From \eqref{relation-beta}, it is easy to show by induction that $\beta_n\geq c_2 n^2$ for all $n\geq1$ with $c_2:= \min(\varrho^2/9, \beta_1) > 0$. This completes the proof.

\end{proof}

\end{appendix}
\bibliographystyle{abbrv}
\bibliography{files/cxktex}

\begin{thebibliography}{10}

\bibitem{Arjovsky2017nets}
M.~Arjovsky, S.~Chintala, and L.~Bottou.
\newblock {Wasserstein generative adversarial networks}.
\newblock {In Proceedings of the 34th International Conference on Machine
  Learning}, pages {214--22}, {2017}.

\bibitem{Beck2017book}
A.~Beck.
\newblock {\em First-Order Methods in Optimization}.
\newblock MOS-SIAM Series on Optimization. SIAM-Society for Industrial and
  Applied Mathematics, 2017.

\bibitem{Bertsekas1982Projection}
D.~P. Bertsekas and E.~M. Gafni.
\newblock Projection methods for variational inequalities with application to
  the traffic assignment problem.
\newblock {\em Mathematical Programming Study}, 17:139--159, 1982.

\bibitem{Bouwmans2016Handbook}
T.~Bouwmans, N.~S. Aybat, and E.~H. Zahzah.
\newblock {\em Handbook of ``Robust low-rank and sparse matrix decomposition:
  applications in image and video processing"}, volume~45.
\newblock 2016.

\bibitem{Chambolle2011A}
A.~Chambolle and T.~Pock.
\newblock A first-order primal-dual algorithm for convex problems with
  applications to imaging.
\newblock {\em Journal of Mathematical Imaging and Vision}, 40(1):120--145,
  2011.

\bibitem{Chambolle2016ergodic}
A.~Chambolle and T.~Pock.
\newblock {On the ergodic convergence rates of a first-order primal-dual
  algorithm}.
\newblock {\em {Mathematical Programming}}, {159}({1--2}):{253--287}, {SEP}
  {2016}.

\bibitem{ChY2022relaxed}
X.~Chang and J.~Yang.
\newblock Grpda revisited: relaxed condition and connection to chambolle-pock's
  primal-dual algorithm.
\newblock {\em Journal of Scientific Computing}, 2022.

\bibitem{ChYZ2022GRPDAL}
X.~Chang, J.~Yang, and H.~Zhang.
\newblock Golden ratio primal-dual algorithm with linesearch.
\newblock {\em SIAM J. Optim.}, 32(3):1584--1613, 2022.

\bibitem{ChY2020Golden}
X.~Chang and J.~F. Yang.
\newblock A golden ratio primal-dual algorithm for structured convex
  optimization.
\newblock {\em Journal of Scientific Computing}, 87, 2021.

\bibitem{App2018GANs}
V.~S. Constantinos~Daskalakis, Andrew~Ilyas and H.~Zeng.
\newblock {Training GANs with Optimism}.

\bibitem{Esser2010General}
E.~Esser, X.~Zhang, and T.~F. Chan.
\newblock {A general framework for a class of first order primal-dual
  algorithms for convex optimization in imaging science}.
\newblock {\em {SIAM Journal on Imaging Sciences}}, {3}({4}):{1015--1046},
  {2010}.

\bibitem{Goodfellow2014nets}
I.~Goodfellow, J.~Pouget-Abadie, M.~Mirza, B.~Xu, D.~Warde-Farley, S.~Ozair,
  A.~Courville, and Y.~Bengio.
\newblock {Generative adversarial nets}.
\newblock {In Advances in neural information processing systems}, pages
  {2672--2680}, {2014}.

\bibitem{GLL}
L.~Grippo, F.~Lampariello, and S.~Lucidi.
\newblock A nonmonotone line search technique for {N}ewton's method.
\newblock {\em SIAM J. Numer. Anal.}, 23(4):707--716, 1986.

\bibitem{EYNS2021}
E.~Y. Hamedani and N.~S. Aybat.
\newblock A primal-dual algorithm with line search for general convex-concave
  saddle point problems.
\newblock {\em SIAM Journal on Optimization}, 31(2):1299--1329, 2021.

\bibitem{Hayden2013A}
S.~Hayden and O.~Stanley.
\newblock A low patch-rank interpretation of texture.
\newblock {\em SIAM Journal on Imaging Sciences}, 6(1):226--262, 2013.

\bibitem{He22On}
B.~He, S.~Xu, and X.~Yuan.
\newblock On convergence of the arrow-hurwicz method for saddle point problems.
\newblock {\em Journal of Mathematical Imaging and Vision}, 64:662--671, 2022.

\bibitem{He2014On}
B.~He, Y.~You, and X.~Yuan.
\newblock {On the convergence of primal-dual hybrid gradient algorithm}.
\newblock {\em {SIAM Journal on Imaging Sciences}}, {7}({4}):{2526--2537},
  {2014}.

\bibitem{Korpelevich1976}
G.~Korpelevich.
\newblock The extragradient method for finding saddle points and other
  problems.
\newblock {\em Matecon}, 12(1):747--756, 1976.

\bibitem{Malitsky2019Golden}
Y.~Malitsky.
\newblock {Golden ratio algorithms for variational inequalities}.
\newblock {\em {Mathematical Programming}}, {184}:{383--410}, {2020}.

\bibitem{Malitsky2018A}
Y.~Malitsky and T.~Pock.
\newblock {A first-order primal-dual algorithm with linesearch}.
\newblock {\em {SIAM Journal on Optimization}}, {28}({1}):{411--432}, {2018}.

\bibitem{MOS2020}
A.~Mokhtari, A.~E. Ozdaglar, and S.~Pattathil.
\newblock Convergence rate of $\mathcal{O}(1/k)$ for optimistic gradient and
  extragradient methods in smooth convex-concave saddle point problems.
\newblock {\em SIAM Journal on Optimization}, 30(4):3230--3251, 2020.

\bibitem{Nedic2009Subgradient}
A.~Nedic and A.~Ozdaglar.
\newblock {Subgradient methods for saddle-point problems}.
\newblock {\em {Journal of Optimization Theory \& Applications}},
  {142}({1}):{205--228}, {JUL} {2009}.

\bibitem{Nemirovski04siam}
A.~Nemirovski.
\newblock Prox-method with rate of convergence o(1/t) for variational
  inequalities with lipschitz continuous monotone operators and smooth
  convex-concave saddle point problems.
\newblock {\em SIAM Journal on Optimization}, 15(1):229--251, 2004.

\bibitem{Nesterov2013gradient}
Y.~Nesterov.
\newblock {Gradient methods for minimizing composite functions}.
\newblock {\em {Mathematical Programming}}, {140}({1}):{125--161}, {2013}.

\bibitem{Popov1980}
L.~D. Popov.
\newblock A modification of the arrow-hurwicz method for search of saddle
  points.
\newblock {\em Mathematical Notes of the Academy of Sciences of the USSR 28.5},
  pages 845--848, 1980.

\bibitem{SabT22SIOPT}
S.~Sabach and M.~Teboulle.
\newblock Faster lagrangian-based methods in convex optimization.
\newblock {\em SIAM Journal on Optimization}, 32(1):204--227, 2022.

\bibitem{Teboulle2014Rate}
R.~Shefi and M.~Teboulle.
\newblock Rate of convergence analysis of decomposition methods based on the
  proximal method of multipliers for convex minimization.
\newblock {\em SIAM Journal on Optimization}, 24(1):269--297, 2014.

\bibitem{Sun2014A}
D.~Sun, K.-C. Toh, and L.~Yang.
\newblock A convergent 3-block semi-proximal alternating direction method of
  multipliers for conic programming with $4$-type of constraints.
\newblock {\em SIAM Journal on Optimization}, 25(2):882--915, 2015.

\bibitem{Tseng1995}
P.~Tseng.
\newblock On linear convergence of iterative methods for the variational
  inequality problem.
\newblock {\em Journal of Computational and Applied Mathematics},
  60(1--2):237--252, 1995.

\bibitem{Tseng2000A}
P.~Tseng.
\newblock A modified forward-backward splitting method for maximal monotone
  mappings.
\newblock {\em SIAM Journal on Control and Optimization}, 38(2):431--446, 2000.

\bibitem{Uzawa58}
H.~Uzawa.
\newblock {\em Iterative methods for concave programming}.
\newblock Studies in Linear and Nonlinear Programming (K. J. Arrow, L. Hurwicz
  and H. Uzawa, eds). Stanford University Press, Stanford, CA, 1958.

\bibitem{Yang2011Alternating}
J.~Yang and Y.~Zhang.
\newblock Alternating direction algorithms for $\ell_1$-problems in compressive
  sensing.
\newblock {\em SIAM Journal on Scientific Computing}, 33(1):250--278, 2011.

\bibitem{HZ}
H.~Zhang and W.~W. Hager.
\newblock A nonmonotone line search technique and its application to
  unconstrained optimization.
\newblock {\em SIAM J. Optim.}, 14(4):1043--1056, 2004.

\bibitem{Zhu23On}
Y.~Zhu, D.~Liu, and Q.~Tran-Dinh.
\newblock New primal-dual algorithms for a class of nonsmooth and nonlinear
  convex-concave minimax problems.
\newblock {\em SIAM Journal on Optimization}, 32(4):2580--2611, 2022.

\bibitem{Unified2023PDA}
Z.~Zhu, F.~Chen, J.~Zhang, and Z.~Wen.
\newblock A unified primal-dual algorithm framework for inequality constrained
  problems.
\newblock {\em Journal of Scientifc Computting}, 97(2), NOV 2023.

\end{thebibliography}

\end{document}